\documentclass{article}
\usepackage[utf8]{inputenc}
\usepackage[T1]{fontenc}
\usepackage{graphicx} % Required for inserting images
\usepackage{amsmath}
\usepackage{amsthm}
\usepackage{amssymb}
\usepackage{mathtools}
\usepackage{enumitem}
\usepackage{float}
\usepackage{subfig}
\usepackage{cancel}
\usepackage{romannum}
\usepackage[
    top=0.75in,
    bottom=0.75in,
    left=0.9in,
    right=0.9in
]{geometry}
\usepackage{setspace} 
\usepackage{empheq}
\usepackage{titlesec, titletoc} 
\usepackage[numbers,sort]{natbib}
\usepackage{url} 
\usepackage[colorlinks,  citecolor=green, linkcolor = black, urlcolor = cyan]{hyperref}

\theoremstyle{plain}
\newtheorem{theorem}{Theorem}

\newtheorem{lemma}[theorem]{Lemma}

\theoremstyle{definition}

\theoremstyle{remark}
\newtheorem{remark}[theorem]{Remark}

\newcommand{\R}{\mathbb{R}}

\newcommand{\bfxi}{\boldsymbol{\xi}}

\newcommand{\bfeta}{\boldsymbol{\eta}}
\newcommand{\eps}{\varepsilon}

\DeclareMathOperator*{\argmin}{arg\,min}

\title{A Phase-Field Method for Curvature Flow of Networks\\ with Triple Junction Drag}
\author{Yuchuan Yang\thanks{{Department of Mathematics, University of Michigan. email: \texttt{yuchuan@umich.edu}}} \and Selim Esedo\={g}lu\thanks{Department of Mathematics, University of Michigan. email: \texttt{esedoglu@umich.edu}}}
\date{July 2026}

\begin{document}

\pagenumbering{arabic}

\maketitle

\begin{abstract}
We develop a new phase-field (diffuse interface) approximation of multiphase curvature motion with triple junction drag -- an important sharp interface model for the evolution of microstructure in polycrystalline materials during heat treatment.
This sharp interface model arises as gradient flow for the total length of the interfacial network with respect to a certain metric.
Accordingly, we derive our diffuse interface approximation -- a coupled system of partial differential equations  that is a variant of the Allen-Cahn system -- from a variational perspective, in the style of minimizing movements: starting from a discrete in time approximation that entails a convex optimization problem to advance from one time step to the next.
In the process, we propose a simple integral expression that counts the number of junctions using the order parameter that appears to be new even for the standard multiphase Allen-Cahn system.
The convergence of the resulting flow to the desired sharp interface limit is then verified via the method of matched asymptotic expansions.
Numerical convergence studies against both known exact solutions as well as highly accurate benchmark solutions obtained via front tracking provide clear further evidence for this convergence.
Moreover, the method retains the most desirable feature of diffuse interface methods: Automatic handling of topological changes in the network of interfaces.
\end{abstract}

\section{Introduction}
\label{sec:intro}
Multiphase mean curvature motion of interfacial networks plays a central role in many applications.
For example, in materials science, it gets used since the work of Mullins \cite{mullins} in 1950's to model the evolution of grain boundaries in polycrystalline materials during annealing.
One of the challenging aspects of this dynamics is the presence of free boundaries known as junctions along which three or more interfaces meet.
The motion of junctions is part of the unknown in the problem; it is uniquely determined by the (mean) curvature motion of the interfaces away from the junctions, and free boundary conditions that are imposed along the junctions.
The most well known boundary condition along junctions is that of {\em Herring} \cite{herring} which, in physical terms, stipulates that surface tension forces applied by the interfaces on a triple junction must vanish at all times.
The force balance can equivalently be expressed in terms of the angles formed between the interfaces intersecting along a given triple junction, resulting in what is known as the {\em Herring angle condition}.
During the evolution, interfaces and junctions may collide, pinch off, or altogether disappear, and new junctions may be nucleated.
To deal with these inevitable topological events, implicit interface methods such as phase-field (see e.g. \cite{lqchen, lqchen2, fan}), level-set (see e.g. \cite{osher_sethian,elsey1,bernacki,elsey2,saye_sethian_2011}), and threshold dynamics (see e.g. \cite{mbo92,esedoglu_otto,elsey3,peng_2022,rohrer_annual}) have been developed and employed.
There have been numerous large scale simulations of microstructural evolution reported with this model and its variants in the literature, with varying degrees of success in reproducing certain statistics, such as the grain size distribution or the misorientation distribution function, that are of interest to materials scientists.

Recently, detailed non-destructive experimental measurements of grain boundary motion made possible by high energy diffraction microscopy and diffraction contrast tomography have spurred intense activity in comparing simulation results against individual grain boundaries in 4D data \cite{mckenna,peng_2022,rohrer_annual}.
Significant differences have been noted, leading to speculation that Mullins' traditional model may be inadequate or require substantial modification.
One possible modification that has been proposed is {\em triple junction drag} \cite{GOTTSTEIN2002703,barmak,epshteyn,yangesedoglu25}: The traditional Herring angle condition is replaced by a relaxed version which, in physical terms, no longer requires instantaneous force balance.
Instead, triple junctions move with finite speed to restore the force balance.

There is a shortage of numerical methods in the literature suitable for large scale simulations of grain boundary motion with triple junction drag that can seamlessly handle topological changes in two or three dimensions.
The major exception is \cite{johnsonvoorhees} that proposes a phase-field method; see also the very recent follow-ups \cite{miyoshi2025,miyoshi2026} that carry out large scale simulations using closely related methods.
The primary goal of \cite{johnsonvoorhees,miyoshi2025,miyoshi2026} appear to be capturing some {\em qualitative} implications of junction drag, without claiming convergence to the sharp interface description of the evolution (no careful convergence study is offered).
That level of qualitative approximation may well suffice to explore how junction drag influences statistical measures,  such as size distribution, of grain networks.
However, as we demonstrate in Section 3, the approach of \cite{johnsonvoorhees} does not always converge {\em quantitatively} to the solution of the desired system of partial differential equations (PDEs) that describe curvature motion with triple junction drag; see equation  (\ref{eq:pde}) in Section 2 below for a precise statement of this system.
Recent impetus in the materials science community towards comparing the evolution of {\em individual} interfaces in simulations against experimental measurements goes far beyond statistics, and calls for numerical methods that are guaranteed to converge to the intended PDE model in the quantitative, traditional numerical analysis sense.
The purpose of the present study is to address the lack of such methods in existing literature.

Our task in this work is thus to follow up on works such as \cite{johnsonvoorhees,miyoshi2025} by taking a further step towards developing phase-field methods that are equally robust (in handling topological changes) and practical (suitable for simulations with many grains in two or three dimensions), but also {\em convergent} to the well-known, precise sharp interface limit (\ref{eq:pde}) of curvature motion with triple junction drag.
To that end, we propose a novel phase-field method for triple junction drag, substantiate its consistency by formal matched asymptotic analysis, and investigate its convergence with careful numerical experiments against exact solutions as well as highly accurate benchmarks obtained via front tracking. We also demonstrate how topological changes are automatically handled, as is expected from any practical phase-field method.

\section{Sharp-interface formulation}
\label{sec:sharpinterface}
In this section, we give a precise description of {\em multiphase curvature motion with triple junction drag} that appears in models of microstructural evolution.
It constitutes the precise dynamics, described by a PDE system, that existing numerical methods such as \cite{johnsonvoorhees,miyoshi2025,miyoshi2026} as well as the new one proposed in this study aim to approximate.
As the situation in $\mathbb{R}^2$ is already interesting, we focus on planar networks in the remainder of the paper, but always with an eye towards seamless extension to $\mathbb{R}^3$ (a hallmark of phase-field methods), which will be taken up in subsequent work.

We fix $\Omega\subset\R^2$ as our domain and let $\Sigma_i(t)\subset\Omega$, 
$i=1,2,3$, be the (evolving) region occupied by the $i-$th grain. The sets $\Sigma_i$ form a partition of $\Omega$ so that $\Omega = \Sigma_1 \cup \Sigma_2 \cup \Sigma_3$ and $|\Sigma_i\cap\Sigma_j| = 0$ whenever $i\neq j$. 
(We write $|S|$ to denote the area of a set $S\subset\mathbb{R}^2$).
Let $\Gamma_i(t)$ denote the boundary between $\Sigma_{i-1}(t)$ and $\Sigma_i(t)$ (Figure \ref{fig:network}).
We assume that the curves $\Gamma_i(t)$ meet at a triple junction $p(t)$ for the entire duration of the evolution considered. \textit{Motion by curvature with triple junction drag} refers to the evolution of the curves $\Gamma_i(t)$ governed by the law:
\begin{align}\label{curvaturemotion}
    v_i = \sigma_i m_{GB}^{(i)} \kappa_i,  \qquad \text{along }\Gamma_i 
\end{align}
\begin{align}\label{drag}
    \frac{d}{dt}p(t) = m_{TJ} \sum_{i=1}^3 \sigma_i \tau_i, \qquad \text{at the triple junction}
\end{align}
Here, $\sigma_i$ is the surface tension coefficient associated with the interface $\Gamma_i$, $m_{GB}$ is the mobility of the grain boundary, $m_{TJ}$ is the mobility of the triple junction, $v_i$ is the normal velocity of $\Gamma_i$, $\kappa_i$ is the curvature of $\Gamma_i$ and $\tau_i$ is the unit tangent vector of $\Gamma_i$ at the triple junction $p(t)$, pointing away from it.

\begin{figure}[H]
    \centering
    \includegraphics[width=0.7\linewidth]{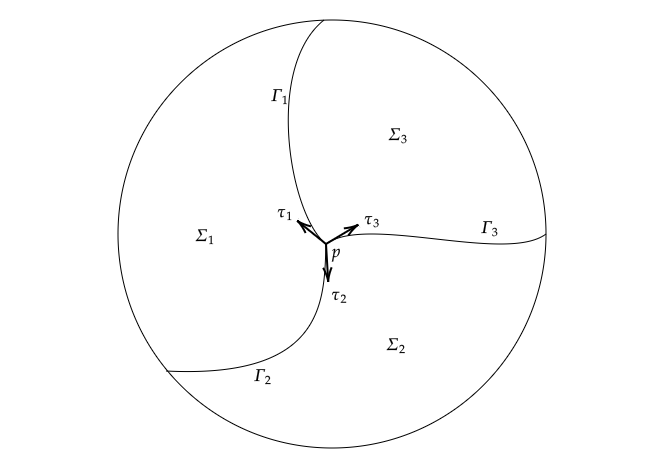}%
    \caption{A network with a single triple junction}
    \label{fig:network}%
\end{figure}

The system \eqref{curvaturemotion} \& \eqref{drag} has an underlying variational principle (see \cite{yangesedoglu25}), which will be the basis of our phase-field model. 
In words, \eqref{curvaturemotion} \& \eqref{drag} is a formal gradient flow of the total length functional
\begin{align}
    \mathcal{L}(\Gamma) \coloneqq \sum_{i=1}^3 \sigma_i \text{Length}(\Gamma_i)
\end{align}
with respect to a metric comprising
\begin{enumerate}
    \item the $L^2$ norm of the perturbation to the curves in the normal direction and
    \item the Euclidean norm of the perturbation to the triple junction's location.
\end{enumerate}
To be precise, let $\gamma_i(\cdot,t)$ be parametrizations of the curves $\Gamma_i(t)$ so that $\gamma_i(0,t) = p(t)$, the location of the triple junction.
Then \eqref{curvaturemotion} \& \eqref{drag} is formally gradient flow for the energy
\begin{equation}
\label{eq:parametricenergy}
\mathcal{L}(\Gamma) = \sum_{i=1}^3 \sigma_i \int_0^1 |\partial_x \gamma_i| \, dx
\end{equation}
with respect to the following inner product (modulo degeneracy in the tangential direction) defined on a pair of triplets of vector fields $\bfxi = \{\xi^i\}_{i=1}^3$ and $\bfeta := \{\eta^i\}_{i=1}^3$:
\begin{equation}
\label{eq:metric}
\langle \bfxi , \bfeta  \rangle \coloneqq \sum_{i=1}^3 \frac{1}{m_{GB}^{(i)}}\int_0^1 \big(\xi^i(x) \cdot \nu_i(x)\big)\big( \eta^i(x)\cdot \nu_i(x)\big) |\partial_x \gamma_i(x,t)| \; dx + \frac{1}{m_{TJ}} \xi^1 (0)\cdot \eta^1(0).
\end{equation}
where $\xi^i \, , \, \eta^i \, : \Gamma_i(t) \to \mathbb{R}^2$ obey $\xi^i(0) = \xi^j(0)$ and $\eta^i(0)=\eta^j(0)$ for any $i\not=j$, and $\nu_i(x)$ denotes the unit normal to $\Gamma_i(t)$ at $x\in\Gamma_i$.
We note that, formally, in the limit $m_{TJ} \to \infty$, the pointwise term in (\ref{eq:metric}) drops out and one recovers the standard variational formulation of curvature motion subject to Herring angle conditions.

The geometric evolution problem \eqref{curvaturemotion} \& \eqref{drag} can then be described by a system of PDEs (initial-boundary value problem) in terms of the parametrizations $\gamma_i(x,t)$ of the curves $\Gamma_i(t)$ as follows:
\begin{equation}\label{eq:pde}
    \begin{aligned}
        \gamma_{it}(x,t) &= \sigma_i m_{GB}^{(i)} \frac{\gamma_{ixx}(x,t)}{|\gamma_{ix}(x,t)|^2}, & (x,t)\in[0,1]\times[0,T], \qquad i\in\{1,2,3\}\\
        \gamma_i(0,t) &= \gamma_j(0,t), &t\in[0,T], \quad i,j\in\{1,2,3\}\\
        \gamma_{1t}(0,t) &= m_{TJ}\sum_{j=1}^3 \sigma_j \frac{\gamma_{jx}(0,t)}{|\gamma_{jx}(0,t)|}, &t\in[0,T] \\
        \gamma_i(1,t) &\equiv \gamma_i^0(1), &t\in[0,T]\qquad i\in\{1,2,3\}\\
        \gamma_i(x,0) &= \gamma_i^0(x), &x\in[0,1]\qquad i\in\{1,2,3\}.
    \end{aligned}
\end{equation}
 Short time well-posedness of \eqref{eq:pde} was established in \cite{yangesedoglu25}, verifying that the triple junction drag condition \eqref{drag} is a valid boundary condition that uniquely determines the evolution of the network. One may immediately observe that while \eqref{curvaturemotion} only specifies the normal velocity of the curves, \eqref{eq:pde} further specifies the tangential velocity of the curves as well (this follows \cite{bronsardreitich}).
 It is easy to check that any solution of \eqref{eq:pde} is a solution of \eqref{curvaturemotion} \& \eqref{drag}. Conversely, it is proven in \cite{yangesedoglu25} that given any three sufficiently regular curves satisfying \eqref{curvaturemotion} \& \eqref{drag}, one can always find parametrizations for them that satisfy \eqref{eq:pde}. We note that there is an extensive literature on curvature motion of networks (e.g. \cite{bronsardreitich,mantegazzanovagatortorelli,pluda,irregularnetworks}) formulated as such, in terms of parametrized curves.

Before the formation of any singularities or topological changes in the network, such descriptions of the network evolution in terms of explicitly parametrized curves, beyond being of fundamental theoretical interest, are also physically relevant, and naturally suggest {\em front tracking} as a highly efficient and accurate numerical method.
However, in simulations of relevance to materials scientists, who are often keenly interested in the coarsening dynamics of the network as many grains shrink and disappear, topological changes are inevitable.
While the treatment of topological changes in a front tracking implementation might be feasible in two dimensions under very special circumstances (e.g. equal surface tensions and mobilities, which in the Herring angle case allows a classification of possible transitions in the network into a small number of cases, see e.g. \cite{mantegazzanovagatortorelli, irregularnetworks}), computational exploration of these models for a broad range of parameters, and certainly in three dimensions, strongly suggests implicit interface methods as the preferred numerical method. 
We believe this to be even more so in the presence of triple junction drag, which makes additional topological changes possible even under the simplest of conditions (i.e. two dimensions, equal surface tensions and mobilities); see \cite{yangesedoglu25}.

In this paper, we choose to focus on the phase-field approach as an implicit interface method for network dynamics \eqref{curvaturemotion} \& \eqref{drag} that incorporates junction drag, as there is already existing work in this direction \cite{johnsonvoorhees,miyoshi2025} that we build up on, and phase-field more broadly has a long track record in materials literature; other approaches e.g. threshold dynamics \cite{mbo92,esedoglu_otto} are an enticing possibility left for future investigation.
Since it is already well understood (as part of a very long history) how phase-field methods can be generalized to accommodate different surface tensions and mobilities, we focus on junction drag as the main novelty and thus take $\sigma_i = m_{GB}^{(i)} = 1$ for all $i$ for the rest of the paper.

\section{Previous work}
\label{sec:previouswork}

In \cite{johnsonvoorhees}, the authors propose a phase-field method for triple junction drag which in the case of three grains can be described by the order parameters $\eta_i:\Omega \to \R^3$, $i=1,2,3$ that solve the following system of PDEs:
\begin{align}\label{voorheespde}
    \frac{\partial \eta}{\partial t} = L(\eta)(\Delta \eta - \nabla f_0(\eta)),
\end{align}
where $\eta = (\eta_1,\eta_2,\eta_3)$ and
\begin{align}
    f_0(\eta) = \mu\bigg(\sum_{i=1}^3 \big( -\frac{\alpha}{2}\eta_i^2 + \frac{\beta}{4}\eta_i^4\big) + \gamma\sum_{i=1}^3 \sum_{j\neq i}\eta_i^2 \eta_j^2  \bigg).
\end{align}
The effect of triple junction drag is introduced via the mobility factor 
\begin{align}
\label{eq:voorheesmobility}
    L(\eta) \coloneqq L_{GB} - \exp\bigg(-\Phi_{width}\big(\eta_1+\eta_2+\eta_3 - \Phi_{min}\big)^2\bigg)(L_{GB} - L_{TJ}).
\end{align}
where the parameters $L_{GB}$ and $L_{TJ}$ are to be chosen to induce the grain boundary and triple junction mobilities desired.
The expression \eqref{eq:voorheesmobility} interpolates between these two types of mobilities, based on location, 
the sum of the order parameters, i.e. $\eta_1 + \eta_2 + \eta_3$, serving as an indicator (detector) of triple junction locations.
Indeed, the authors observe that
\begin{align*}
    \eta_1(x) + \eta_2(x) + \eta_3(x) \approx
    \begin{cases}
        1.0, &\text{ inside grains }\\
        1.0, &\text{ along grain boundaries }\\
        1.08, &\text{ at triple junctions}
    \end{cases}
\end{align*}
for the particular potential $f_0$ they have chosen, allowing this expression to distinguish between junction locations and elsewhere; see Figure \ref{fig:voorhees_mobility}.

To discover the relation between the parameters $L_{GB}$ and $L_{TJ}$ and the mobilities $m_{GB}$ and $m_{TJ}$ that they induce, the authors adopt a data driven approach.
Relying on a class of exact, traveling wave solutions (an extension of the well-known "grim-reaper" solutions in the Herring angle case) obtained in \cite{GOTTSTEIN2002703} for the sharp interface PDE system \eqref{curvaturemotion} \& \eqref{drag}, the authors are able to determine this correspondence by matching simulation results using (\ref{voorheespde}) \& (\ref{eq:voorheesmobility}) to the exact solutions.

\begin{figure}[H]
    \centering    \includegraphics[width=.6\linewidth]{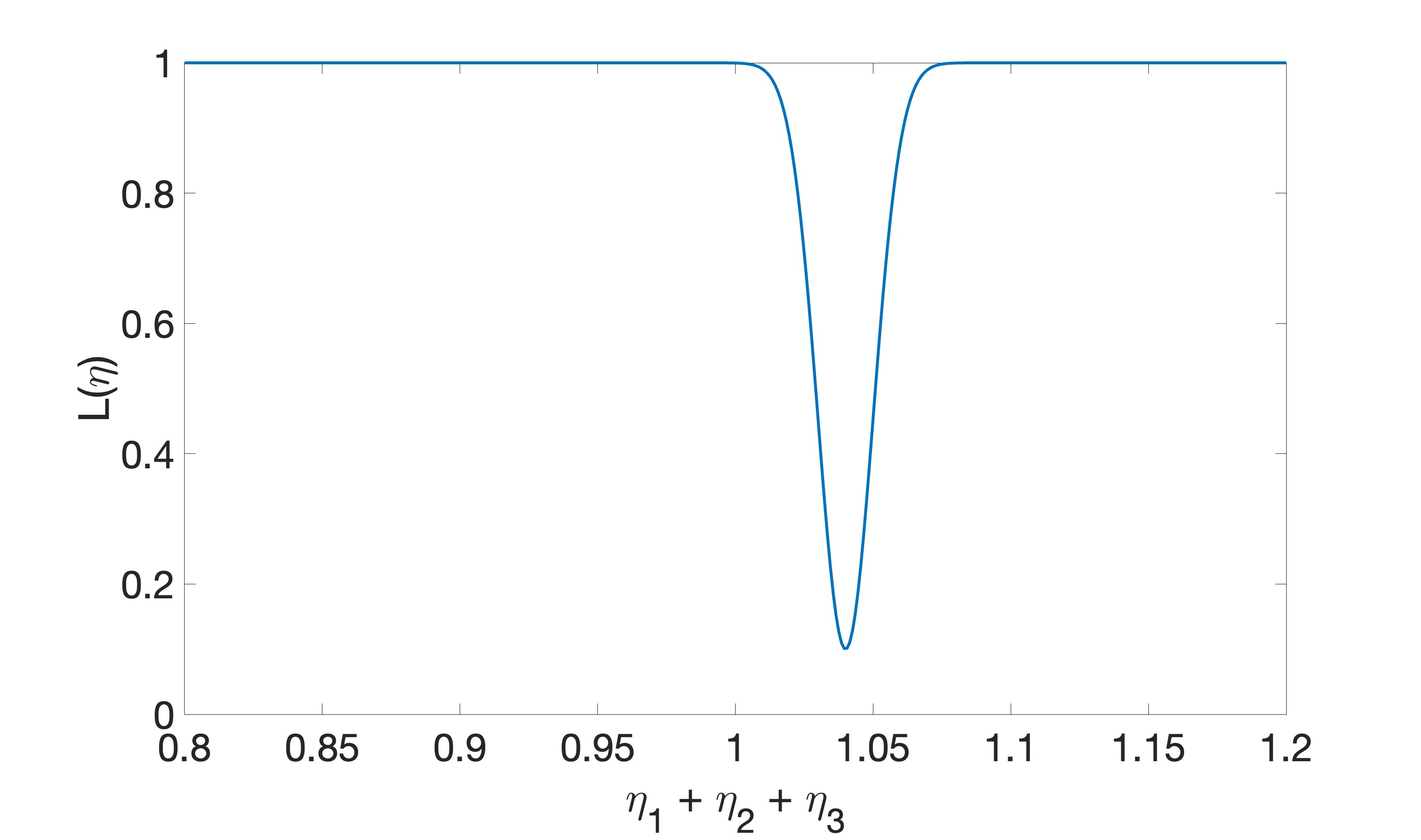}%
        \caption{Triple junction detector (indicator) used in \cite{johnsonvoorhees}.}
        \label{fig:voorhees_mobility}
\end{figure}

\noindent Approximating (\ref{curvaturemotion}) \& (\ref{drag}) with this approach requires making several important assumptions:
\begin{enumerate}
\item The effective mobility induced on the triple junction by factor (\ref{eq:voorheesmobility}) is assumed to be {\em independent} of the the configuration of angles formed at, i.e. the local geometry of, the triple junction.
Note that under the evolution (\ref{curvaturemotion}) \& (\ref{drag}), angles formed at the junction are generically time dependent (they cannot be prescribed or assumed constant).
\item The effective mobility induced by (\ref{eq:voorheesmobility}) is also assumed to be independent of the {\em direction of motion} of the triple junction.
For example, in order to converge to the desired evolution (\ref{curvaturemotion}) \& (\ref{drag}), the factor (\ref{eq:voorheesmobility}) needs to retard the motion of the junction approximately equally in every direction, regardless of whether that direction is aligned or not with the unit normal to one of the interfaces meeting at the junction.
\item The sum of order parameters $\eta_1(x) + \eta_2(x) + \eta_3(x)$ used as a "junction detector" is assumed to be independent of the local geometry of the triple junction.
\item The mobility factor (\ref{eq:voorheesmobility}) is assumed not to require any explicit dependence on the interface thickness.
\end{enumerate}

Our numerical experiments indicate these assumptions may not be valid.
In particular, we repeated the experiment in \cite{johnsonvoorhees} with the same choice of parameters ($L_{TJ} = 0.1$, $L_{GB}=1$, $\Phi_{width}=5000$, $\Phi_{min}=1.04$) but on a {\em different exact solution} and observed a different $m_{TJ}$ value.
We used the configuration in Figure 4 of  \cite{GOTTSTEIN2002703} instead (see Figure \ref{fig:voorhees}).

\begin{figure}[H]
        \subfloat[n=128]{%
            \includegraphics[width=.5\linewidth,
    trim=10cm 0cm 10cm 0cm,
    clip]{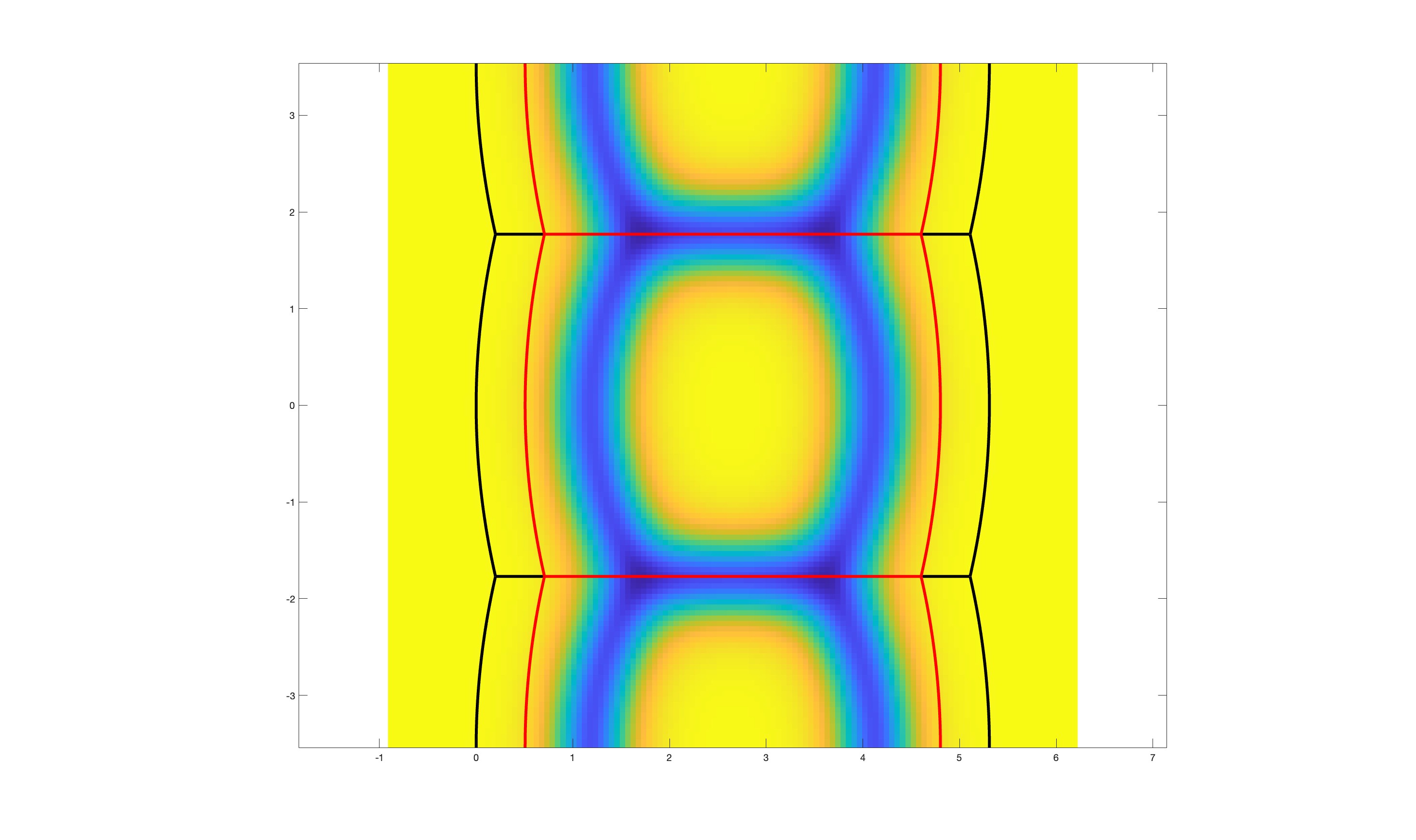}%
            \label{subfig:voorhees128}%
        }\hfill
        \subfloat[n=256]{%
            \includegraphics[width=.5\linewidth,
    trim=10cm 0cm 10cm 0cm,
    clip]{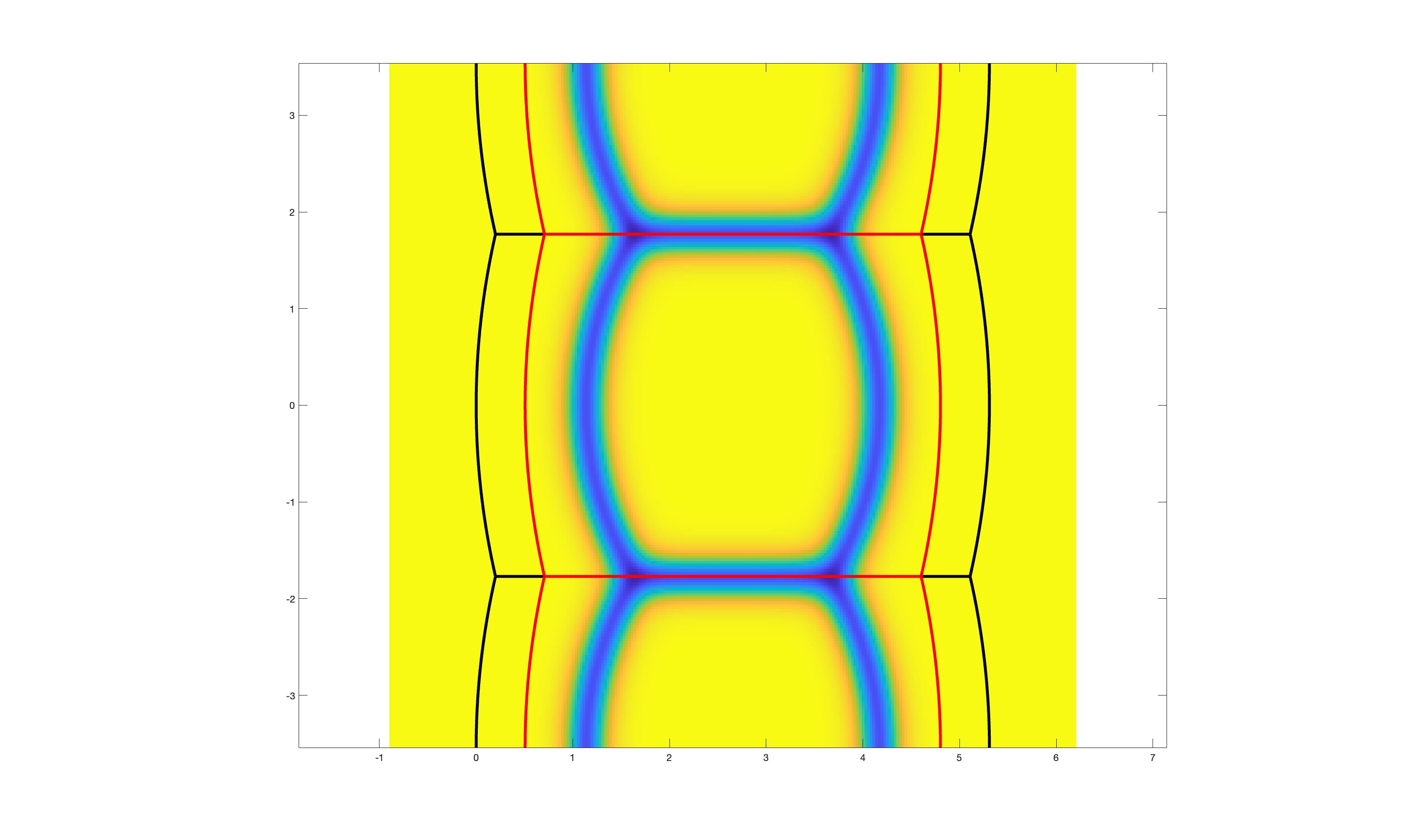}%
            \label{subfig:voorhees256}%
        }\\
        \subfloat[n=512]{%
            \includegraphics[width=.5\linewidth,
    trim=10cm 0cm 10cm 0cm,
    clip]{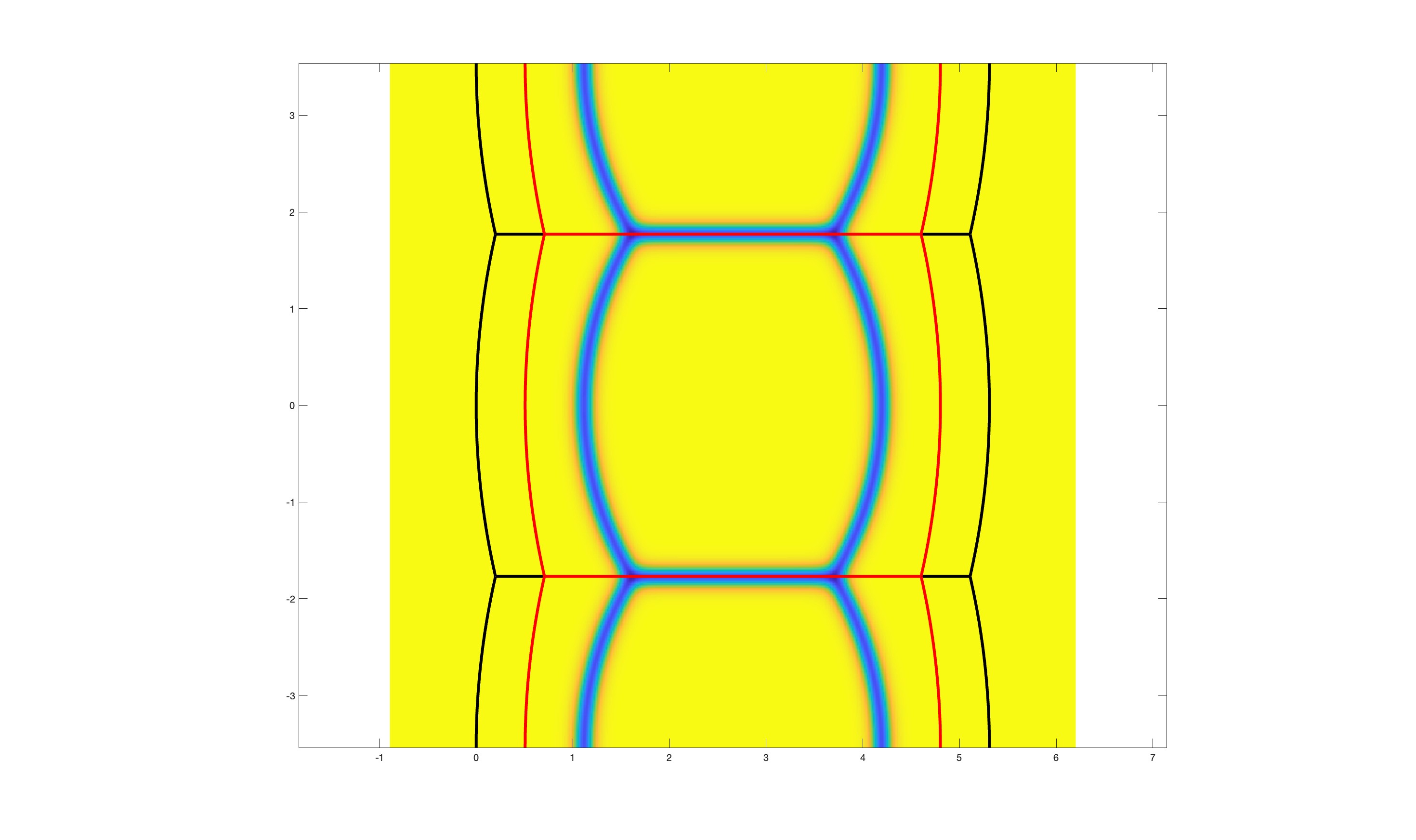}%
            \label{subfig:voorhees512}%
        }\hfill
        \subfloat[n=1024]{%
            \includegraphics[width=.5\linewidth,
    trim=10cm 0cm 10cm 0cm,
    clip]{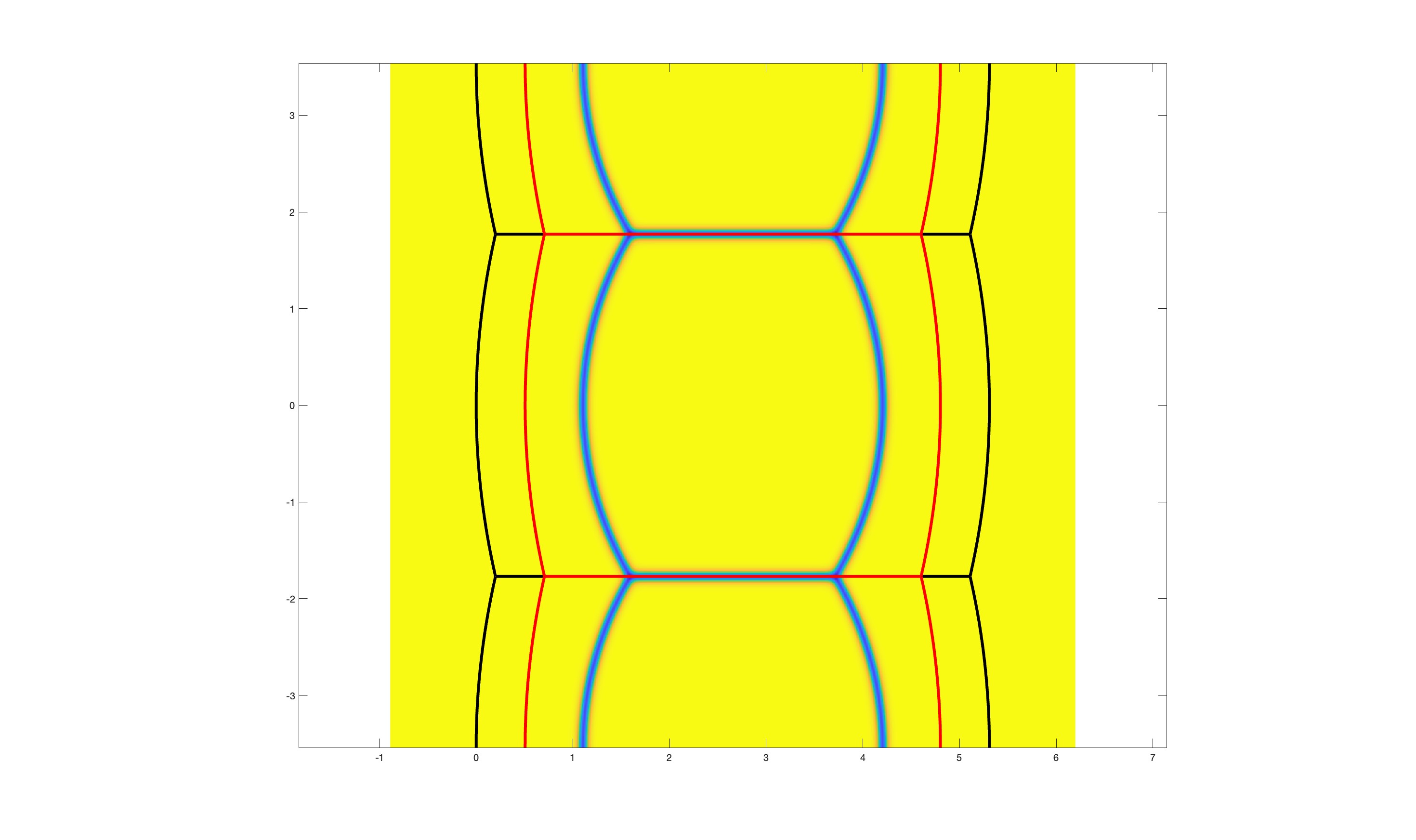}%
            \label{subfig:voorhees1024}%
        }
        \caption{Black curve: Initial condition for an exact, traveling wave solution of (\ref{curvaturemotion}) \& (\ref{drag}). Red curve: The exact solution at final time. The diffuse interface model (\ref{voorheespde}) \& (\ref{eq:voorheesmobility}) of \cite{johnsonvoorhees} appears to converge, but to a different evolution where the mobility of the junction is substantially greater than anticipated.}
        \label{fig:voorhees}
\end{figure}

We compute the error between the phase-field solution and the exact solution by comparing the average area of the symmetric differences of the grains (see \eqref{eq:avg_error} \& \eqref{eq:L1error} for a more detailed explanation of how the error is computed). Table \ref{tab:table_voorhees} and Figure \ref{fig:error_voorhees} show that the error does not converge to zero as the spatial resolution $\delta x$ and the width of the diffuse interface are simultaneously refined.

\begin{center}
\begin{minipage}[t]{0.54\textwidth}
    \centering
    \vspace{30pt}
    \def\arraystretch{1.5}
    \small
    \resizebox{\linewidth}{!}{
        \begin{tabular}{|l|l|l|l|l|}
        \hline
        $\delta x$ &  $\delta t$            & Error  & Order \\ \hline
        7.078/(128-1) = 0.0557      &  $6.21 \times 10^{-4}$ & 9.77  & -     \\ \hline
        7.078/(256-1) = 0.0278      &  $1.54 \times 10^{-4}$ & 7.66  & 0.35 \\ \hline
        7.078/(512-1) = 0.0139      &  $3.83 \times 10^{-5}$ & 7.15  & 0.10  \\ \hline
        7.078/(1024-1) = 0.0069     &  $9.57\times 10^{-6}$  & 6.80 & 0.07 \\ \hline
        \end{tabular}
        }
        \vspace{20pt}
        \captionof{table}{Stagnation of error and lack of convergence to the exact solution shown in Figure \ref{fig:voorhees}  of method (\ref{voorheespde}) \& (\ref{eq:voorheesmobility}). }
    \label{tab:table_voorhees}
    \end{minipage}
    \hfill
    \begin{minipage}[t]{0.44\textwidth}
        \centering
        % \vspace{0pt}
        \begin{figure}[H]
            \centering
            \includegraphics[width=1\linewidth]{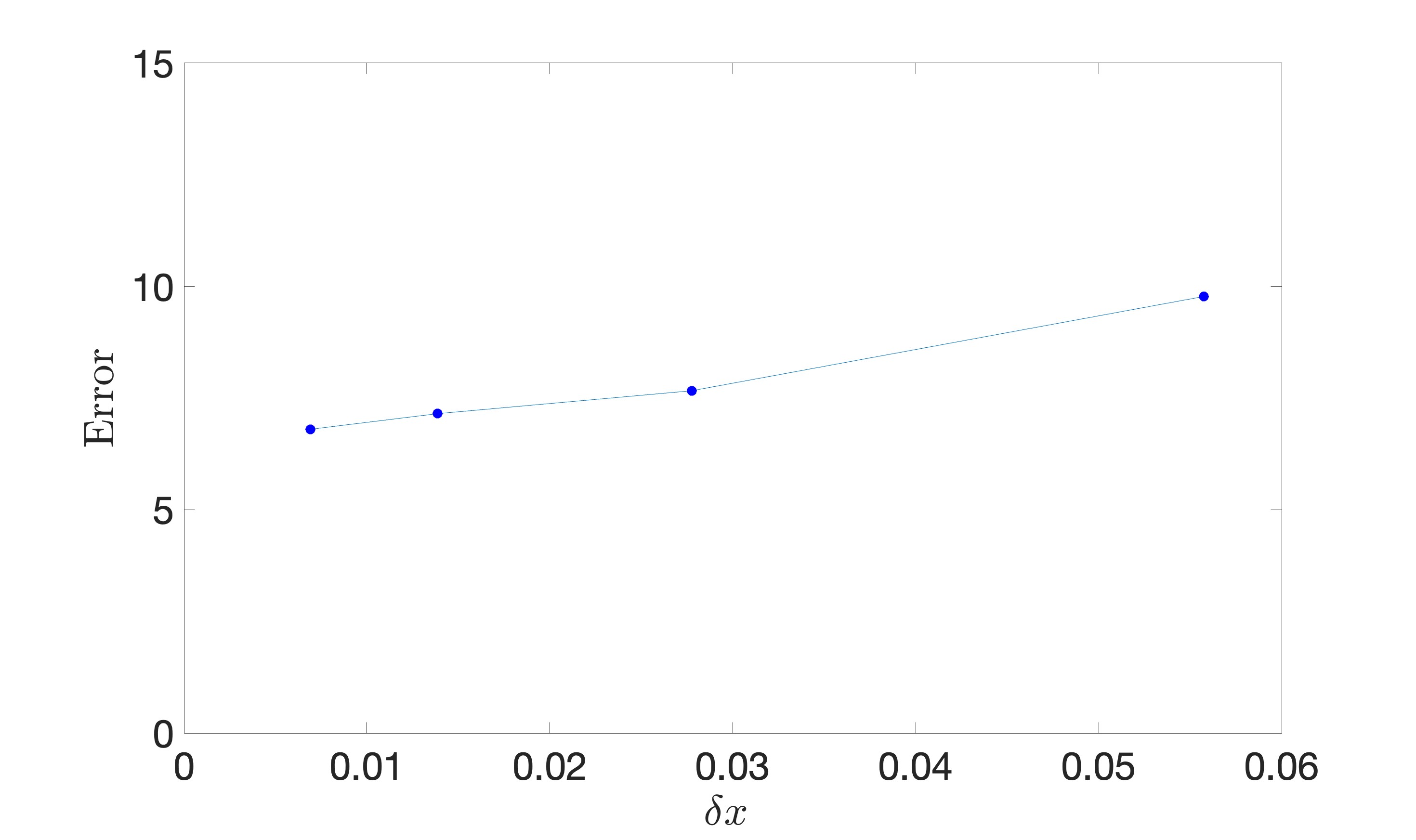}%
            \caption{Plot of error against $\delta x$ for method (\ref{voorheespde}) \& (\ref{eq:voorheesmobility}) on the test shown in Figure \ref{fig:voorhees}.}
            \label{fig:error_voorhees}
        \end{figure}
    \end{minipage}
\end{center}

\begin{figure}[H]
    \centering
        \includegraphics[width=.6\linewidth]{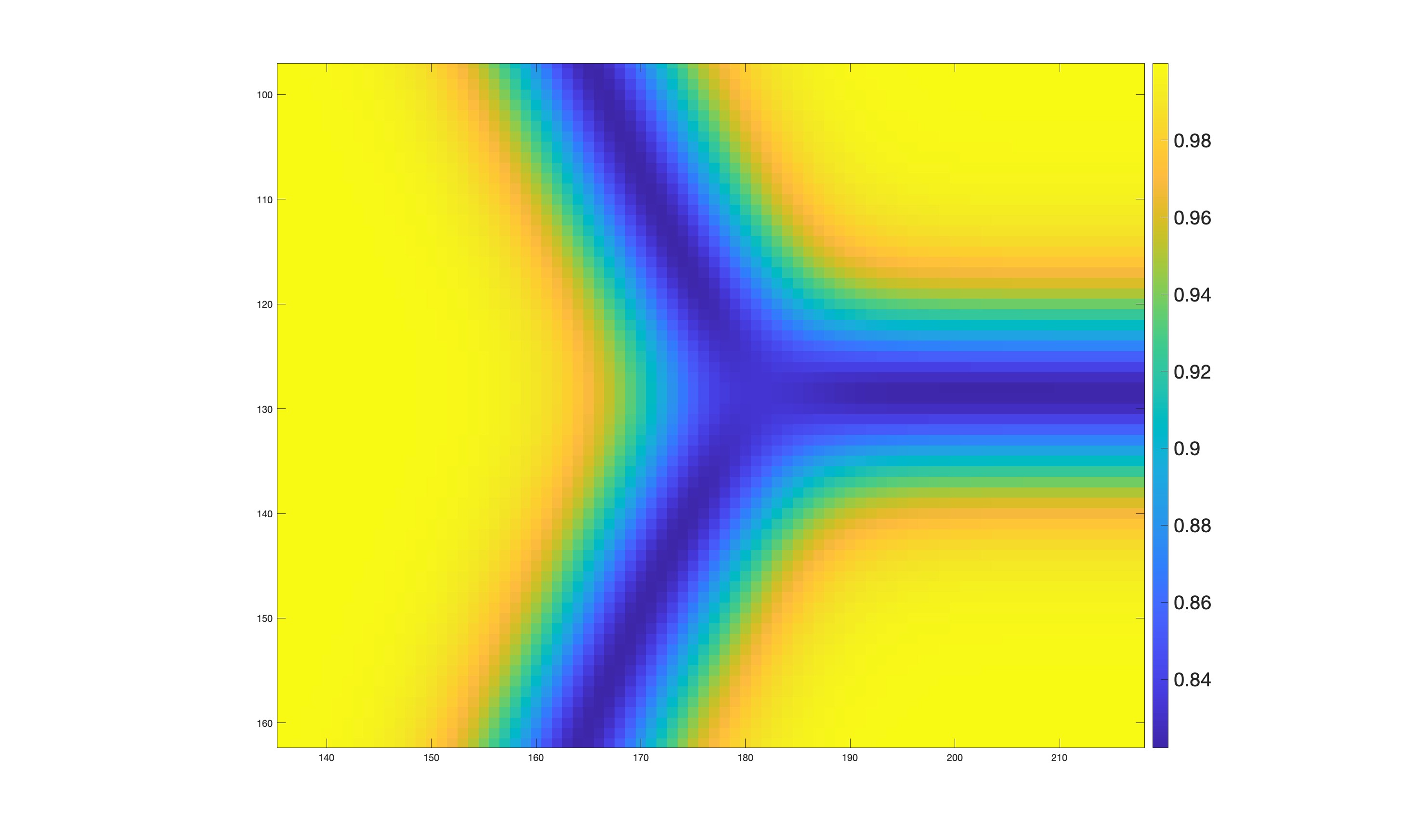}%
        \caption{Colormap of $\eta_1 + \eta_2 + \eta_3$ used in \cite{johnsonvoorhees} as a junction indicator, near a triple junction}
        \label{fig:voorhees_triple_junction}
\end{figure}

In particular, we find that at the triple junction of this exact solution, $\eta_1(x) + \eta_2(x) + \eta_3(x) \approx 0.83$ (see Figure \ref{fig:voorhees_triple_junction}), resulting in $L(\eta)\approx 1$ throughout the entire domain, which then implies the mobility of this junction is far greater than the anticipated value $0.228$.
The point is that with the approach of \cite{johnsonvoorhees}, {\em different triple junctions get assigned different mobilities}.
In Figure \ref{fig:voorhees}, the black and red curves represent the initial and final configuration of an exact traveling wave solution to \eqref{curvaturemotion} \& \eqref{drag}, given by \eqref{translating_solution}, \eqref{angle} \& \eqref{exact_velocity} with parameters $m_{TJ} = 0.228$ and $x_0 = 0.2$. The curves are plotted over a colormap of the function $|\eta(x,T)|^2 = \sum_{i=1}^3 \eta_i(x,T)^2$, where $\eta(x,t)$ was numerically computed using an explicit Euler discretization of \eqref{voorheespde}. The yellow regions represent the grains, demarcated by a thin diffuse interface (blue). Figure \ref{fig:voorhees} demonstrates that the choice of parameters ($L_{TJ} = 0.1$, $L_{GB}=1$, $\Phi_{width}=5000$, $\Phi_{min}=1.04$) very likely induced a triple junction mobility that is significantly higher than $0.228$.

In the more recent work \cite{miyoshi2025}, the authors propose another phase-field method for triple junction drag by incorporating an \emph{order-parameter dependent scalar mobility factor} into the multi-phase-field model of \cite{steinbach}. A similarity between \cite{miyoshi2025} and \cite{johnsonvoorhees} is that both models use the order parameters to detect the locations of the triple junctions and then assign a low mobility value to the vicinity of the triple junctions. \cite{miyoshi2025} also adopts a data-driven approach to estimate the relationship between the assigned mobility value and the observed (physical) triple junction mobility value.
However, \cite{miyoshi2025} does not investigate whether this phase-field method induces the \emph{same} physical triple junction mobility value for all possible configurations of triple junctions and all directions of their motion.

The central idea shared by \cite{johnsonvoorhees,miyoshi2025} of a mobility factor that assigns a lower value in the vicinity of triple junctions is natural, and forms the inspiration for the approach in this study. 
However, designing this mobility factor to induce a desired mobility at all triple junctions is nontrivial since, as explained above, an infinitude of junction configurations can arise in model (\ref{curvaturemotion}) \& (\ref{drag}) that all need to be retarded equally.

\section{New phase-field approximation for triple junction drag}
\label{sec:ourmodel}
Inspired by \cite{johnsonvoorhees}, our goal is to develop a phase-field approximation to \eqref{curvaturemotion} \& \eqref{drag} that can effectively slow down the motion of triple junctions. The challenge is to design a mobility factor $M$ so that the motion of any triple junction is retarded by a factor that is {\em independent} of the junction profile (i.e. the junction angles, which may change in time) as well as the direction of its motion (in relation to the interfaces forming the junction). 

To explain our proposed phase-field method in the simplest setting, let us assume that there are $N=3$ phases in our spatial domain that is a square $\Omega = [-L,L]^2$ with periodic boundary conditions. The evolution of our system will be described, in the style of Baldo \cite{BALDO199067}, by a vectorial order parameter $u : \Omega \times [0,T] \to \R^2$. Let $W:\R^2 \to \R$ be a triple-well potential defined by
\begin{align}\label{triplewellpotential}
    W(x) \coloneqq C_\alpha |x - \alpha_1|^2 |x - \alpha_2|^2 |x - \alpha_3|^2
\end{align}
where $\alpha = (\alpha_1,\alpha_2,\alpha_3)$ are the three wells. To ensure symmetry, we choose the wells to be the cube roots of unity, i.e. $\alpha_1 = (1,0)$, $\alpha_2 = (-1/2 , \sqrt{3}/2)$, $\alpha_3 = (-1/2 , -\sqrt{3}/2)$. The normalization constant $C_\alpha \approx 1.837^{-2}$ ensures that the \emph{geodesic distance} \cite{BALDO199067} between each pair of wells (and hence the surface tension between each pair of neighboring grains) is equal to one.
This geodesic distance is defined for $\alpha,\beta\in\R^2$ by
\begin{align}
\label{eq:distance}
    d(\alpha,\beta)\coloneqq \inf_{c}\bigg\{ \int_0^1 \sqrt{2W(c(s))}|\dot{c}(s)|\;ds \;:\; c\in C^1([0,1];\R^2),\;c(0)=\alpha,\;c(1) = \beta \bigg\}.
\end{align}
The surface tension coefficients $\sigma_{ij}$ are given by the geodesic distances between the wells of $W$:
\begin{align}
    \sigma_{ij} = d(\alpha_i,\alpha_j).
\end{align}
We remark that for certain choices of potential $W$ (such as \eqref{triplewellpotential}), the geodesics between each pair of the three wells and hence the $\sigma_{ij}$ can be explicitly determined \cite{ALIKAKOS_BETELU_CHEN_2006}.
For other choices of potential $W$ where no explicit formulas are available, one can compute the geodesic distances numerically.

Finally, our proposed phase-field method for curvature motion with triple junction drag is given by the PDE system
\begin{equation}\label{eq:dragphasefield}
    u_t(x,t) = M(\nabla u(x,t))\bigg( \Delta u (x,t) - \frac{1}{\eps^2}\nabla W(u(x,t)) \bigg)
\end{equation}
where
\begin{equation}\label{eq:mobility}
    M(\nabla u) \coloneqq I_{2\times 2} - \frac{1}{A_W}J_u\nabla u \bigg[\eps m_{TJ}\bigg(\nabla u^\top \nabla u + I_{2\times 2} \bigg)^2 + \frac{1}{A_W}J_u \nabla u^\top \nabla u \bigg]^{-1}\nabla u^\top 
\end{equation}
is a matrix-valued mobility factor. Here, $J_u$ is the Jacobian determinant
\begin{equation}
\label{eq:jacobian}
    J_u = \sqrt{\det(\nabla u(x,t)^\top \nabla u(x,t))}.
\end{equation}
The constant $A_W \approx 0.86$ is completely determined by the choice of potential $W$, and will be defined and explained in the next section.

We can use the Woodbury formula to rewrite the mobility factor \eqref{eq:mobility} as
\begin{equation}\label{eq:woodbury}
    M(\nabla u)^{-1} = I_{2\times 2} + \frac{1}{\eps m_{TJ}}\frac{J_u}{A_W}\nabla u \bigg( \nabla u^\top \nabla u + I_{2\times 2} \bigg)^{-2}\nabla u^\top.
\end{equation}
which makes it apparent that the symmetric matrix $M(\nabla u)$ is positive definite.
Therefore, our method (i.e. \eqref{eq:dragphasefield} \& \eqref{eq:mobility}) satisfies the energy dissipation identity
\begin{equation}\label{eq:energydissipation}
    \frac{d}{dt} E_\eps (u) = -\eps \int_\Omega \bigg(\Delta u - \frac{1}{\eps^2}\nabla W(u)\bigg)^\top M(\nabla u) \bigg(\Delta u - \frac{1}{\eps^2}\nabla W(u)\bigg)\;dx \leq 0
\end{equation}
where the energy $E_\varepsilon$ is the well-known Modica-Mortola \cite{modica_mortola} approximation to perimeter:
\begin{align}
\label{eq:mm}
    E_\eps(u)\coloneqq \int_{\Omega}\frac{\eps}{2}|\nabla u(x)|^2 + \frac{1}{\eps}W(u(x))\;dx, \qquad u\in H^1(\Omega;\R^2).
\end{align}

We highlight two main differences between our mobility factor \eqref{eq:mobility} and those of \cite{johnsonvoorhees,miyoshi2025}.
First, \eqref{eq:mobility} depends on $\nabla u(x,t)$ vs. $u(x,t)$.
Second, \eqref{eq:mobility} is matrix-valued, whereas those of \cite{johnsonvoorhees,miyoshi2025} are scalar-valued.
These differences are essential in enabling our method to achieve the correct triple junction dynamics for \emph{all} triple junction configurations.

\section{The Jacobian determinant}
\label{sec:jacobian}
A distinctive feature of our method \eqref{eq:dragphasefield} \& \eqref{eq:mobility} is its reliance on the Jacobian determinant $J_u$ (see \eqref{eq:jacobian}) of the order parameter $u(x,t)$. The purpose of this section is to explain its geometric meaning and the role that it plays in influencing triple junction dynamics. 

\begin{figure}[H]
            \includegraphics[width=.5\linewidth,
    trim=10cm 0cm 10cm 0cm,
    clip]{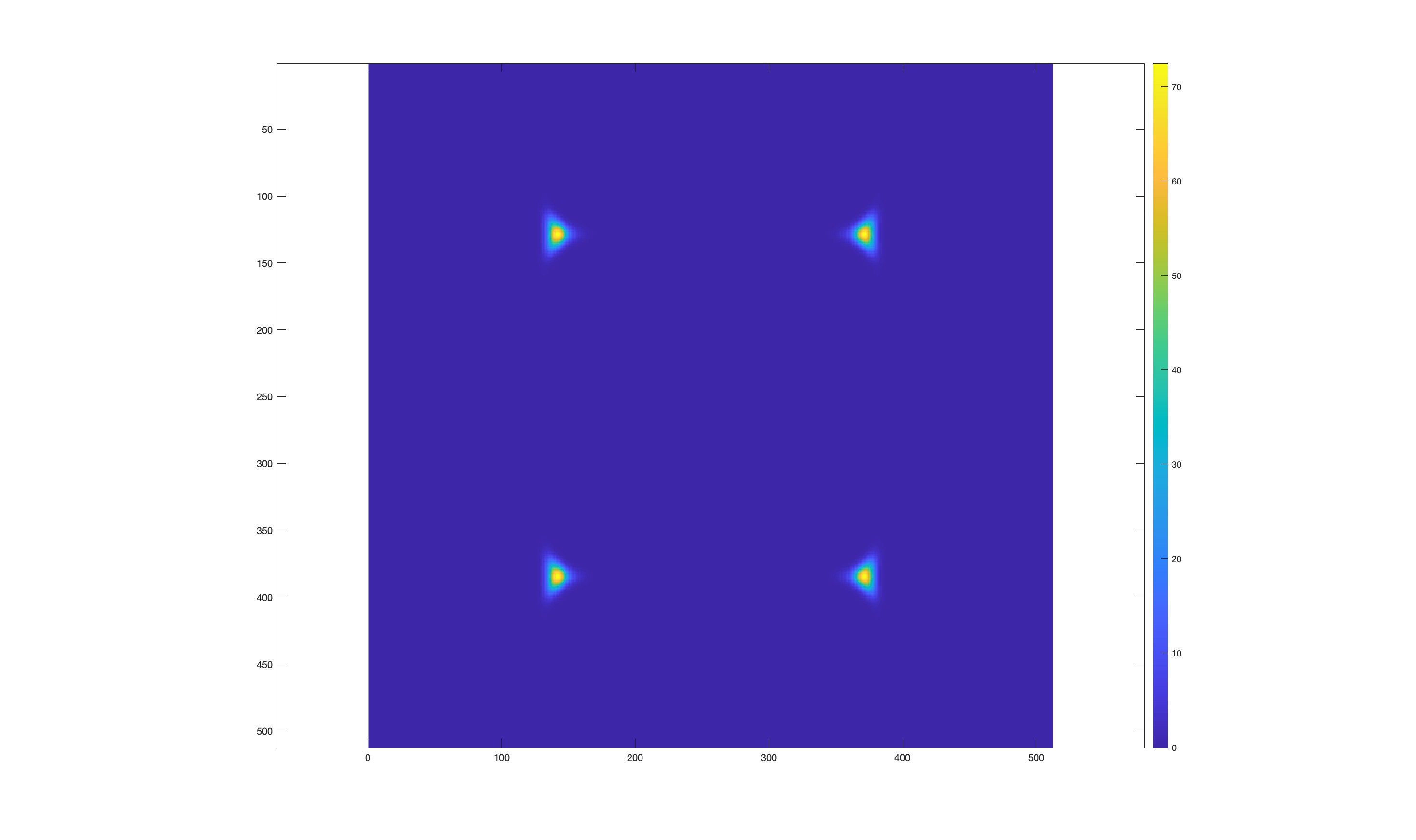}%
            \label{subfig:a}%
            \includegraphics[width=.5\linewidth,
    trim=10cm 0cm 10cm 0cm,
    clip]{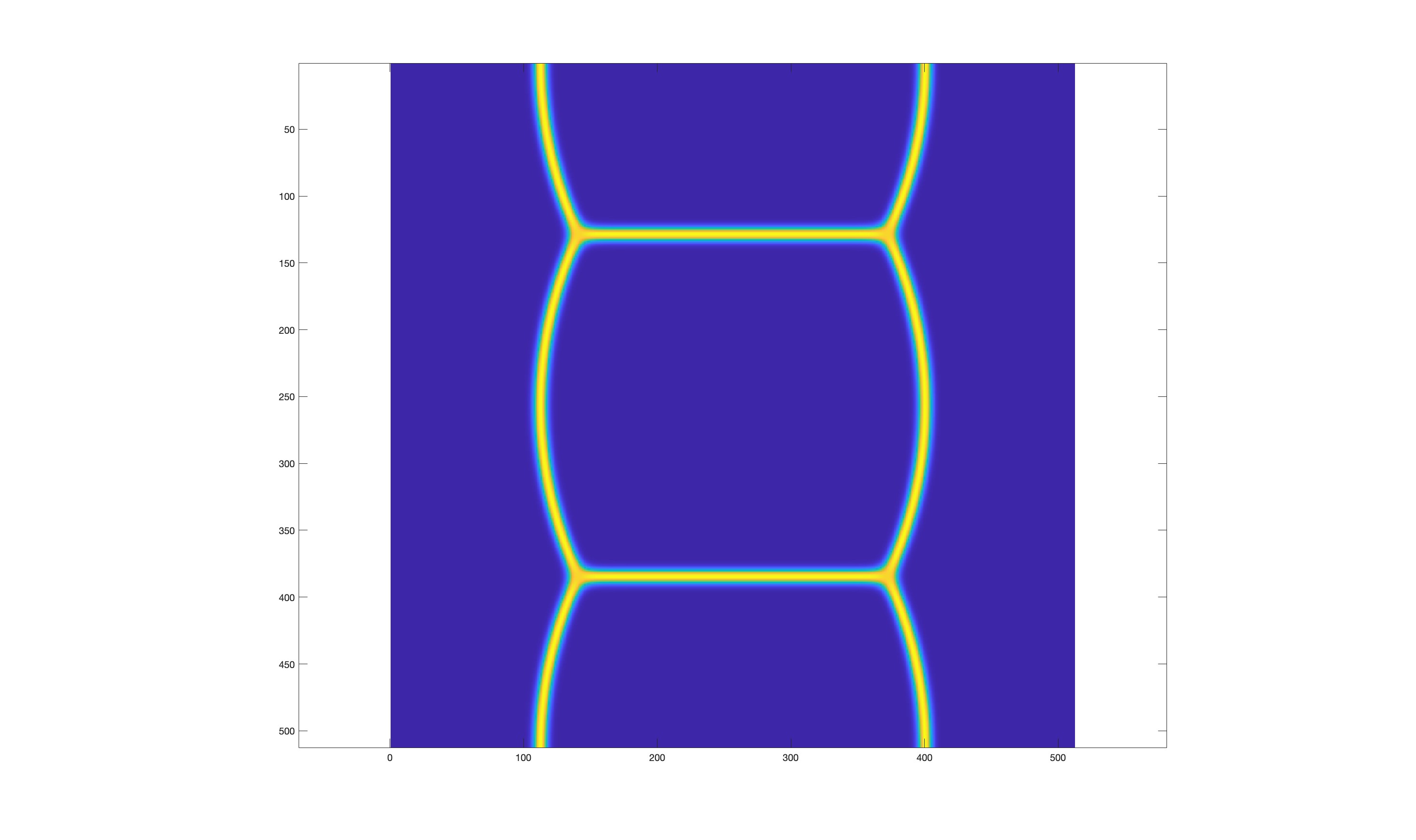}%
            \label{subfig:b}%
        %}
        \caption{The Jacobian $J_u$ given in \eqref{eq:jacobian} concentrates near triple junctions, as shown in the left panel. Corresponding diffuse grain boundaries are shown in the right panel. }
        \label{fig:jacobian}
\end{figure}

To fix ideas, let $u(x,t)$ be a solution of the usual vectorial Allen-Cahn equation (with constant mobility)
\begin{align}
   u_t = \Delta u - \frac{1}{\eps^2}\nabla W(u).
\end{align}
Assume that $u$ is of `triple-junction type': it divides the domain $\Omega$ into three regions in each of which $u \approx \alpha_1$, $\alpha_2$ or $\alpha_3$, and the (diffuse) grain boundaries meet at a small neighborhood $B_\eps$ of the triple junction.
Away from the triple junction, the diffuse layer quickly settles into an optimal one-dimensional profile as the solution transitions from $u\approx \alpha_i$ to $u\approx \alpha_j$, with $i\not= j$
This optimal profile is given by a particular parametrization of the geodesic connecting $\alpha_i$ to $\alpha_j$ (\cite{bronsardreitich} Lemma 1).
Hence, roughly speaking, $u(\cdot,t):\Omega \to \R^2$ maps the interior of each grain to the corresponding well $\alpha_i$, and the diffuse grain boundaries (away from the junction) to the geodesics between each pair of wells $\alpha_i$ and $\alpha_j$.
Therefore, we expect $J_u$ to be concentrated in a neighborhood $B_\eps$ of triple junctions just like in Figure \ref{fig:jacobian}, and $u(\cdot,t)$ to map $\Omega$ onto the \textit{geodesic triangle} $\mathcal{T}\subset\R^2$, which is the region bounded by the geodesics connecting the wells $\alpha_i$, $i=1,2,3$ (see Figure \ref{fig:geodesics}).
\begin{figure}[H]
    \centering        \includegraphics[width=.7\linewidth]{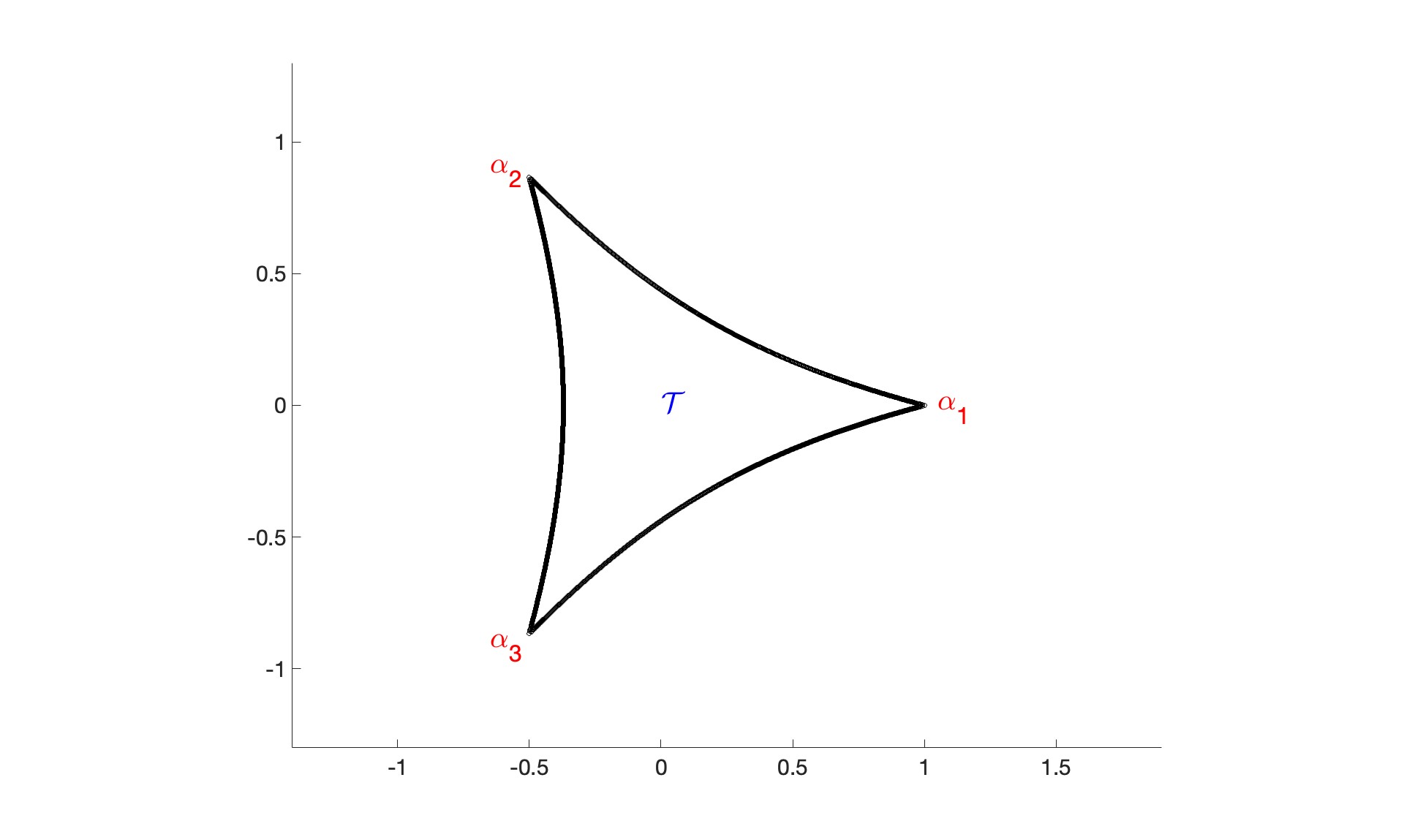}%
        \caption{Geodesics between three wells $\alpha_1$, $\alpha_2$, $\alpha_3 \in \R^2$ and the enclosed geodesic triangle $\mathcal{T}$.}
        \label{fig:geodesics}
\end{figure}

Moreover, the change-of-variables formula tells us that
\begin{equation}\label{eq:changeofvar}
\int_{B_\eps} J_u (x) \;dx \approx \mbox{Area} \big( u(B_\eps,t) \big) \approx \mbox{Area}(\mathcal{T}) \eqqcolon A_{W},
\end{equation}
which defines the constant $A_W$ solely in terms of the potential $W$ as promised, and as a byproduct gives us a nifty new formula for the total number of junctions $N_{TJ}(t)$: At any time $t>0$ during the evolution when there are no critical events occurring (e.g. a collision among triple junctions), we can compute the number of triple junctions $N_{TJ}(t)$ by
\begin{align}
\label{eq:ntj}
    \# \text{Triple junctions} := N_{TJ}(t) \approx \frac{1}{A_{W}}\int_\Omega J_u(x) \;dx
\end{align}
(see Figure \ref{fig:jacobianherring}). Therefore, the function $\frac{1}{A_{W}} J_u$ acts as a normalized (unit mass) bump (approximate delta) function at each triple junction and can be readily used to locate and track them during the evolution (see Figures \ref{fig:herring1},\ref{fig:herring2} \& \ref{fig:herring3}).
Furthermore, the topological nature of the argument given above justifying \eqref{eq:changeofvar}, as well as its reliance only on the potential $W$, also makes it very robust: As we will see, it holds even for Allen-Cahn systems of the form \eqref{eq:dragphasefield} with nontrivial mobilities, where the local {\em geometry} of (i.e. the angles at) a triple junction can vary in time -- a feature essential for our derivation of the new method \eqref{eq:dragphasefield} \& \eqref{eq:mobility} below.

\begin{figure}[H]
    \centering        \includegraphics[width=.9\linewidth]{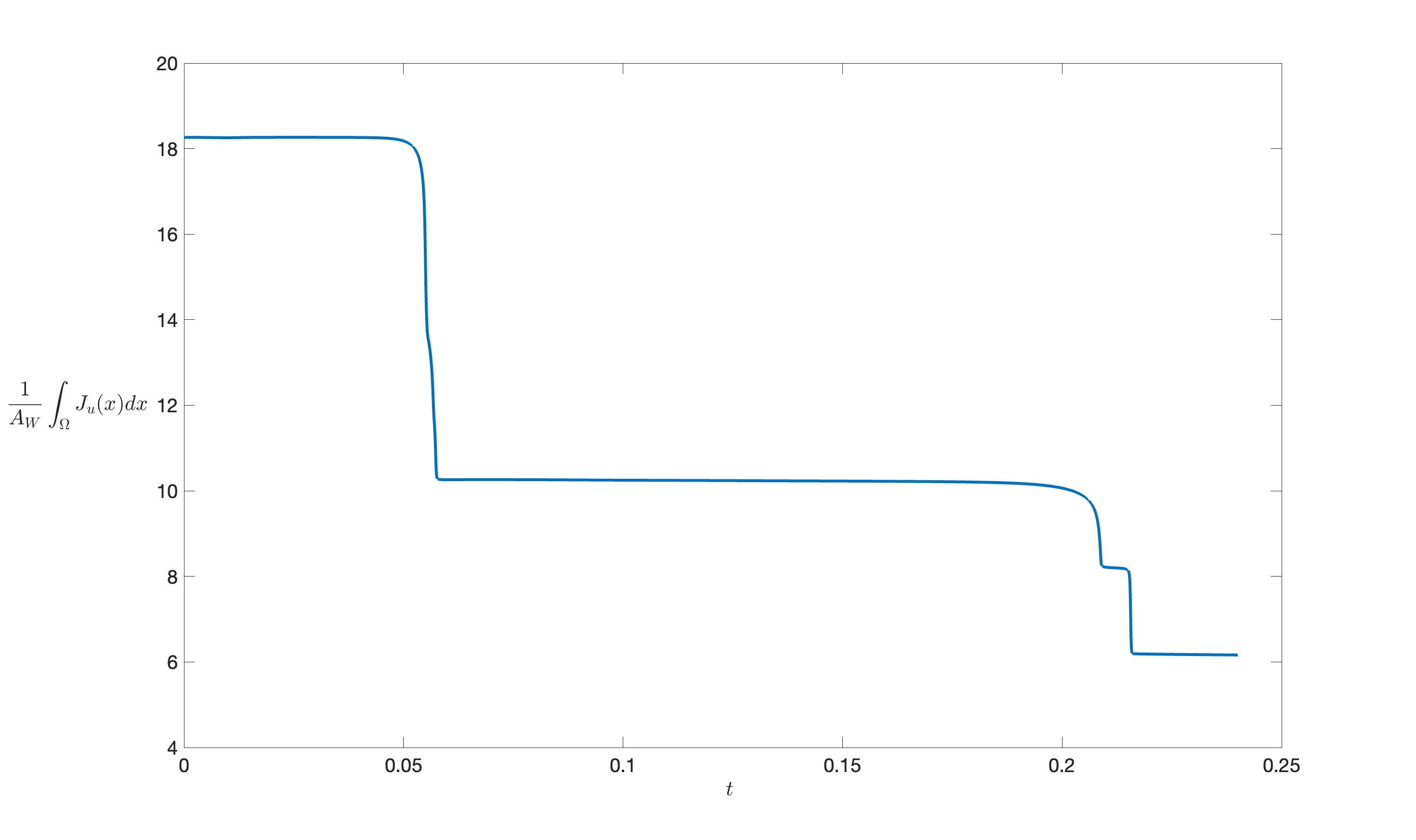}%
        \caption{A plot of the quantity $\frac{1}{A_{W}}\int_\Omega J_u(x) \;dx$ against time $t$. Here, $u(x,t)$ is a solution to the standard vectorial Allen-Cahn equation with constant mobility and symmetric potential, so that the usual Herring angle condition is induced at all triple junctions.}
        \label{fig:jacobianherring}
\end{figure}

\begin{figure}[H]
    \centering        \includegraphics[width=.9\linewidth]{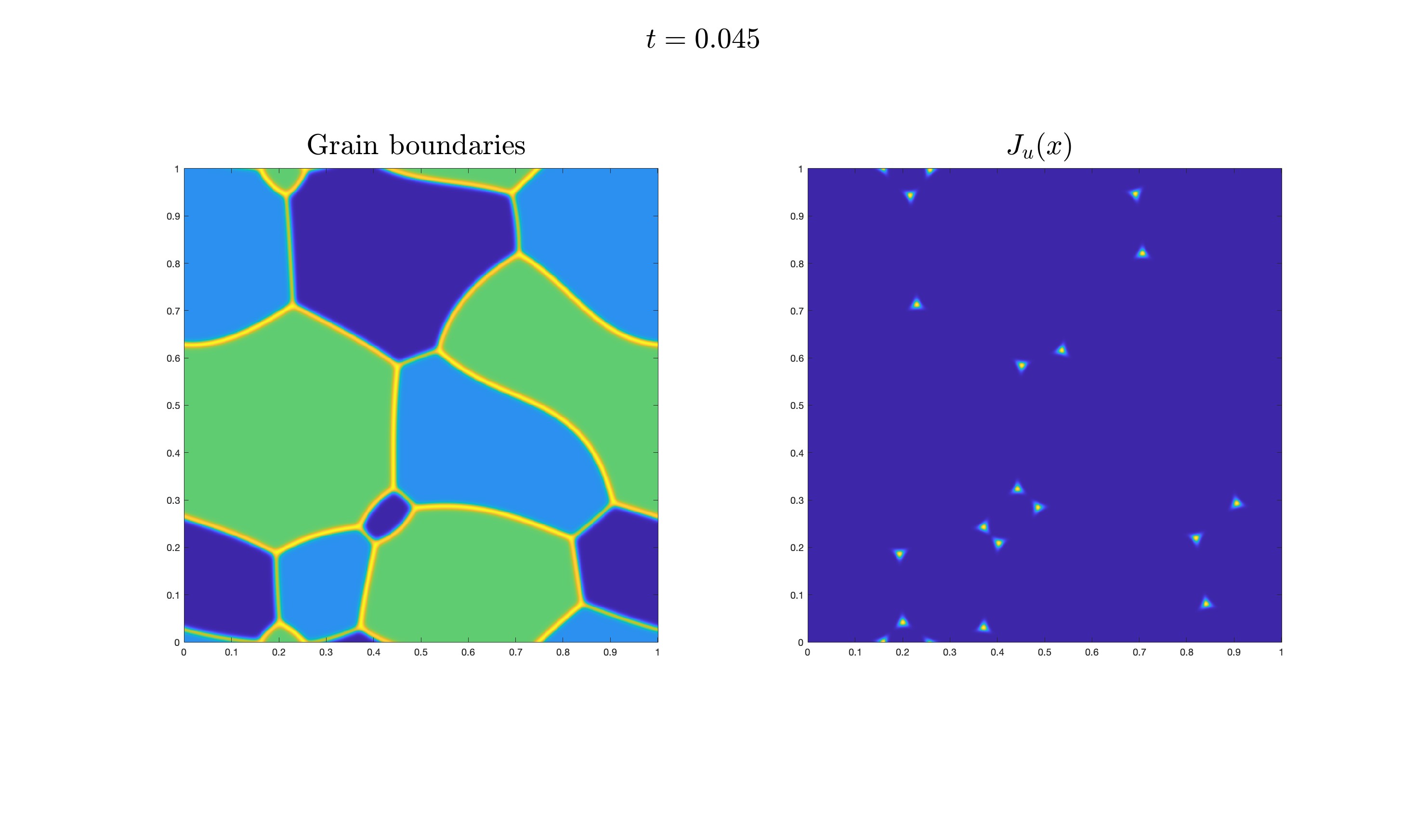}%
        \caption{Evolution of grain boundaries (3 phases) at time $t=0.045$. Each phase is represented by a different color and there are 18 triple junctions.}
        \label{fig:herring1}
\end{figure}

\begin{figure}[H]
    \centering        \includegraphics[width=.9\linewidth]{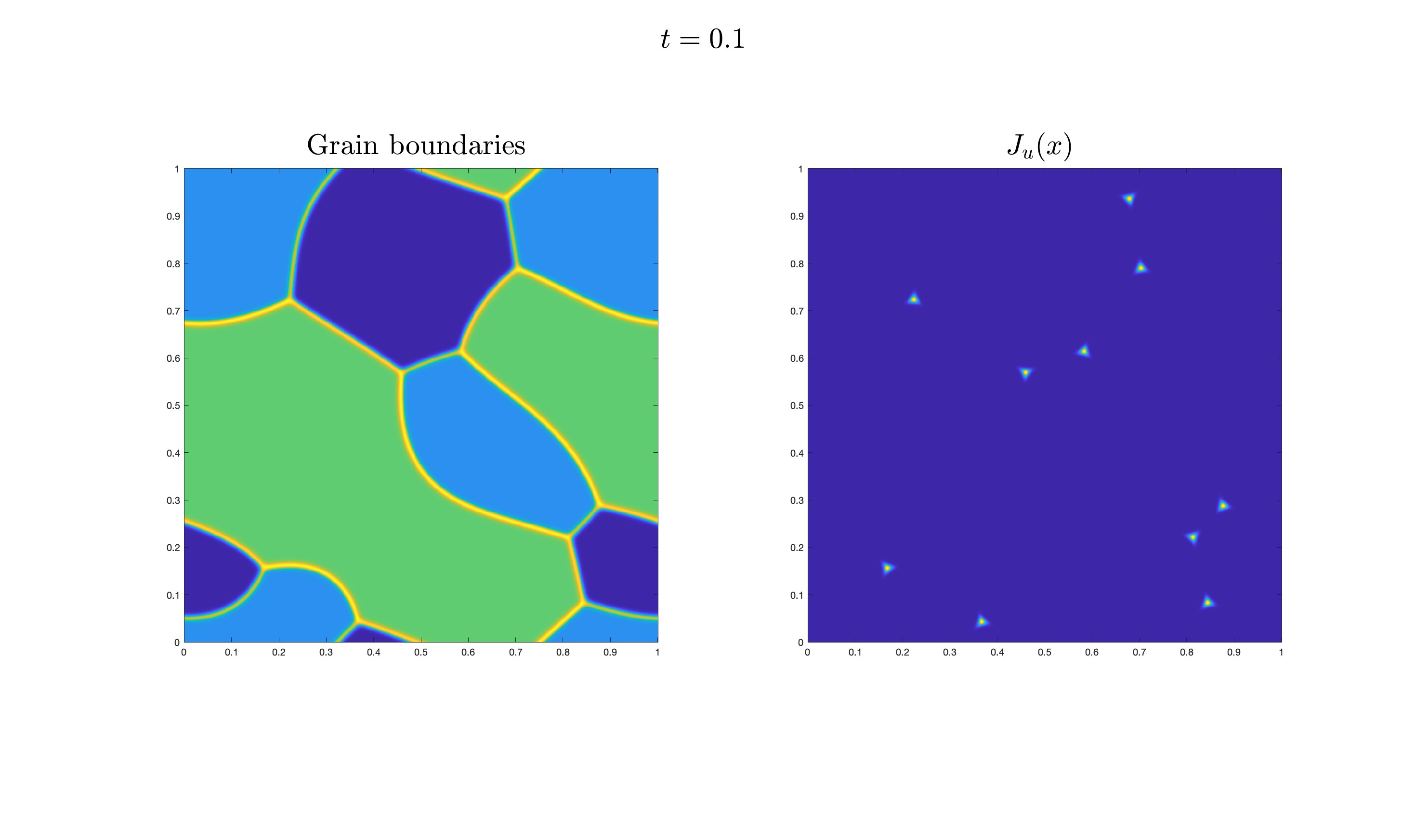}%
        \caption{Evolution of grain boundaries (3 phases) at time $t=0.1$. 10 triple junctions remain.}
        \label{fig:herring2}
\end{figure}

\begin{figure}[H]
    \centering        \includegraphics[width=.9\linewidth]{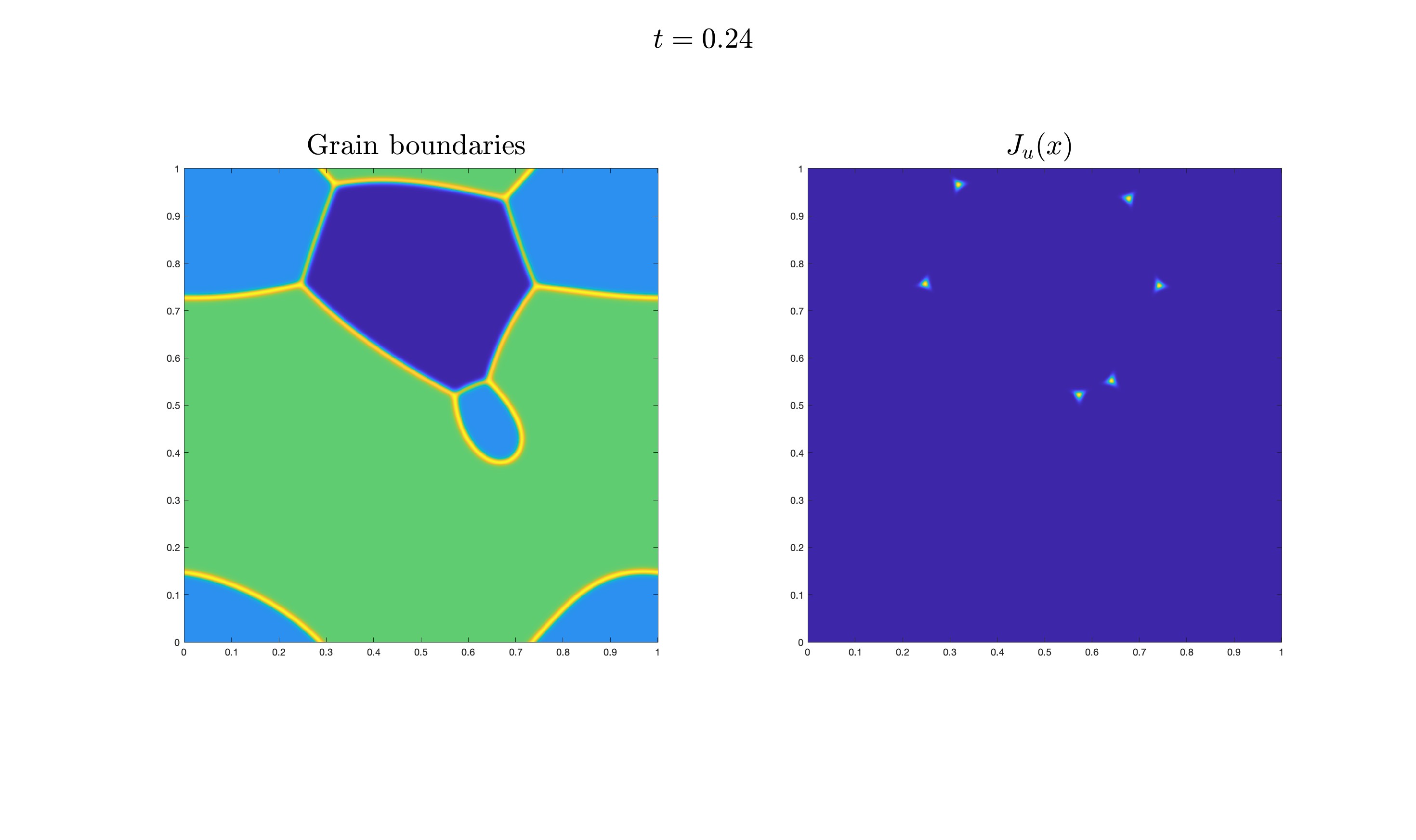}%
        \caption{Evolution of grain boundaries (3 phases) at time $t=0.24$. 6 triple junctions remain. Jacobian $J_u$ is an effective junction indicator throughout the evolution.}
        \label{fig:herring3}
\end{figure}

While the Jacobian plays an important role in the theory of Ginzburg-Landau vortices (e.g. \cite{ginzburglandau}), it and its relation to Baldo's triangle $\mathcal{T}_W$ appears to have received little attention in the study of Allen-Cahn systems.
We believe that the role of the Jacobian in the context of phase-field methods for multiphase curvature flows may be of interest more broadly.

\section{Variational derivation of \eqref{eq:dragphasefield}}
\label{sec:variational}
The mobility factor \eqref{eq:mobility} can be derived from a variational perspective. Recall from \eqref{eq:metric} that triple junction drag arises as a consequence of the triple junction's contribution to the metric. In this vein, we would like to modify the gradient flow structure of the Allen-Cahn equation by introducing an additional penalty term that approximates the movement of the triple junction. Before that, we first review the basic gradient flow structure of the Allen-Cahn equation.

For an energy $E(\cdot)$ defined on a metric space $(X,d)$, there is a natural time-discretization of its gradient flow via the \textit{minimizing-movement scheme} \cite{degiorgi}
\begin{equation}\label{minimizingmovement}
\begin{split}
    x_{n+1} &\in \argmin_{x} \bigg\{E(x) + \frac{1}{2\delta t}d^2(x,x_n) \bigg\}\\
    x_0 &= x^0, \qquad \text{(initial data)}
\end{split}
\end{equation}
with time step size $\delta t$. When we choose the energy to be the Allen-Cahn (or Modica-Mortola \cite{modica_mortola}) approximation to perimeter given by \eqref{eq:mm}, 
 and the metric to be the (rescaled) $L^2$ norm $\big(\eps\int_{\Omega}|u(x)-u_n(x)|^2\;dx\big)^{1/2}$, the minimizing-movement scheme \eqref{minimizingmovement} generates a discrete-in-time approximation of the gradient flow
\begin{align}\label{allencahn}
    u_t = -\frac{\delta E_\eps}{\delta u} = \Delta u - \frac{1}{\eps^2}\nabla W(u)
\end{align}
which is the (vectorial) Allen-Cahn equation \cite{allencahn} whose sharp-interface limit (as $\eps\to 0$) is multiphase motion by mean curvature with the Herring angle condition satisfied at triple junctions \cite{bronsardreitich,lauxsimon}. The connection between the Allen-Cahn equation and motion by mean curvature can be motivated by the fact that the energies $E_\eps(\cdot)$ $\Gamma-$converge (in the $L^1$ topology) to $E(\cdot)$, which generalizes the total length functional to partitions of the domain $\Omega$ via
\begin{align}
    E(\Sigma_1,\Sigma_2,\Sigma_3) \coloneqq \frac{1}{2}\sum_{i,j=1}^3\sigma_{ij}\mathcal{H}^1(\partial^* \Sigma_i \cap \partial^* \Sigma_j)
\end{align}
for $\Sigma_1\cup\Sigma_2\cup\Sigma_3 = \Omega$ and $|\Sigma_i\cap\Sigma_j| = 0$ whenever $i\neq j$ \cite{BALDO199067}. Here, $\partial^*\Sigma_i$ refers to the reduced boundary of $\Sigma_i$ in $\Omega$. (Recall that we have normalized the potential $W$ so that all surface tension coefficients $\sigma_{ij}$ are equal to one.)

Based on our discussion above, the minimizing-movement scheme for the Allen-Cahn equation \eqref{allencahn} can be written as
\begin{equation}\label{eq:acminmov}
\begin{split}
    u_{n+1} &\in \argmin_{u} \bigg\{ E_\eps(u) + \frac{\eps}{2\delta t}\int_\Omega |u(x)-u_n(x)|^2\;dx\bigg\}\\
    u_0 &= u^0.
\end{split}
\end{equation}
Following the gradient flow interpretation of \eqref{curvaturemotion} \& \eqref{drag}, we will add to the metric term $\frac{\eps}{2\delta t}\int_\Omega |u(x)-u_n(x)|^2\;dx$ in (\ref{eq:acminmov}) a penalty for the perturbation of triple junctions.
To demonstrate this, suppose for simplicity that at the $n-$th time step, $u_n$ approximates a partition of $\Omega$ into three grains, separated by three grain boundaries meeting at a triple junction $p_n \in \mathbb{R}^2$.
Suppose also that $u$ (very close to $u_n$) also approximates such a partition, with a triple junction at $p$ that is close to $p_n$. 
Our scheme will be an approximation to
\begin{align}\label{min_mvmt_drag}
    u_{n+1} \in \argmin_{u} \bigg\{ E_\eps(u) + \frac{\eps}{2\delta t}\int_\Omega |u(x)-u_n(x)|^2\;dx + \bigg(\frac{1}{m_{TJ}}\bigg)\bigg(\frac{1}{2\delta t}\bigg)|p-p_n|^2 \bigg\}.
\end{align}

In order to make scheme (\ref{min_mvmt_drag}) practical, we need to approximate, using the order parameters $u$ and $u_n$, the penalty $|p-p_n|^2$ on deviations from the triple junction location $p_n$ at the $n$-th time step. To do so, suppose 
\begin{equation}
\begin{split}
    u(x) &= \mathcal{U}\bigg(\frac{x-p}{\eps}\bigg),\\
    u_n(x) &= \mathcal{U}\bigg(\frac{x-p_n}{\eps}\bigg)
\end{split}
\end{equation}
where $\mathcal{U}$ is a smooth profile (cf. $u^0$ in \eqref{eq:expansion}). Taking a Taylor expansion,
\begin{equation}
    u(x) - u_n(x) \approx \nabla \mathcal{U}\bigg(\frac{x-p_n}{\eps}\bigg)\bigg(\frac{p-p_n}{\eps}\bigg), \qquad \bigg|\frac{p-p_n}{\eps}\bigg| \ll 1.
\end{equation}
Note that this approximation makes sense since we often take the time step size $\delta t$ to be asymptotically smaller than the width of the diffuse layer $\eps$. Then,
\begin{equation}\label{eq:tjperturbation}
\begin{split}
    &\int_\Omega (u(x) - u_n(x))^\top \frac{J_{u_n}(x)}{A_W}\nabla u_n(x) \big( \nabla u_n(x)^\top  \nabla u_n(x) + I_{2\times 2}\big)^{-2}\nabla u_n(x)^\top (u(x) - u_n(x)) \; dx \\
    &\approx \frac{1}{\eps^2} \int_\Omega (p_n - p)^\top \frac{J_\mathcal{U}}{A_W}(\nabla\mathcal{U})^\top (\nabla\mathcal{U})\bigg((\nabla\mathcal{U})^\top (\nabla\mathcal{U}) + \eps^2 I_{2\times 2}\bigg)^{-2} (\nabla\mathcal{U})^\top (\nabla\mathcal{U}) (p_n - p)\;dx.
\end{split}
\end{equation}
where all terms in the integrand involving $\mathcal{U}$ are evaluated at $(x-p_n)/\eps$. Here, the (normalized) Jacobian determinant $J_{u_n}(x)/A_W$ acts as an approximate delta function at the triple junction $p_n$ while the identity matrix $I_{2\times 2}$ acts as a regularization wherever $\nabla u_n$ may not be rank 2 (i.e. away from the triple junction).

Making a change of variables $y = (x-p_n)/\eps$, the last integral is equal to
\begin{equation}
\begin{split}
    \int_{(\Omega - p_n)/\eps} (p_n - p)^\top &\frac{J_\mathcal{U}}{A_W}(\nabla\mathcal{U})^\top (\nabla\mathcal{U})\bigg((\nabla\mathcal{U})^\top (\nabla\mathcal{U}) + \eps^2 I_{2\times 2}\bigg)^{-2} (\nabla\mathcal{U})^\top (\nabla\mathcal{U}) (p_n - p)\;dy\\
    &= |p-p_n|^2 + \mathcal{O}(\eps|p-p_n|^2)
\end{split}
\end{equation}
where the last equality will be justified in the next section (Lemma \ref{lem:asymptotic}). This means that we can incorporate the integral on the first line of \eqref{eq:tjperturbation} as an approximation to the triple junction perturbation $|p-p_n|^2$ into scheme \eqref{min_mvmt_drag}, rewriting it as
\begin{equation}
\begin{split}
    u_{n+1} &\in \argmin_{u} \bigg\{ E_\eps(u) + \frac{\eps}{2\delta t}\int_\Omega |u(x)-u_n(x)|^2\;dx + \\
    &\bigg(\frac{1}{m_{TJ}}\bigg)\bigg(\frac{1}{2\delta t}\bigg) \int_\Omega (u(x) - u_n(x))^\top \frac{J_{u_n}(x)}{A_W}\nabla u_n(x) \big( \nabla u_n(x)^\top  \nabla u_n(x) + I_{2\times 2}\big)^{-2}\nabla u_n(x)^\top (u(x) - u_n(x)) \; dx
    \bigg\},
\end{split}
\end{equation}
whose formal limit as $\delta t \to 0$ is \eqref{eq:dragphasefield}.

\section{Formal asymptotic analysis}
\label{sec:asymptotics}
The preceding section derives equation \eqref{eq:dragphasefield} from a minimizing-movement perspective, where we designed a proxy for the triple junction perturbation $|p-p_n|^2$. In this section, we further justify equation \eqref{eq:dragphasefield} via formal asymptotic analysis.

The analysis of the solution to \eqref{eq:dragphasefield} away from triple junctions (i.e. motion by curvature of the interfaces) follows those of \cite{RSK,bronsardreitich} closely. Thus, we focus on the behavior near a triple junctions. Let $p(t) \in \R^2$ denote the position of the triple junction at time $t$. Introducing the stretched variables
\begin{equation}
    y = \dfrac{x - p(t)}{\eps},
\end{equation}
we assume an expansion
\begin{equation}\label{eq:expansion}
    u(x,t) = u^0(y,t) + \eps u^1(y,t) +  \ldots.
\end{equation}
Expanding the various terms in \eqref{eq:dragphasefield}, we get
\begin{equation}\label{eq:asymptotic1}
\begin{split}
    u_t(x,t) &= -\frac{1}{\eps}\nabla u^0(y,t)\dot{p}(t) + u_t^0(y,t) + \ldots\\
    \Delta u &- \frac{1}{\eps^2}\nabla W(u) = \frac{1}{\eps^2}\bigg( \Delta u^0 - \nabla W(u^0) \bigg) + \ldots.
\end{split}
\end{equation}

Multiplying both sides of \eqref{eq:dragphasefield} by $(\nabla u)^\top M(\nabla u)^{-1}$, we get
\begin{equation}\label{eq:rearranged}
    (\nabla u)^\top M(\nabla u)^{-1} u_t = (\nabla u)^\top\bigg(\Delta u - \frac{1}{\eps^2} \nabla W(u)\bigg).
\end{equation}
Plugging in the expansion for $u(x,t)$ into the 
expression \eqref{eq:woodbury} for the inverse mobility $M(\nabla u)^{-1}$ gives
\begin{equation}\label{eq:asymptotic2}
    M(\nabla u)^{-1} = \frac{1}{\eps m_{TJ}}\frac{J_{u^0}}{A_W} \nabla u^0 \big( (\nabla u^0)^\top \nabla u^0 +\eps^2 I_{2\times 2}\big)^{-2}(\nabla u^0)^\top + I_{2\times 2} + \ldots.
\end{equation}
The leading order term on the left-hand side of \eqref{eq:rearranged} is
\begin{equation}\label{eq:LHSleading}
\begin{split}
    \bigg(\frac{1}{\eps}(\nabla u^0)^\top \bigg) & \bigg(\frac{1}{\eps m_{TJ}}\frac{J_{u^0}}{A_W} \nabla u^0 \big( (\nabla u^0)^\top \nabla u^0 +\eps^2 I_{2\times 2}\big)^{-2}(\nabla u^0)^\top\bigg)\bigg(-\frac{1}{\eps}\nabla u^0(y,t)\dot{p}(t)\bigg)\\
    = &\frac{-J_{u^0}}{\eps^3 m_{TJ}A_W}B (B + \eps^2 I_{2\times 2})^{-2}B \dot{p}(t)
\end{split}
\end{equation}
where $B \coloneqq (\nabla u^0)^\top \nabla u^0$ is a positive semidefinite matrix.

\begin{lemma}\label{lem:asymptotic}
    \begin{equation}
        J_{u^0} B (B + \eps^2 I_{2\times 2})^{-2}B = J_{u^0}I_{2\times 2} + \mathcal{O}(\eps)
    \end{equation}
    where the constant appearing in the error term depends only on the largest singular value of $\nabla u^0$.
\end{lemma}

\begin{proof}
    We first orthogonally diagonalize 
    \begin{equation}
        B = Q \Lambda Q^\top, \qquad \Lambda = \begin{pmatrix}
            \lambda_1 & 0 \\ 0 & \lambda_2
        \end{pmatrix},
        \qquad \lambda_1,\lambda_2 \geq 0
    \end{equation}
    and note that $J_{u^0} = \sqrt{\lambda_1 \lambda_2}$. Then,
    \begin{equation}
        J_{u^0} B (B + \eps^2 I_{2\times 2})^{-2}B = Q\begin{pmatrix}
            \mu_1(\eps) & 0 \\ 0 & \mu_2(\eps)
        \end{pmatrix}
        Q^\top
    \end{equation}
    where 
    \begin{equation}
        \mu_i(\eps) \coloneqq \frac{\lambda_i^2 \sqrt{\lambda_1 \lambda_2}}{(\lambda_i + \eps^2)^2}, \qquad i = 1,2.
    \end{equation}
    Our goal is to show that 
    \begin{equation}
        |\mu_i(\eps) - \sqrt{\lambda_1 \lambda_2}| \leq C\eps
    \end{equation}
    for some constant $C$. WLOG, we can just consider the case when $i=1$.
    \begin{equation}
        \begin{split}
            |\mu_i(\eps) - \sqrt{\lambda_1 \lambda_2}| &= \sqrt{\lambda_1 \lambda_2}\bigg|1-\frac{\lambda_1^2}{(\lambda_1 + \eps^2)^2}\bigg|\\
            &= \eps\sqrt{\lambda_2}f(\eps),
        \end{split}
    \end{equation}
    where 
    \begin{equation}
        f(\eps) \coloneqq \sqrt{\lambda_1}\bigg(\frac{2\lambda_1 \eps + \eps^3}{(\lambda_1 + \eps^2)^2}\bigg)\geq 0.
    \end{equation}
    We claim that 
    \begin{equation}
        \sup_{\eps>0} f(\eps) \leq 1.
    \end{equation}
    Indeed, setting $z = \eps/\sqrt{\lambda_1}$, we get
    \begin{equation}
        f(\eps) = \frac{z(2+z^2)}{(1+z^2)^2} \leq \frac{2z}{1+z^2} \leq 1.
    \end{equation}
    Thus,
    \begin{equation}
        |\mu_i(\eps) - \sqrt{\lambda_1 \lambda_2}| \leq \sqrt{\lambda_2}\eps \leq C\eps
    \end{equation}
    where 
    \begin{equation}
        C \coloneqq \sup_{y\in\R^2} \sqrt{\lambda_{\max}((\nabla u^0)^\top \nabla u^0)}.
    \end{equation}
    
\end{proof}

By \eqref{eq:LHSleading}, the leading order term on the LHS of \eqref{eq:rearranged} is
\begin{equation}
    \frac{-J_{u^0}}{\eps^3 m_{TJ}A_W} \dot{p}(t)
\end{equation}
while the the leading order term on the RHS of \eqref{eq:rearranged} is
\begin{equation}
    \frac{1}{\eps^3}(\nabla u^0)^\top \bigg( \Delta u^0 - \nabla W(u^0) \bigg).
\end{equation}
Thus, the leading order profile $u^0$ solves
\begin{equation}\label{eq:tjprofile}
    (\nabla u^0)^\top \bigg( \Delta u^0 - \nabla W(u^0) \bigg) = \frac{-J_{u^0}}{m_{TJ}A_W} \dot{p}(t),
\end{equation}
subject to the same matching conditions as \cite{bronsardreitich} equation (18) (see also \cite{alikakos2018elliptic} Definition 3.2). In words, as $|y| \to \infty$ along each of the three directions tangent to an interface (i.e. $\tau_i$), the solution $u^0$ converges to the 1-dimensional optimal profile (geodesic) connecting two adjacent wells of $W$.

\begin{figure}[H]
    \centering        \includegraphics[width=.7\linewidth]{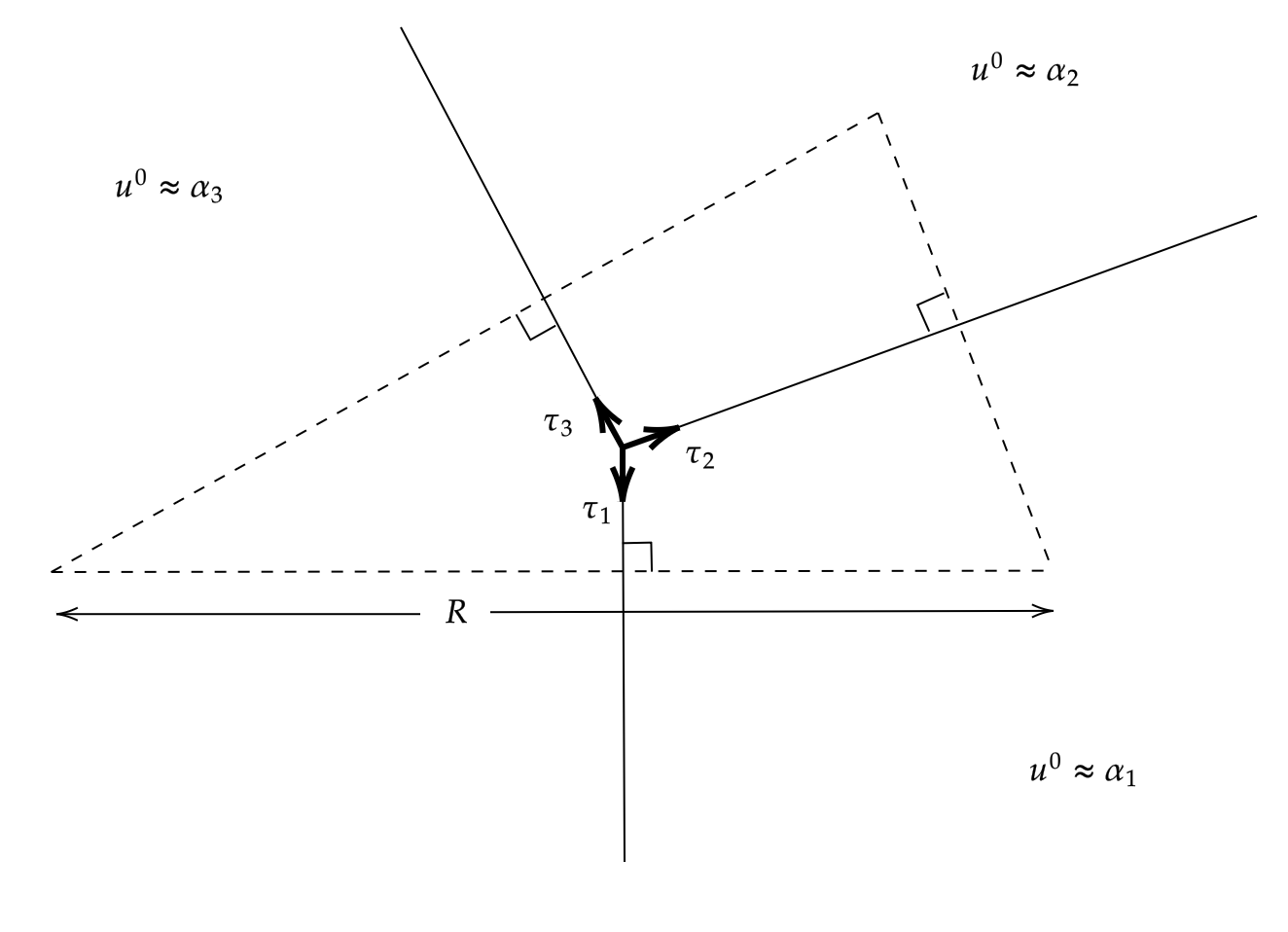}%
        \caption{The domain of integration $T_R$ in \eqref{eq:forcebalance}, same as the one used in \cite{bronsardreitich} for standard Allen-Cahn.}
        \label{fig:triangle}
\end{figure}

Following \cite{bronsardreitich}, we integrate \eqref{eq:tjprofile} over a large triangular domain $T_R$ with base length $R$ (Figure \ref{fig:triangle}):
\begin{equation}\label{eq:forcebalance}
     -\frac{1}{m_{TJ}}\frac{1}{A_W}\bigg(\int_{T_R} J_{u^0}(y)\;dy\bigg) \dot{p}(t) = \int_{T_R}(\nabla u^0)^\top \big(\Delta u^0 - \nabla W(u^0)\big) \; dy.
\end{equation}
The integrand on the RHS is the divergence of a tensor $T$ defined by
\begin{equation}
    T_{ij} = u^0_{y_i} \cdot u^0_{y_j} - \delta_{ij}\bigg(\frac{1}{2}|\nabla u^0|^2 + W(u^0)\bigg), \qquad i,j \in \{1,2\}.
\end{equation}
It was shown in \cite{bronsardreitich} (see also \cite{alikakos2018elliptic} Chapter 3) that 
\begin{equation}
\begin{split}
    \int_{T_R}(\nabla u^0)^\top \big(\Delta u^0 - \nabla W(u^0)\big) \; dy &= \int_{T_R} \nabla\cdot T \; dy \\
    &= \int_{\partial T_R} T\nu \; dS\\
    &\to -(\tau_1 + \tau_2 + \tau_3) , \qquad \text{ as }R \to \infty.
\end{split}
\end{equation}
Recalling the change-of-variables formula from the previous section and the matching condition satisfied by $u^0$ at infinity, we have
\begin{equation}
    \int_{T_R} J_{u^0}(y)\;dy \to A_W, \text{ as } R\to\infty.
\end{equation}
Thus, as $R \to \infty$, we deduce from \eqref{eq:forcebalance} the triple junction drag condition
\begin{equation}
    \dot{p}(t) = m_{TJ}(\tau_1 + \tau_2 + \tau_3).
\end{equation}

\begin{remark}
    With an eye toward efficiency of numerical implementation and scalability, we have intentionally written our mobility factor $M(\nabla u)$ in the form of \eqref{eq:mobility} to emphasize that we only need to invert a $2\times 2$ matrix. This facilitates the generalization of our model to the case where there may be $N\geq 4$ phases and the order parameter $u(x,t)$ takes values in $R^{N-1}$ (see Section \ref{sec:Nphases}). However, if we are only concerned with the case where there are $N=3$ phases and $\nabla u \in \R^{2\times 2}$, then the mobility factor $M(\nabla u)$ in \eqref{eq:mobility} can be substituted for a simplified version, namely
\begin{equation}\label{eq:simplifiedmob}
    M(\nabla u) = \bigg(I_{2\times 2} + \dfrac{J_u}{A_W \eps m_{TJ}} \big(\nabla u \nabla u^\top + I_{2\times 2}\big)^{-1} \bigg)^{-1},
\end{equation}
which yields the same leading order equation \eqref{eq:tjprofile}.
\end{remark}

\section{Numerical simulations}
\label{sec:numerics}
We test convergence of our phase-field method \eqref{eq:dragphasefield} \& \eqref{eq:mobility} to the sharp interface description \eqref{curvaturemotion} \& \eqref{drag} under refinement of discretization size and the diffuse interface thickness $\eps$. Periodic boundary conditions are used on the computational domain that is a square, which is discretized into $n=256, 512, 1024$ and $2048$ grid points along each dimension.
The following finite differences discretization of equation \eqref{eq:dragphasefield} was used on a uniform grid:
\begin{equation}\label{eq:udiscretized}
    \frac{u_{n+1} - u_n}{\delta t} = M(\nabla_{\delta x} u_n)\bigg( \Delta_{\delta x} u_n - \frac{1}{\eps^2}\nabla W(u_n) \bigg).
\end{equation}
where $\Delta_{\delta x}$ denotes the standard five point discretization of the Laplacian on a uniform grid, and $\nabla_{\delta x}$ uses centered differences for first derivatives.
The discretization of (\ref{eq:udiscretized}) was kept fully explicit and as simple as possible, as our focus here is on verifying convergence of \eqref{eq:dragphasefield} \& \eqref{eq:mobility} rather than finding the most efficient implementation.

In practice, we observe that with the typical choice $\eps = 6 \delta x$ for the phase field parameter $\eps$, the scheme \eqref{eq:udiscretized} is stable under the restriction
\begin{equation}\label{timestep}
    \delta t \lesssim (\delta x)^2.
\end{equation}
This is essentially the same CFL condition as that of explicit Euler scheme for the standard (constant mobility) Allen-Cahn equation.
The benign effect of the nontrivial mobility factor $M(\nabla u)$ on stability properties of scheme \eqref{eq:udiscretized} can be easily explained rigorously for a potential $W$ that has bounded Hessian (e.g. obtained by modifying the natural choice \eqref{triplewellpotential} to have quadratic growth at $\infty$).

Indeed, as the second term in \eqref{eq:woodbury} is positive, we see that the symmetric positive definite inverse mobility matrix $M(\nabla_{\delta x} u_n)^{-1}$ has minimum eigenvalue bounded below by one.
%Hence, the largest eigenvalue of $M(\nabla_{\delta x} u_n)$ is bounded from above by one.
We'll show energy stability of (\ref{eq:udiscretized}), with the Lyapunov function
\begin{equation}
\label{eq:stability0}
E_{\delta x , \eps}(u) = \sum \frac{1}{2} \eps |\nabla_{\delta x} u|^2 + \frac{1}{\eps} W(u)
\end{equation}
where the summation is over all grid points;
this is a natural discrete version of Modica-Mortola energy \eqref{eq:mm}.
To that end, first note that scheme (\ref{eq:udiscretized}) has the variational formulation
\begin{equation}
\label{eq:stability1}
\begin{split}
u_{n+1} = \argmin_u \Big\{ G_n(u) := E_{\delta x,\eps}(u_n) + &\sum -\varepsilon(u-u_n)\Delta_{\delta x} u_n + \frac{1}{\eps} \nabla W(u_n) \cdot (u-u_n)\\
+&\sum\frac{\eps}{2\delta t} \langle M(\nabla_{\delta x}u_n)^{-1} (u-u_n) \, , \,  (u-u_n) \rangle \Big\}
\end{split}
\end{equation}
Observe that the first three terms in $G_n(u)$ of (\ref{eq:stability1}) are the linearization of (\ref{eq:stability0}) at $u=u_n$, and the last term is a strongly convex proximal (movement limiting) term.
Hence, first of all,
$$ G_n(u_n) = E_{\delta x,\eps}(u_n).$$
Moreover, since by assumption $\sup_\xi \| D^2W(\xi) \|$ is bounded, there is a constant $C>0$ so that whenever $\delta t < C \min\{ \delta x^2 \, , \, \eps^2 \}$ we have
$$ E_{\delta x, \eps} (u) \leq G_n(u) \mbox{ for all } u.$$
That means
$$ E_{\delta x,\eps}(u_{n+1}) \leq G_n(u_{n+1}) \leq G_n(u_n) = E_{\delta x,\eps} (u_n) $$
establishing energy stability under the advertised CFL condition.

\subsection{Experiment 1: Traveling wave solution with $m_{TJ} = 2$}

As our first benchmark, we use the exact traveling wave solutions of \eqref{curvaturemotion} \& \eqref{drag} given in  \cite{GOTTSTEIN2002703}; these are the analogue in the presence of junction drag of ``grim-reaper'' solutions of the no-drag, Herring angle setting.
The profile of these solutions is given by
\begin{align}\label{translating_solution}
    y(x) = -\frac{x_0}{\ln(\sin(\theta))}\arccos\bigg( \exp\bigg(  \frac{x}{x_0}\ln(\sin(\theta))\bigg) \bigg)
\end{align}
where the angle $\theta$ can be determined via the relation
\begin{align}\label{angle}
    -\frac{\ln(\sin(\theta))}{1-2\cos(\theta)} = \frac{m_{TJ}x_0}{m_{GB}} = m_{TJ}x_0.
\end{align}
See Figure \ref{fig:geometry2} for an illustration. 
There, curved interfaces on the left translate rightwards with a constant velocity of
\begin{align}\label{exact_velocity}
    V = -\frac{1}{x_0}\ln(\sin(\theta)) = m_{TJ}(1-2\cos(\theta))
\end{align}
while the curved interfaces on the right translate leftwards with equal speed.
In this special benchmark solution, the third interface at each triple junction remains a straight line segment that evolves by shortening from its endpoints.
\begin{figure}[H]
    \centering
    \includegraphics[width=1\linewidth]{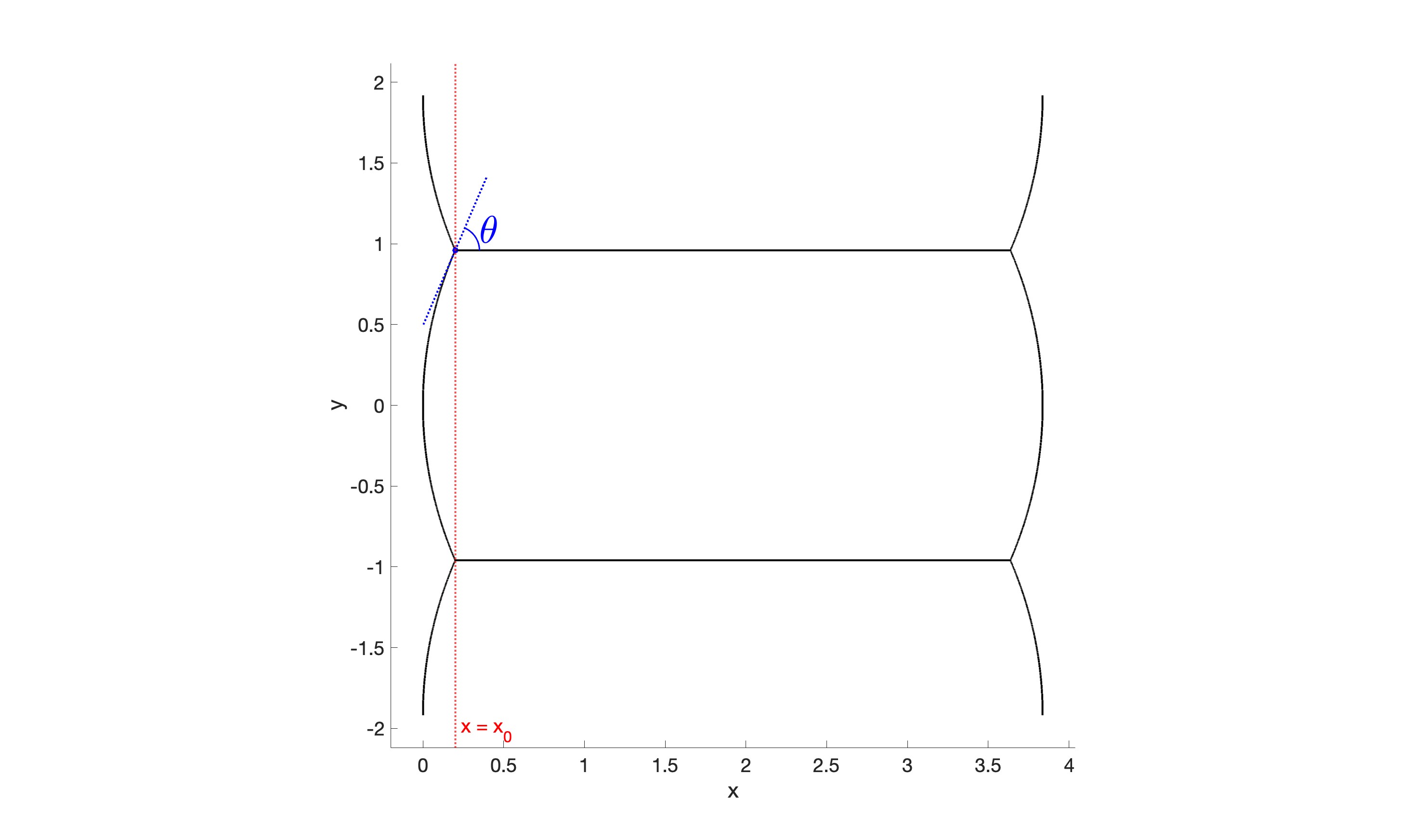}%
    \caption{Example of an exact translating solution with triple junction drag.}
    \label{fig:geometry2}%
\end{figure}

In Figures \ref{fig:highmob} \& \ref{fig:lowmob}, the black network of curves is the initial condition.
The red network of curves represents the benchmark, exact solution given by (\ref{translating_solution}) , (\ref{angle}), and (\ref{exact_velocity}).
The deviation of the exact solution from the superimposed phase-field approximation, given by equation \eqref{eq:dragphasefield}, was measured using the area of symmetric difference for each phase, i.e.
\begin{equation}\label{eq:avg_error}
\mbox{error} = \frac{1}{3}\sum_{i=1}^3 \text{Area} \Big( \Sigma_i^{\text{pf}}(T) \Delta\Sigma_i^{\text{ex}}(T) \Big)
\end{equation}
where $\Sigma_i^{\text{pf}}(t)$ represents the region occupied by the $i$-th phase at time $t$ according to the phase-field computation, and $\Sigma_i^{\text{ex}}(t)$ represents the same according to the exact solution.
The sets $\Sigma_i^{\text{pf}}(T)$ were obtained from the (vectorial) order parameters $u(x,T)$ via
\begin{align}\label{eq:L1error}
     \Sigma_i^{\text{pf}}(T) := \{x\in\Omega \;:\; |u(x,T) - \alpha_i|<\lambda\}
\end{align}
where $\lambda$ is a small threshold. In all of our numerical simulations, there are $N=3$ phases, shown in different colors.
\medskip

\begin{figure}[H]
\begin{center}
        \subfloat[$n=256$]{%
            \includegraphics[width=.31\linewidth,
    trim=20cm 0cm 20cm 0cm,
    clip]{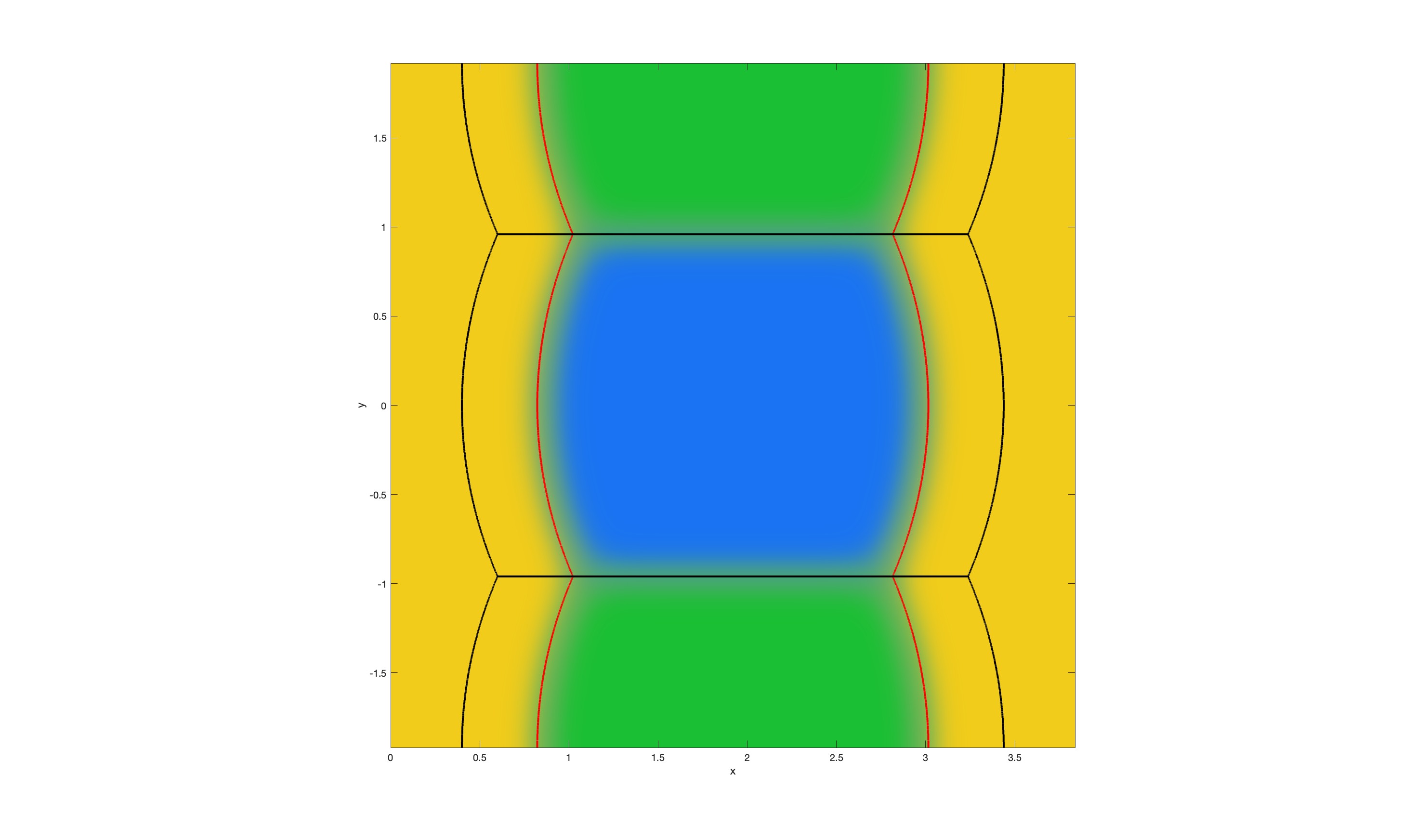}%
            \label{subfig:highmob1}%
        }\hspace{40pt}
        \subfloat[$n=512$]{%
            \includegraphics[width=.31\linewidth,
    trim=20cm 0cm 20cm 0cm,
    clip]{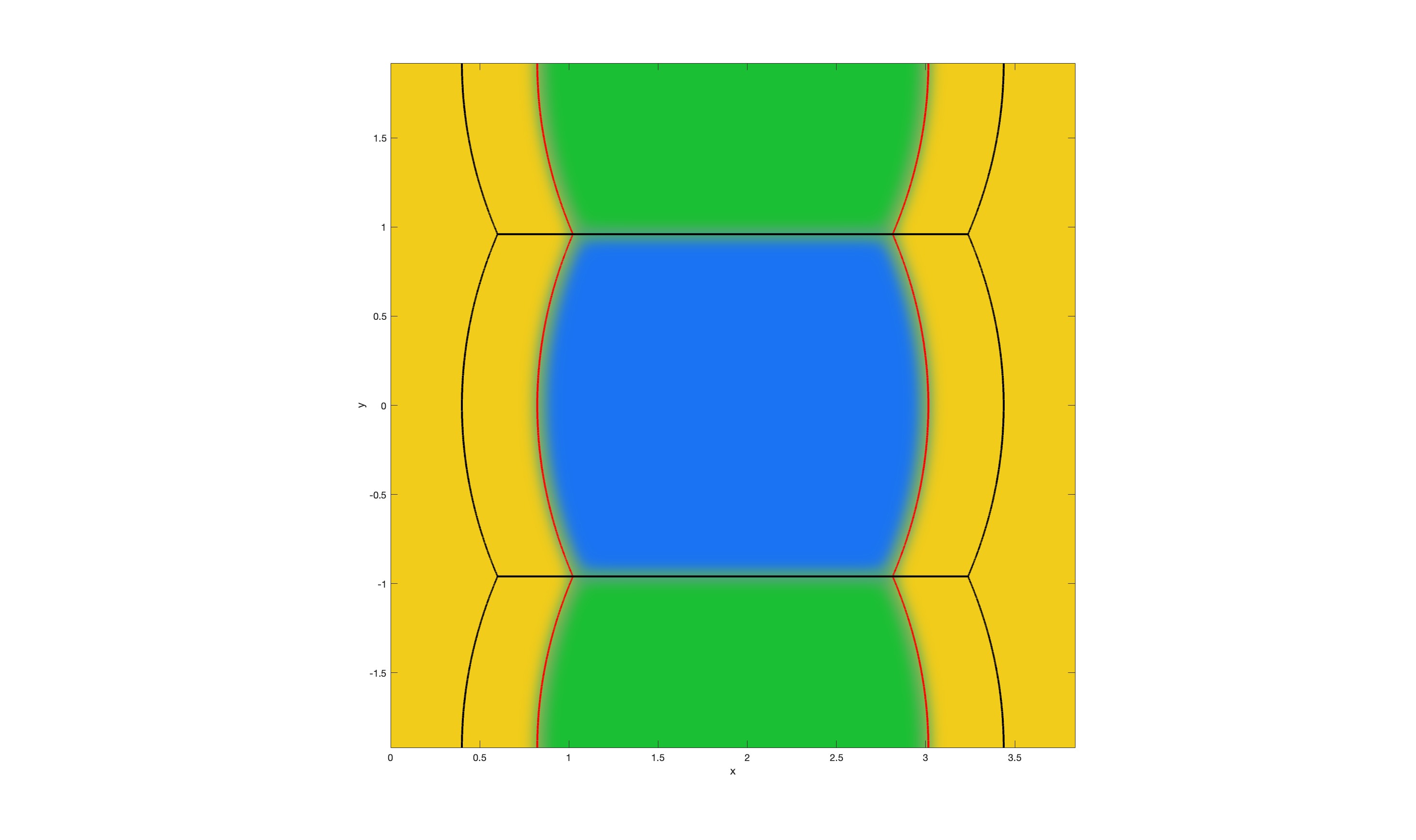}%
            \label{subfig:highmob2}%
        }\\
        \subfloat[$n=1024$]{%
            \includegraphics[width=.31\linewidth,
    trim=20cm 0cm 20cm 0cm,
    clip]{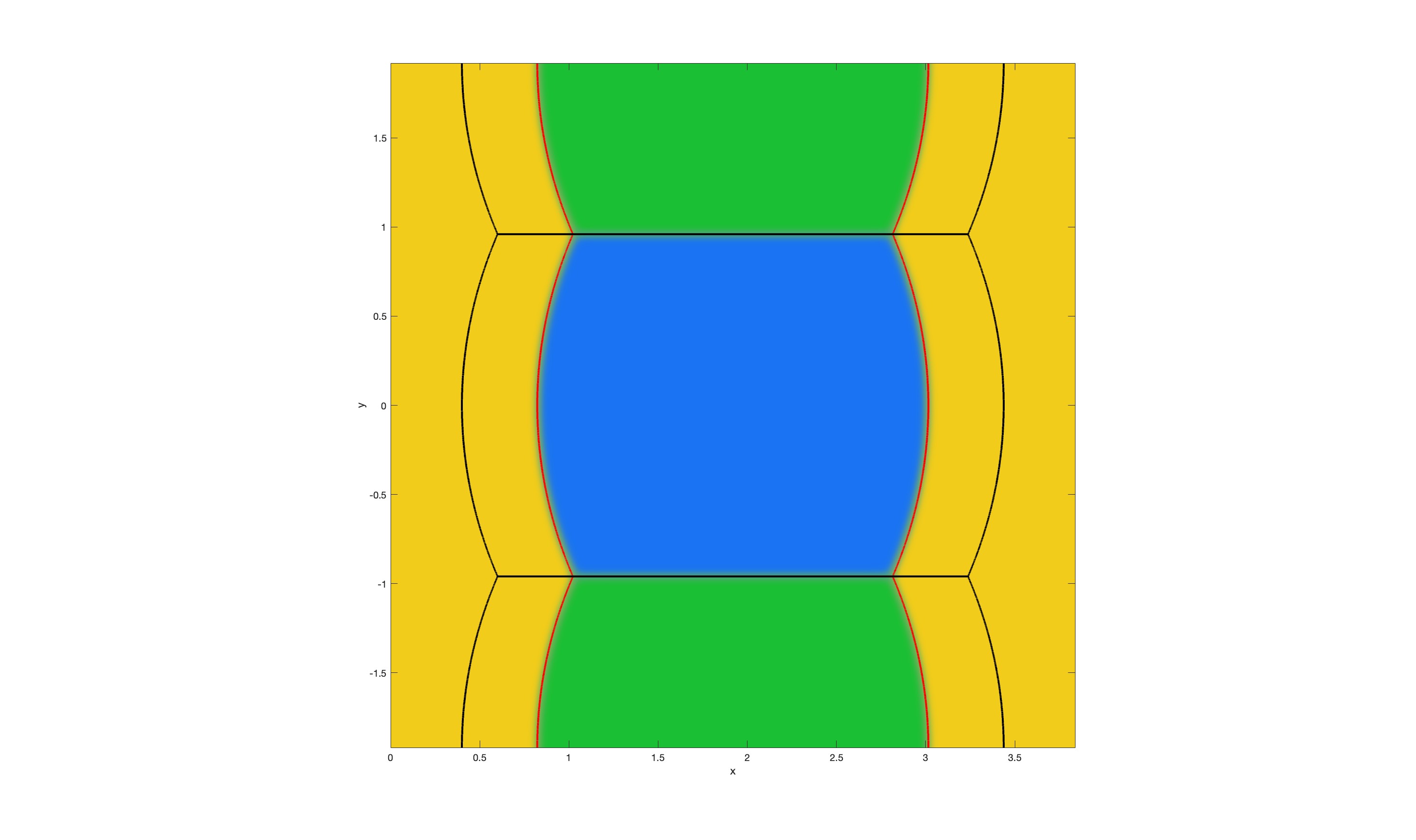}%
            \label{subfig:highmob3}%
        }\hspace{40pt}
        \subfloat[$n=2048$]{%
            \includegraphics[width=.31\linewidth,
    trim=20cm 0cm 20cm 0cm,
    clip]{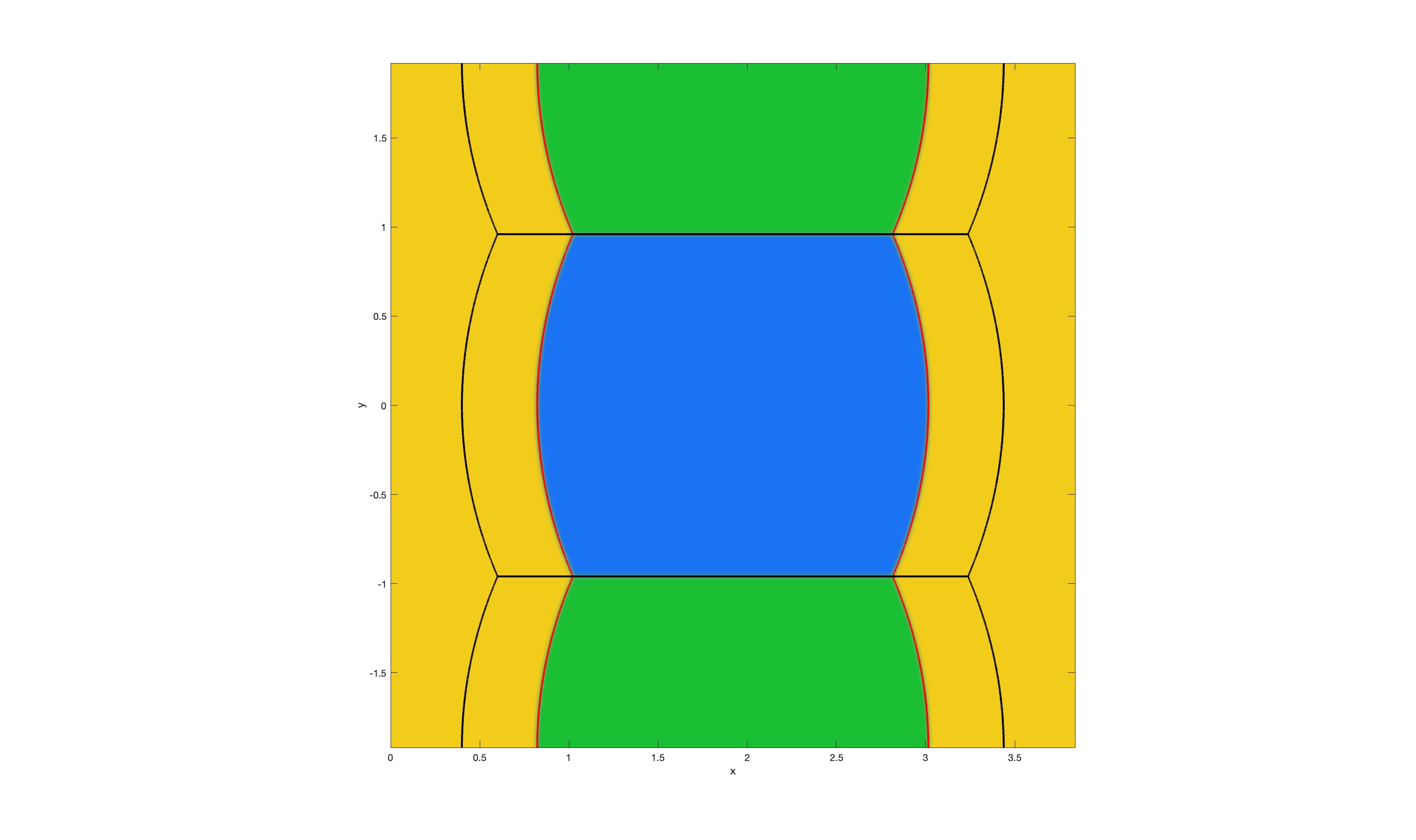}%
            \label{subfig:highmob4}%
        }\hfill
        \caption{Convergence study to a traveling wave exact solution with $m_{TJ} = 2$; in this exact solution, the junction angle is a constant $\theta = 66.7^{\circ}$ throughout the evolution. Final time $T=1$.}
        \label{fig:highmob}
\end{center}
\end{figure}

\begin{center}
\begin{minipage}[t]{0.54\textwidth}
    \centering
    \vspace{30pt}
    \def\arraystretch{1.5}
    \small
    \resizebox{\linewidth}{!}{
        \begin{tabular}{|l|l|l|l|l|}
        \hline
        $\delta x$ & $\eps$ &  $\delta t$            & Error  & Order \\ \hline
        3.837/(256-1) = 0.0150    & 0.0902  &  $4.52 \times 10^{-5}$ & 0.931  & -     \\ \hline
        3.837/(512-1) = 0.00750    & 0.0450  &  $1.12 \times 10^{-5}$ & 0.472  & 0.978 \\ \hline
        3.837/(1024-1) = 0.00375    & 0.0225 &  $2.81 \times 10^{-6}$ & 0.238  & 0.988  \\ \hline
        3.837/(2048-1) = 0.00187    & 0.0112 &  $7.02\times 10^{-7}$  & 0.119 & 0.999 \\ \hline
        \end{tabular}
        }
        \vspace{30pt}
        \captionof{table}{Error and order of convergence for the numerical test shown in Figure \ref{fig:highmob}, using the proposed method \eqref{eq:dragphasefield} \& \eqref{eq:mobility}.}
        \label{tab:table_highmob}
    \end{minipage}
    \hfill
    \begin{minipage}[t]{0.44\textwidth}
        \centering
        % \vspace{0pt}
        \begin{figure}[H]
        \centering
        \includegraphics[width=1\linewidth]{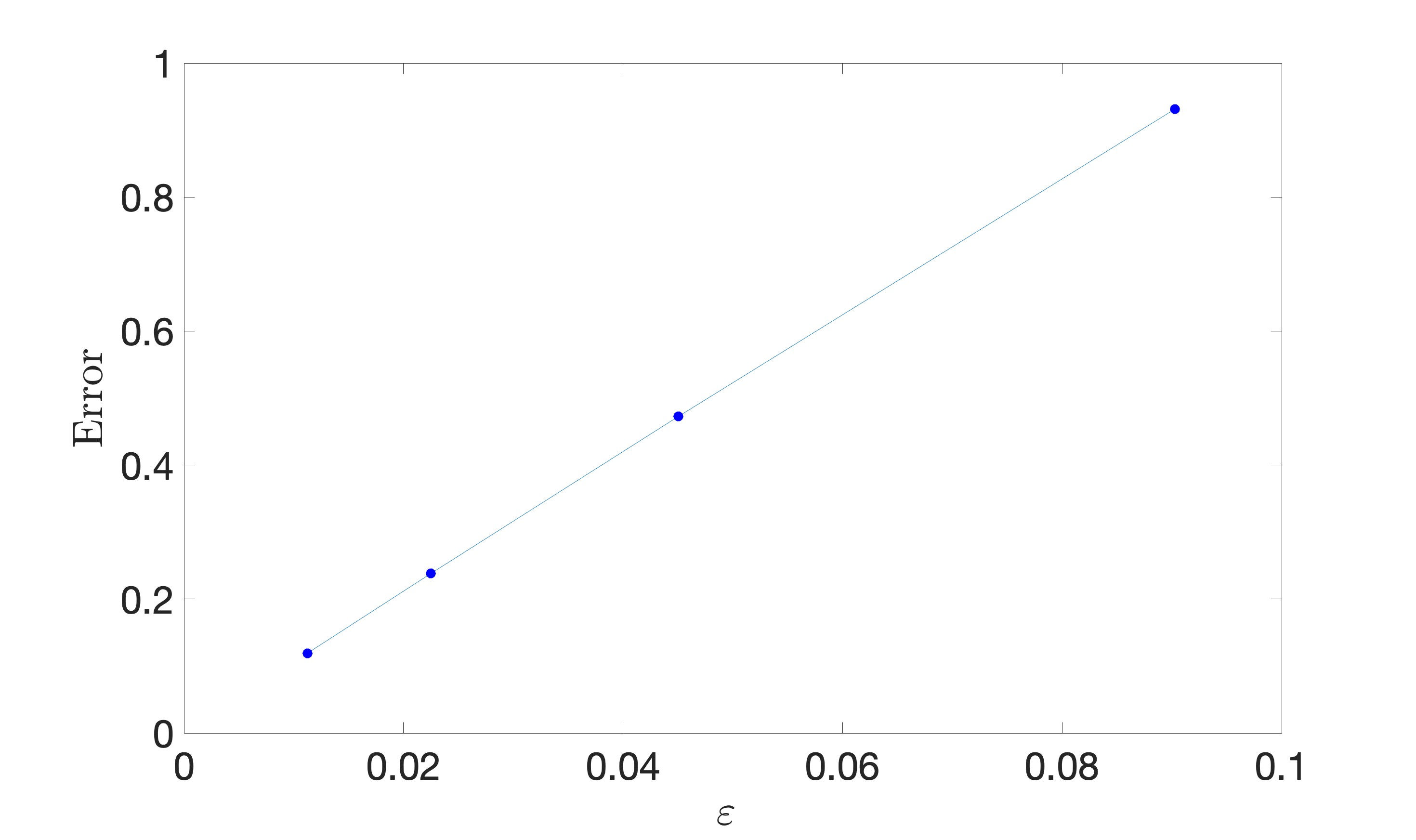}%
        \caption{Behavior of the error in the numerical solution generated by the proposed method \eqref{eq:dragphasefield} \& \eqref{eq:mobility}, on the test case of Figure \ref{fig:highmob}.}
        \label{fig:error_highmob}
        \end{figure}
    \end{minipage}
\end{center}

\clearpage

\subsection{Experiment 2: Translating solution with $m_{TJ} = 0.1$}
In the second experiment, we use a similar traveling wave solution but with a lower triple junction mobility ($m_{TJ} = 0.1$). The interfaces here are less curved than in the previous experiment.
\begin{figure}[H]
\begin{center}
        \subfloat[$n=256$]{%
            \includegraphics[width=.31\linewidth,
    trim=20cm 0cm 20cm 0cm,
    clip]{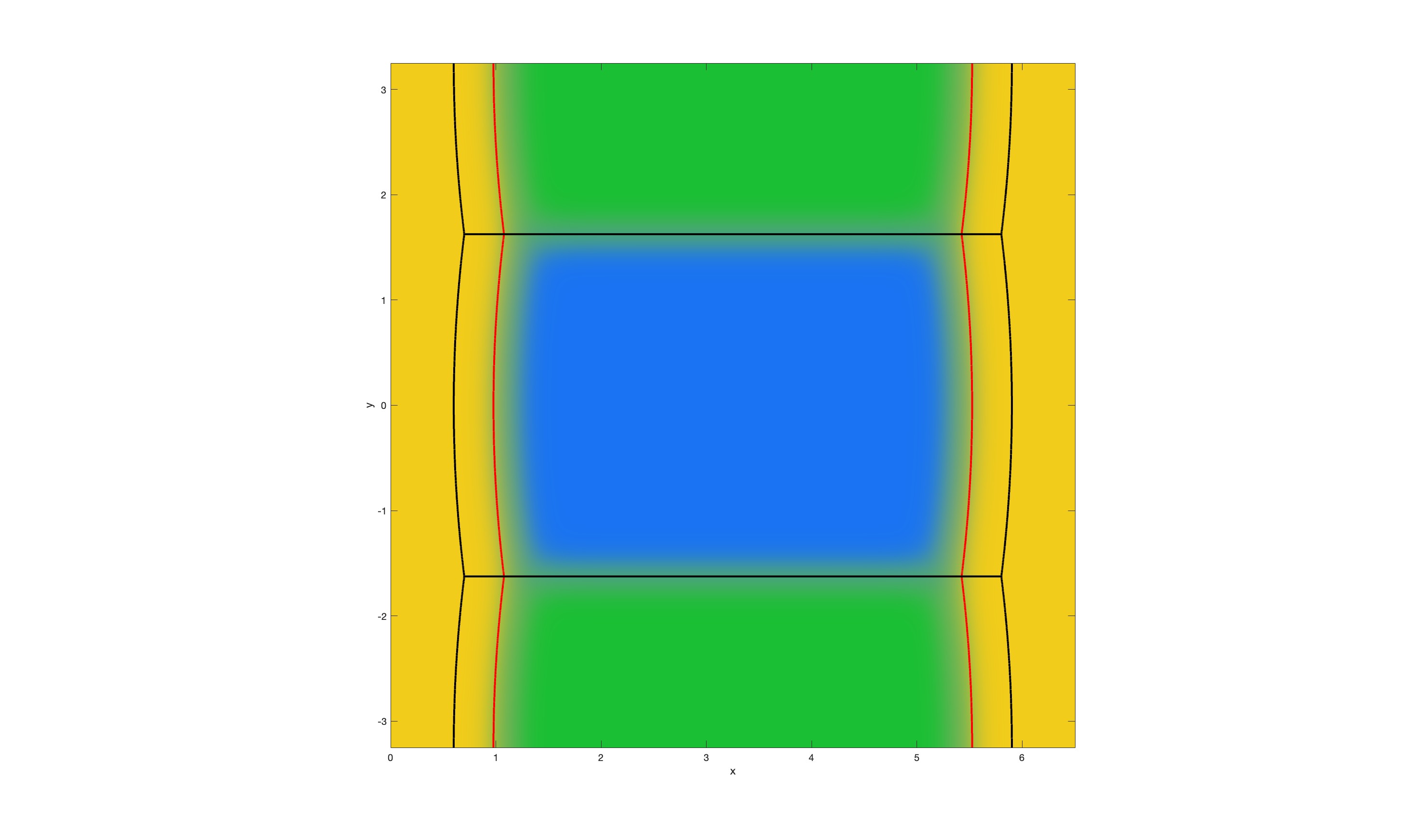}%
            \label{subfig:lowmob1}%
        }\hspace{40pt}
        \subfloat[$n=512$]{%
            \includegraphics[width=.31\linewidth,
    trim=20cm 0cm 20cm 0cm,
    clip]{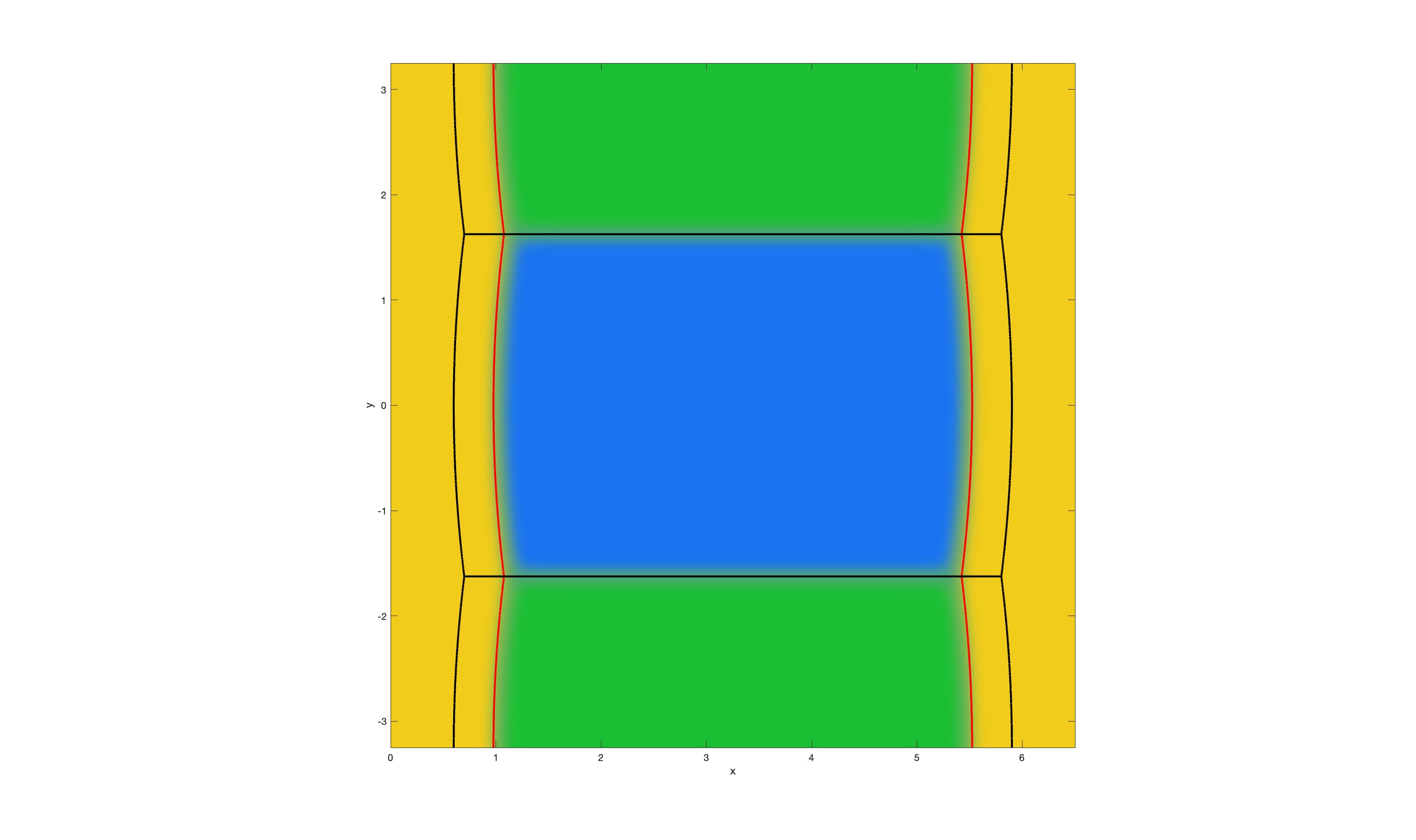}%
            \label{subfig:lowmob2}%
        }\\
        \subfloat[$n=1024$]{%
            \includegraphics[width=.31\linewidth,
    trim=20cm 0cm 20cm 0cm,
    clip]{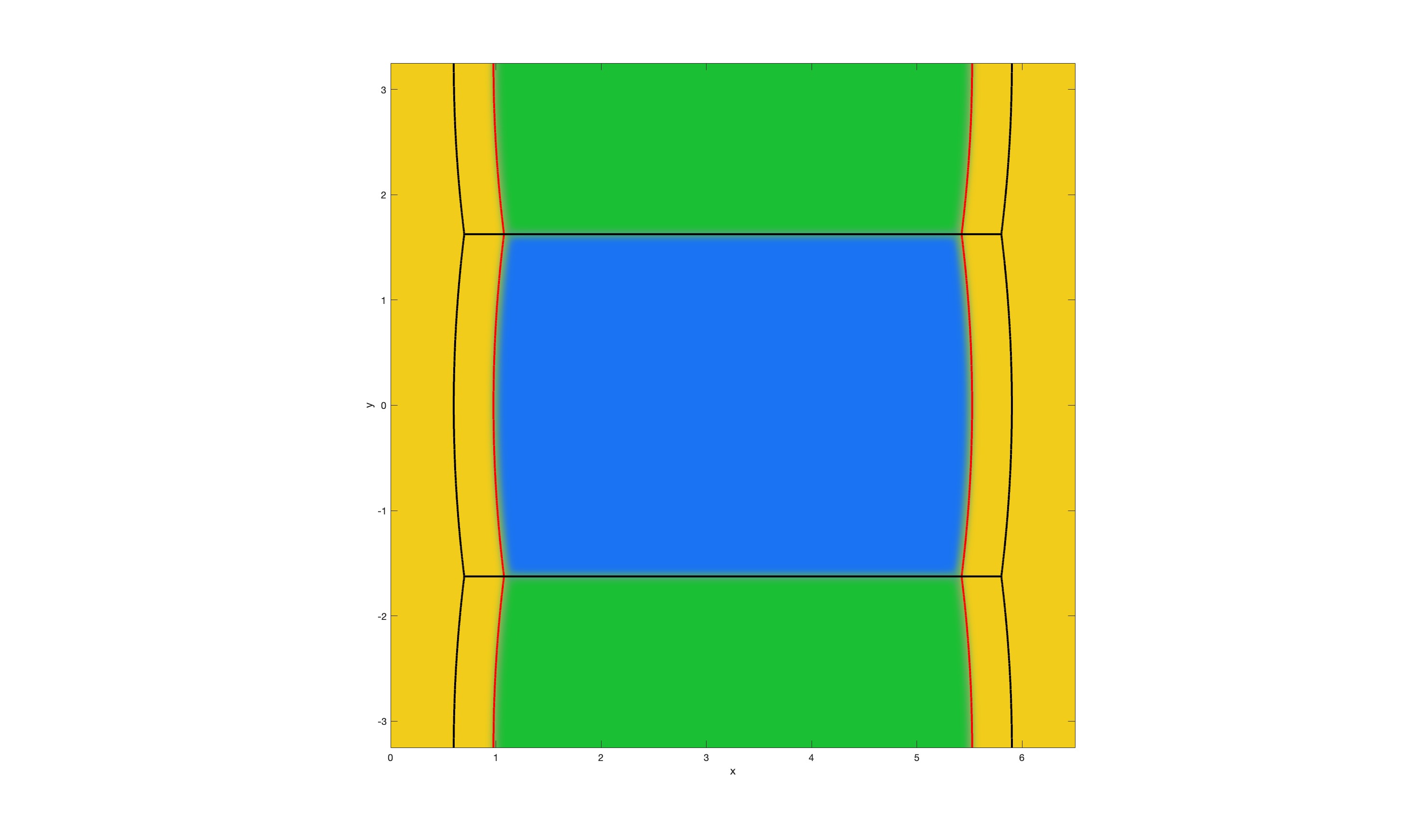}%
            \label{subfig:lowmob3}%
        }\hspace{40pt}
        \subfloat[$n=2048$]{%
            \includegraphics[width=.31\linewidth,
    trim=20cm 0cm 20cm 0cm,
    clip]{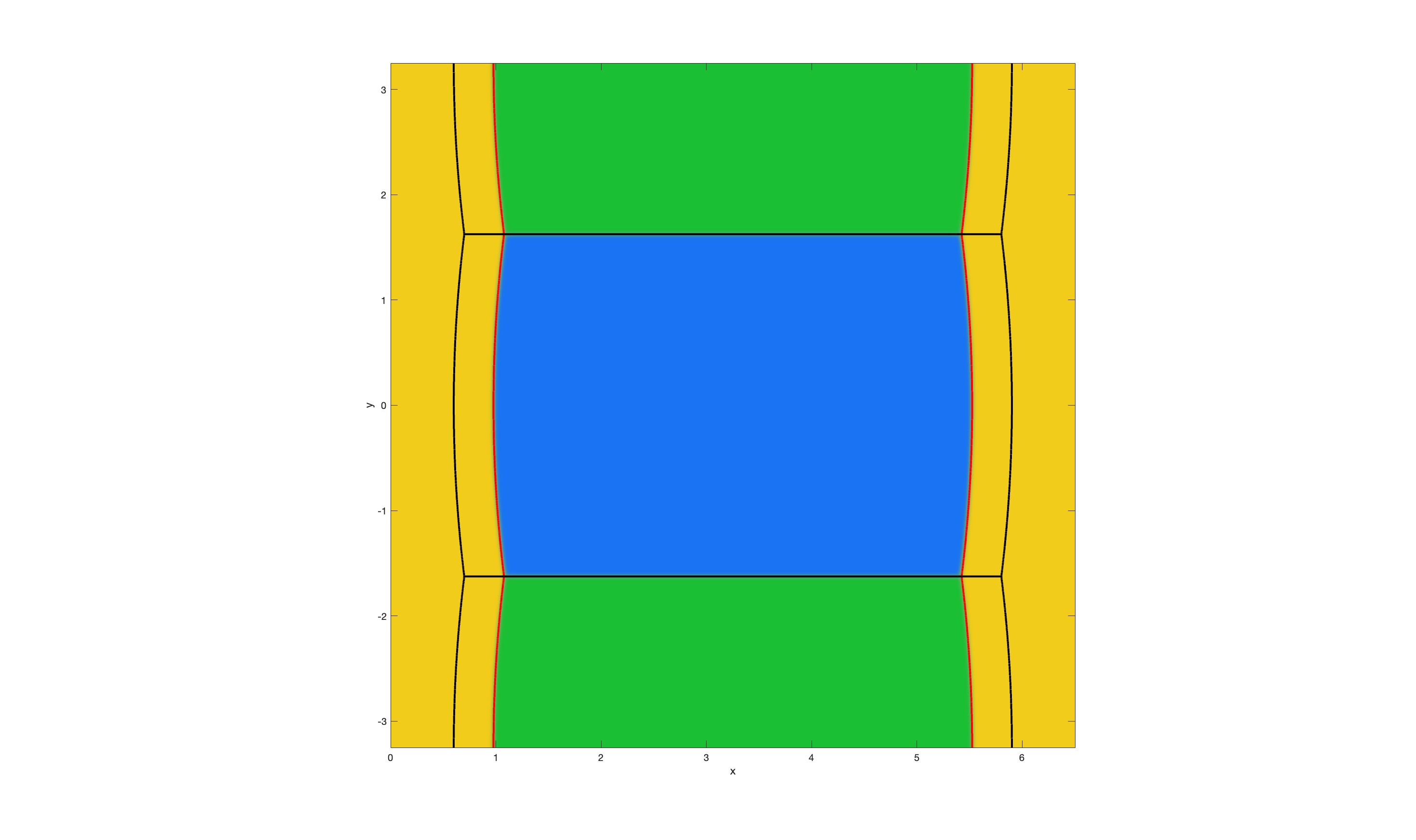}%
            \label{subfig:lowmob4}%
        }\hfill
        \caption{Convergence study to a traveling wave exact solution with $m_{TJ} = 0.1$; in this exact solution, the junction angle is a constant $\theta = 83^{\circ}$ throughout the evolution. Final time $T=5$.}
        \label{fig:lowmob}
        \end{center}
\end{figure}

\begin{center}
\begin{minipage}[t]{0.54\textwidth}
    \centering
    \vspace{30pt}
    \def\arraystretch{1.5}
    \small
    \resizebox{\linewidth}{!}{
        \begin{tabular}{|l|l|l|l|l|}
        \hline
        $\delta x$ & $\eps$ &  $\delta t$            & Error  & Order \\ \hline
        6.50/(256-1) = 0.0254    & 0.152  &  $1.30 \times 10^{-4}$ & 2.938  & -     \\ \hline
        6.50/(512-1) = 0.0127    & 0.0763  &  $3.23 \times 10^{-5}$ & 1.506  & 0.964 \\ \hline
        6.50/(1024-1) = 0.00635    & 0.0381 &  $8.07 \times 10^{-6}$ & 0.762  & 0.983  \\ \hline
        6.50/(2048-1) = 0.00317    & 0.0190 &  $2.01\times 10^{-6}$  & 0.381 & 1.000 \\ \hline
        \end{tabular}
        }
        \vspace{30pt}
        \captionof{table}{Error and order of convergence for the numerical test shown in Figure \ref{fig:lowmob}, using the proposed method \eqref{eq:dragphasefield} \& \eqref{eq:mobility}.}
        \label{tab:table_lowmob}
    \end{minipage}
    \hfill
    \begin{minipage}[t]{0.44\textwidth}
        \centering
        % \vspace{0pt}
        \begin{figure}[H]
        \centering
        \includegraphics[width=1\linewidth]{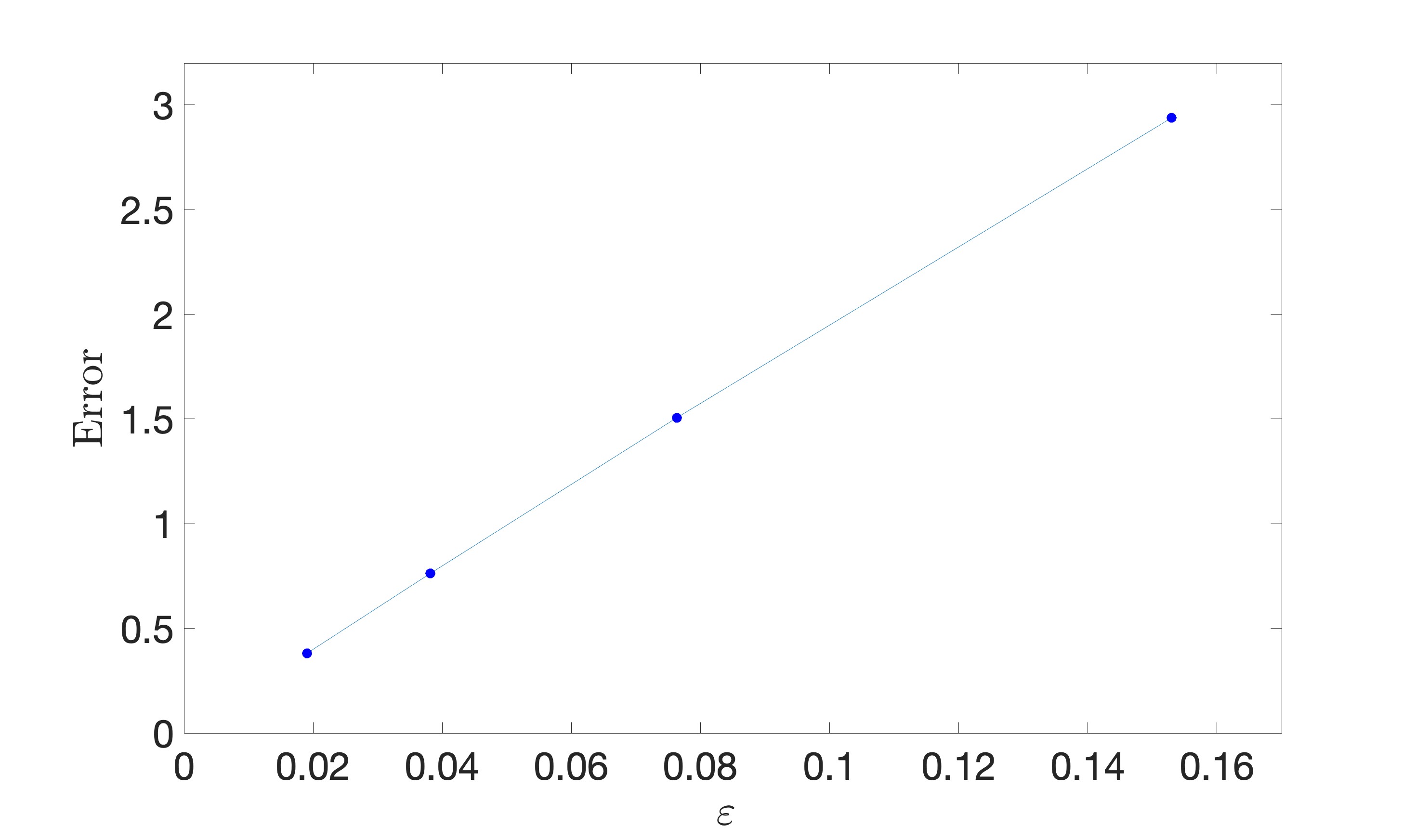}%
        %\caption{Plot of error against $\eps$}
        \caption{Behavior of the error in the numerical solution generated by the proposed method \eqref{eq:dragphasefield} \& \eqref{eq:mobility}, on the test case of Figure \ref{fig:lowmob}.}
        \label{fig:error_lowmob}
        \end{figure}
    \end{minipage}
\end{center}

\clearpage

\subsection{Experiment 3: Dynamic angles}

In this experiment, we test our phase-field method \eqref{eq:dragphasefield} \& \eqref{eq:mobility} on a configuration where the angles at the triple junction change over time. Similar to the previous experiments, the black network of curves is the initial condition. Here, the benchmark network of curves at final time, shown again in red, was computed using a front tracking method (explicit parametrization of the curves) with very fine space and time discretization.

\begin{figure}[H]
\begin{center}
        \subfloat[$n=256$]{%
            \includegraphics[width=.31\linewidth,
    trim=20cm 0cm 20cm 0cm,
    clip]{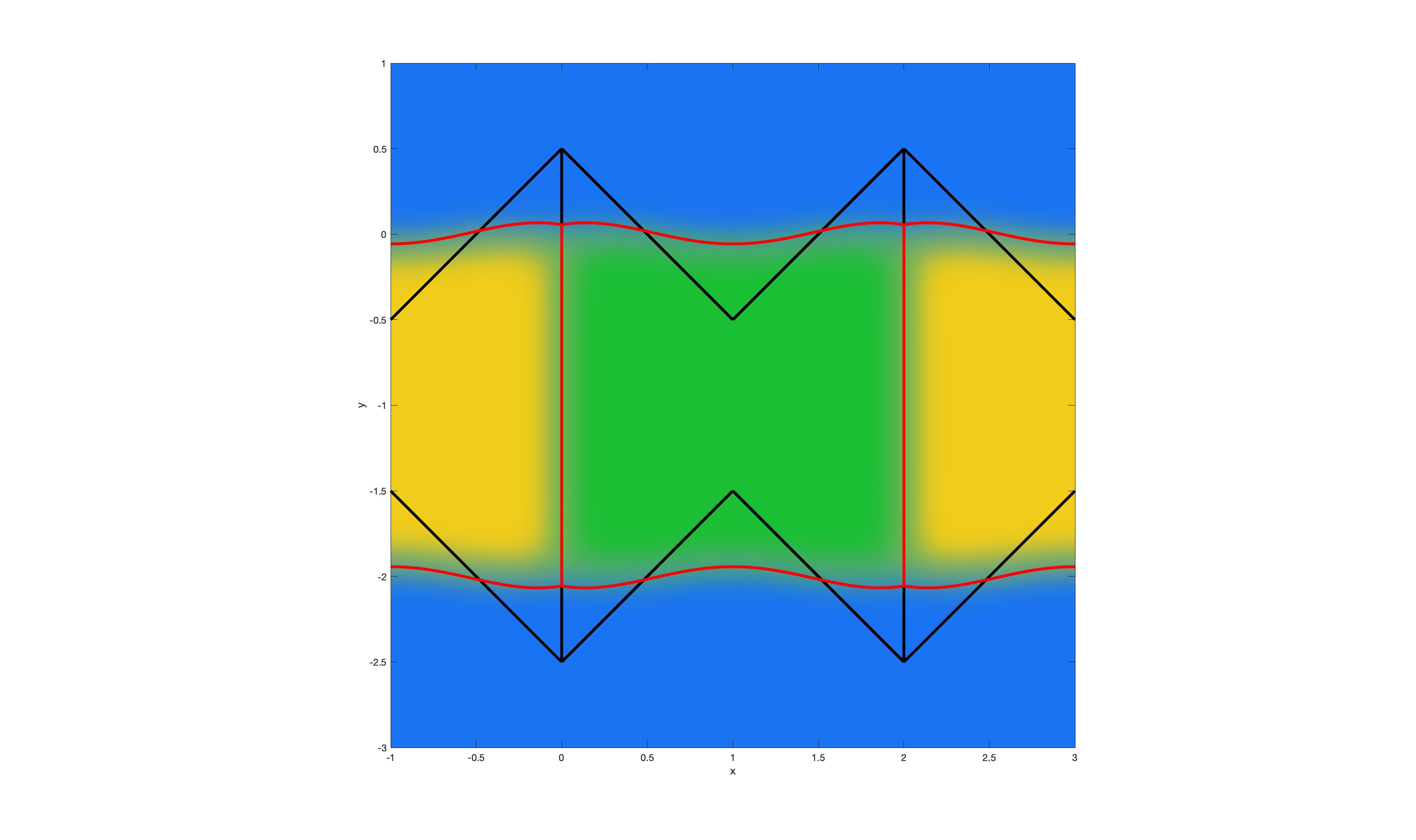}%
            \label{subfig:fork1}%
        }\hspace{40pt}
        \subfloat[$n=512$]{%
            \includegraphics[width=.31\linewidth,
    trim=20cm 0cm 20cm 0cm,
    clip]{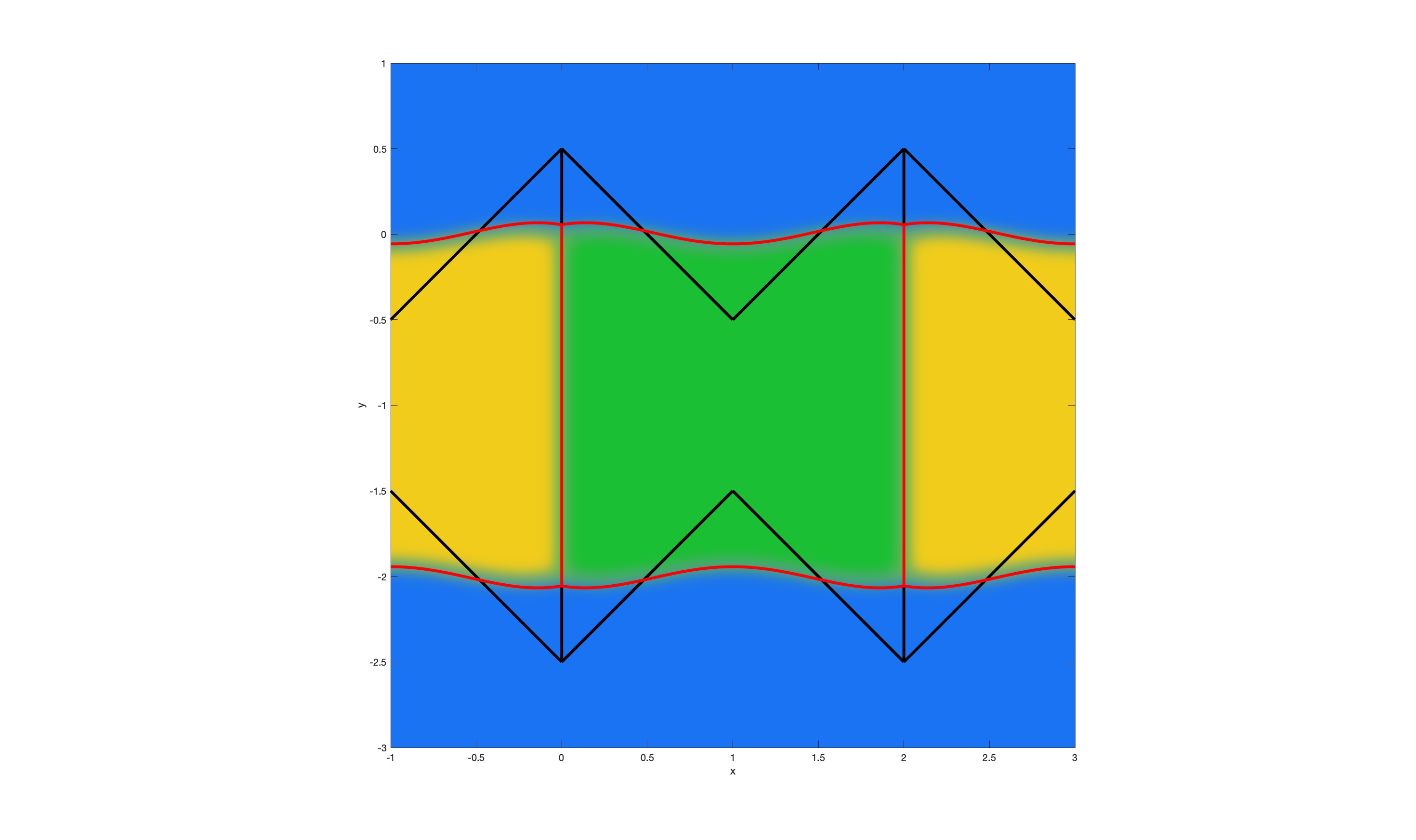}%
            \label{subfig:fork2}%
        }\\
        \subfloat[$n=1024$]{%
            \includegraphics[width=.31\linewidth,
    trim=20cm 0cm 20cm 0cm,
    clip]{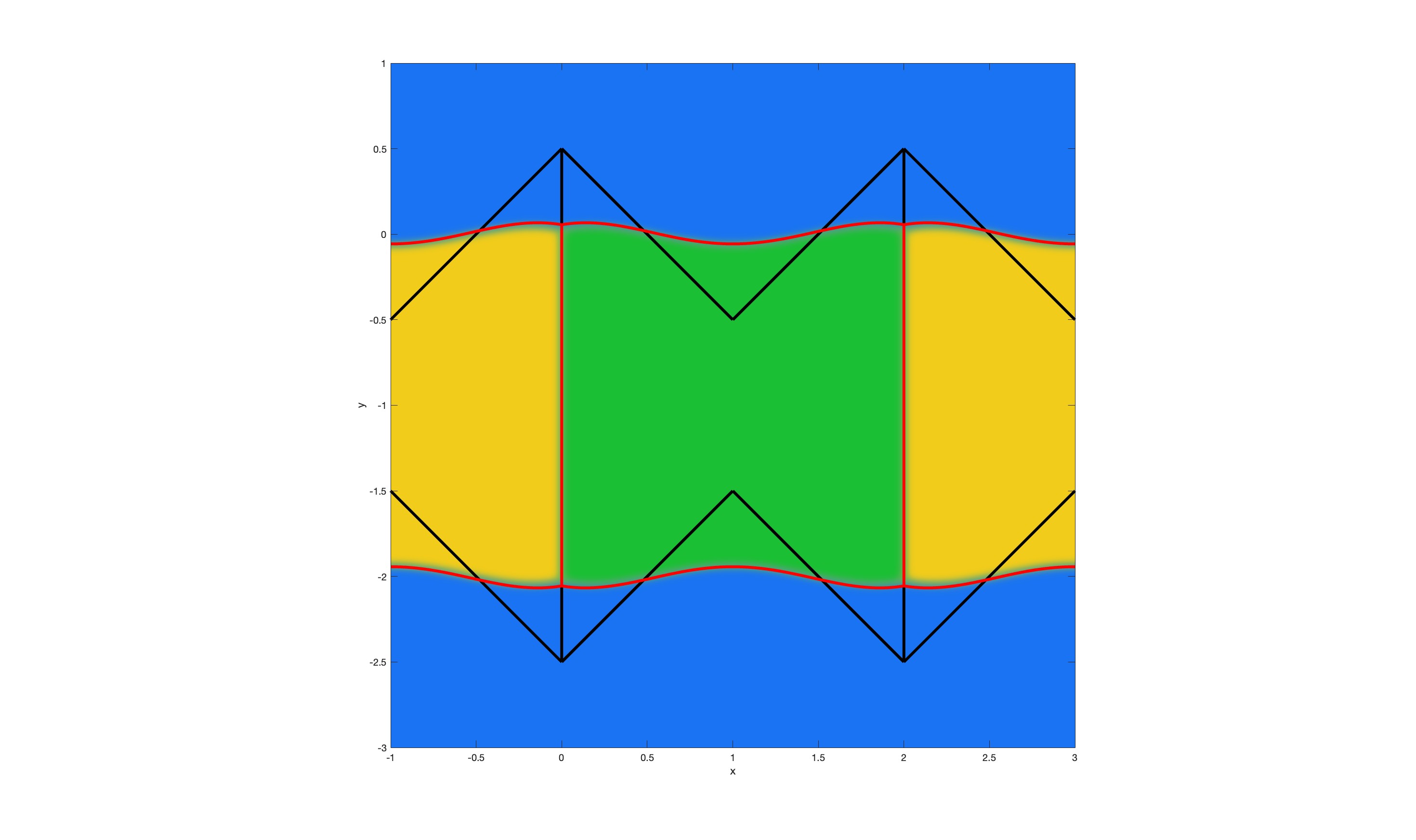}%
            \label{subfig:fork3}%
        }\hspace{40pt}
        \subfloat[$n=2048$]{%
            \includegraphics[width=.31\linewidth,
    trim=20cm 0cm 20cm 0cm,
    clip]{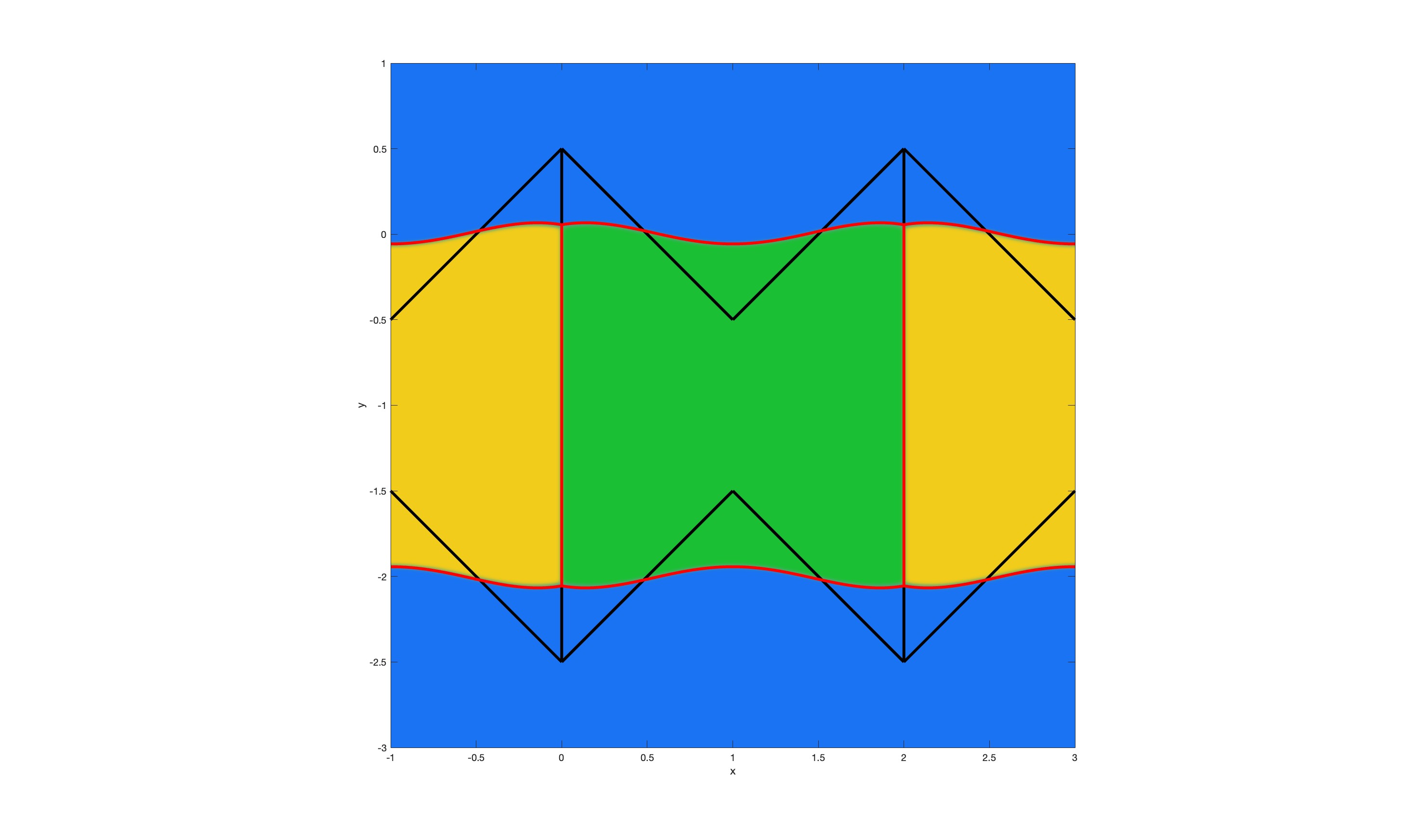}%
            \label{subfig:fork4}%
        }\hfill
        \caption{Convergence study to a configuration where the angles at the triple junction change dramatically over time. Benchmark curves (red) were obtained via a front tracking method. $m_{TJ} = 2$, final time $T=0.2$.}
        \label{fig:fork}
\end{center}
\end{figure}

\begin{center}
\begin{minipage}[t]{0.54\textwidth}
    \centering
    \vspace{30pt}
    \def\arraystretch{1.5}
    \small
    \resizebox{\linewidth}{!}{
        \begin{tabular}{|l|l|l|l|l|}
        \hline
        $\delta x$ & $\eps$ &  $\delta t$            & Error  & Order \\ \hline
        4/(256-1) = 0.0156    & 0.0941  &  $4.92 \times 10^{-5}$ & 1.010  & -     \\ \hline
        4/(512-1) = 0.00782    & 0.0469  &  $1.22 \times 10^{-5}$ & 0.520  & 0.956 \\ \hline
        4/(1024-1) = 0.00391    & 0.0234 &  $3.05 \times 10^{-6}$ & 0.264  & 0.977  \\ \hline
        4/(2048-1) = 0.00195    & 0.0117 &  $7.63\times 10^{-7}$  & 0.132 & 0.994 \\ \hline
        \end{tabular}
        }
        \vspace{30pt}
        \captionof{table}{Error and order of convergence for the numerical test shown in Figure \ref{fig:fork}, using the proposed method \eqref{eq:dragphasefield} \& \eqref{eq:mobility}.}
        \label{tab:table_fork}
    \end{minipage}
    \hfill
    \begin{minipage}[t]{0.44\textwidth}
        \centering
        \begin{figure}[H]
        \centering
        \includegraphics[width=1\linewidth]{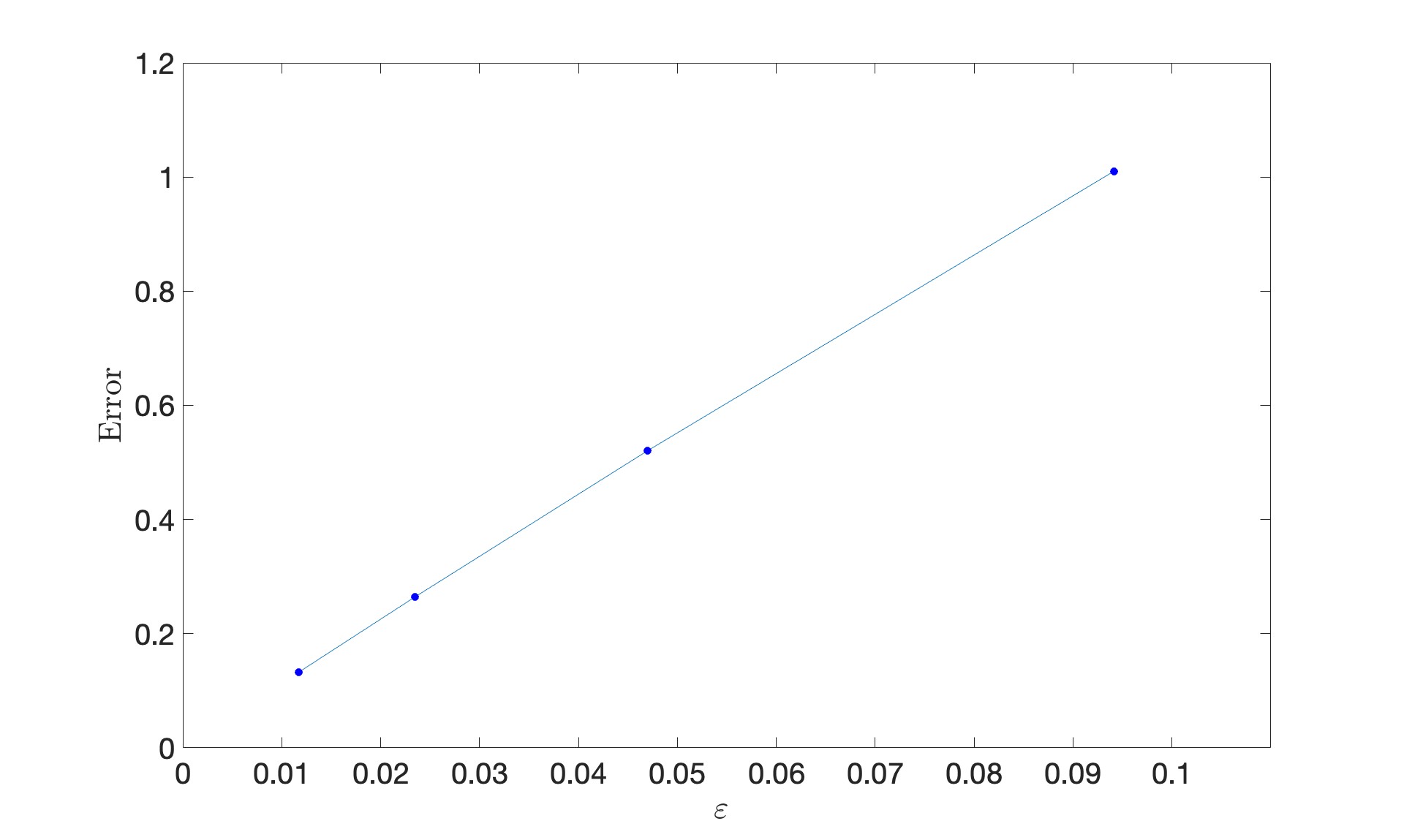}%
            \caption{Behavior of the error in the numerical solution generated by the proposed method \eqref{eq:dragphasefield} \& \eqref{eq:mobility}, on the test case of Figure \ref{fig:fork}.}
        \label{fig:error_fork}
        \end{figure}
    \end{minipage}
\end{center}

\clearpage

\subsection{Experiment 4: Asymmetric, dynamic angles}
In this final convergence study, we put our phase-field method \eqref{eq:dragphasefield} \& \eqref{eq:mobility} to an even more challenging test: a configuration where the angles at the triple junction are changing over time, and the triple junction does not follow a straight path (nor is its motion aligned with one of the interfaces,
unlike in the previous experiments). 
\begin{figure}[H]
\begin{center}
        \subfloat[$n=256$]{%
            \includegraphics[width=.31\linewidth,
    trim=20cm 0cm 20cm 0cm,
    clip]{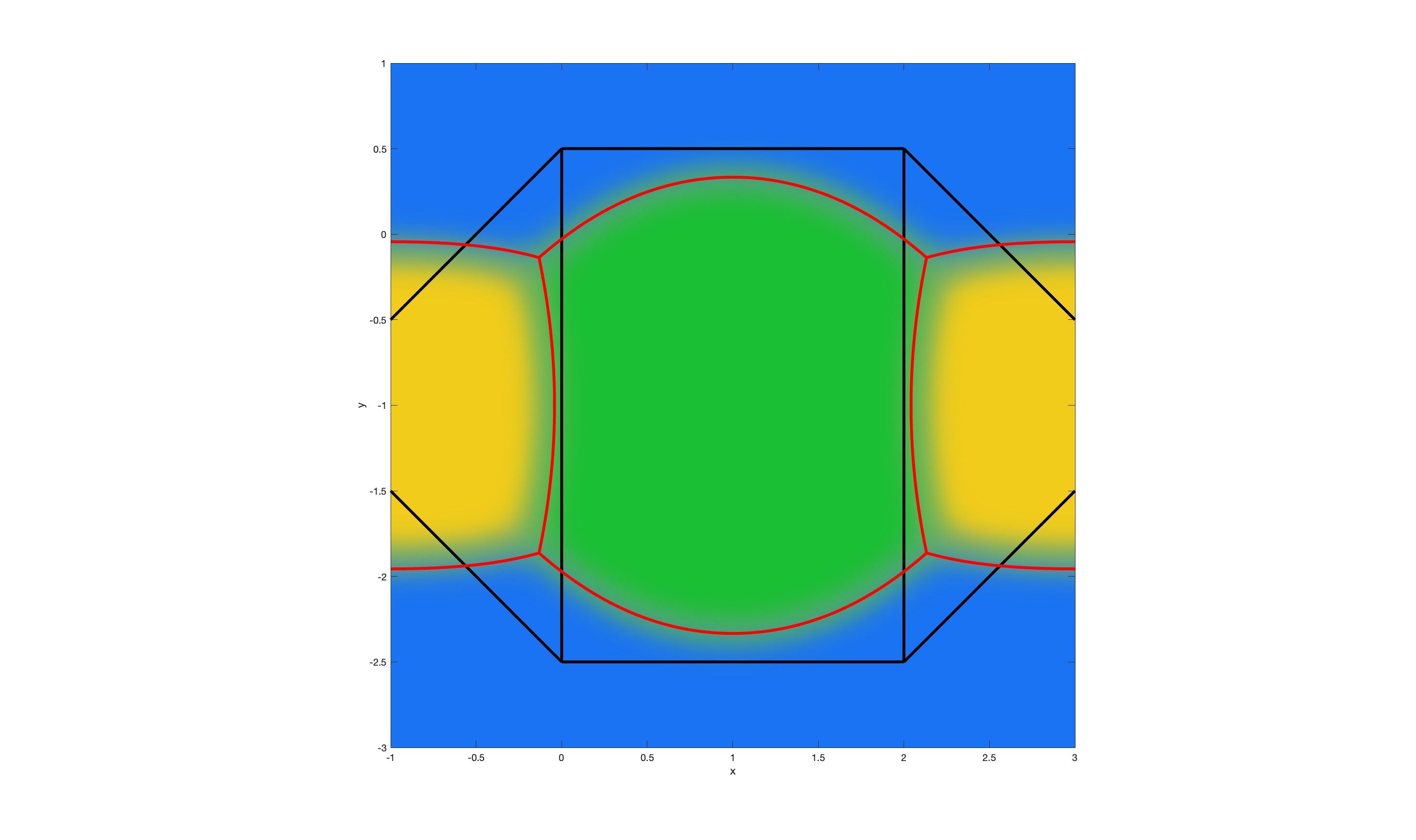}%
            \label{subfig:forkasym1}%
        }\hspace{40pt}
        \subfloat[$n=512$]{%
            \includegraphics[width=.31\linewidth,
    trim=20cm 0cm 20cm 0cm,
    clip]{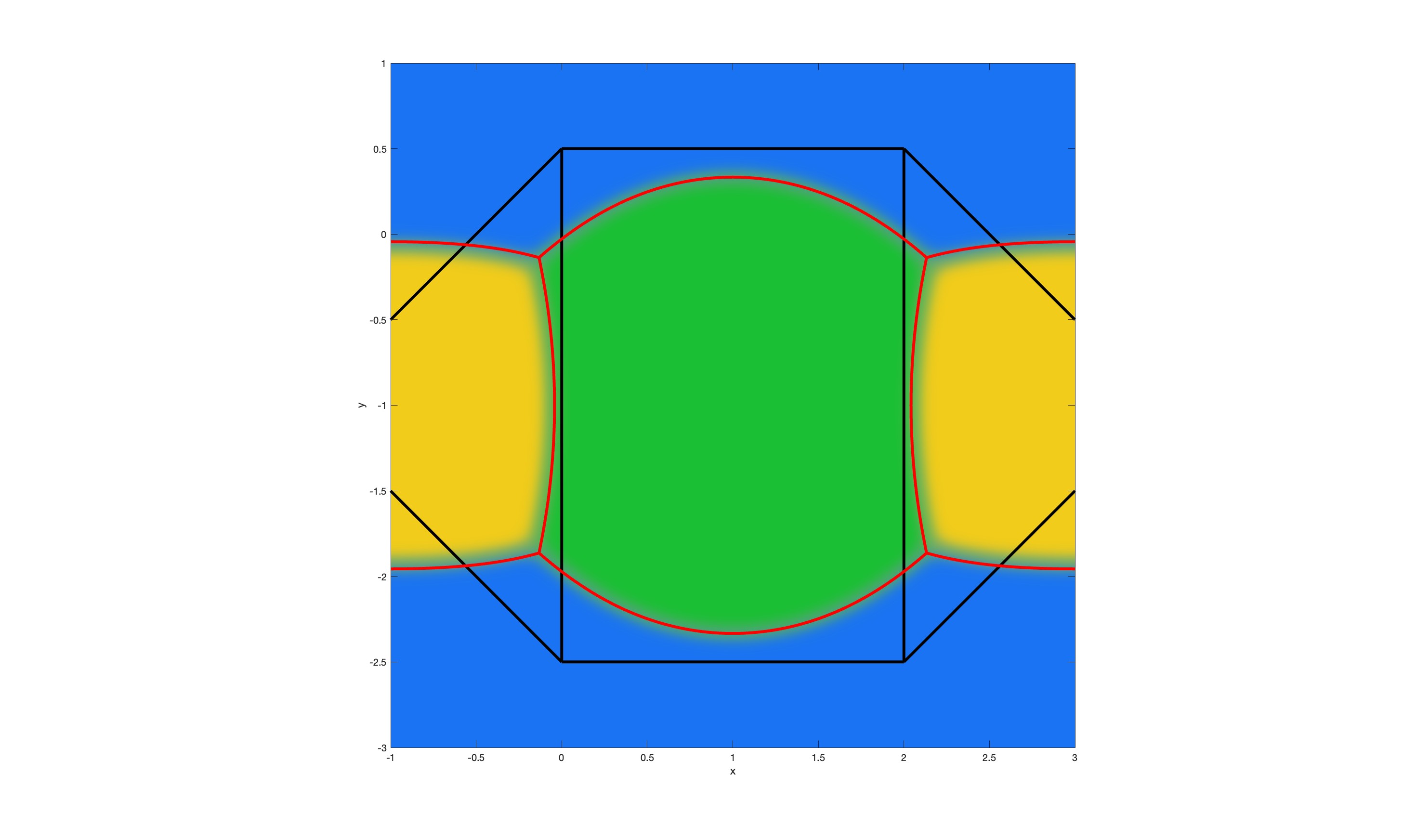}%
            \label{subfig:forkasym2}%
        }\\
        \subfloat[$n=1024$]{%
            \includegraphics[width=.31\linewidth,
    trim=20cm 0cm 20cm 0cm,
    clip]{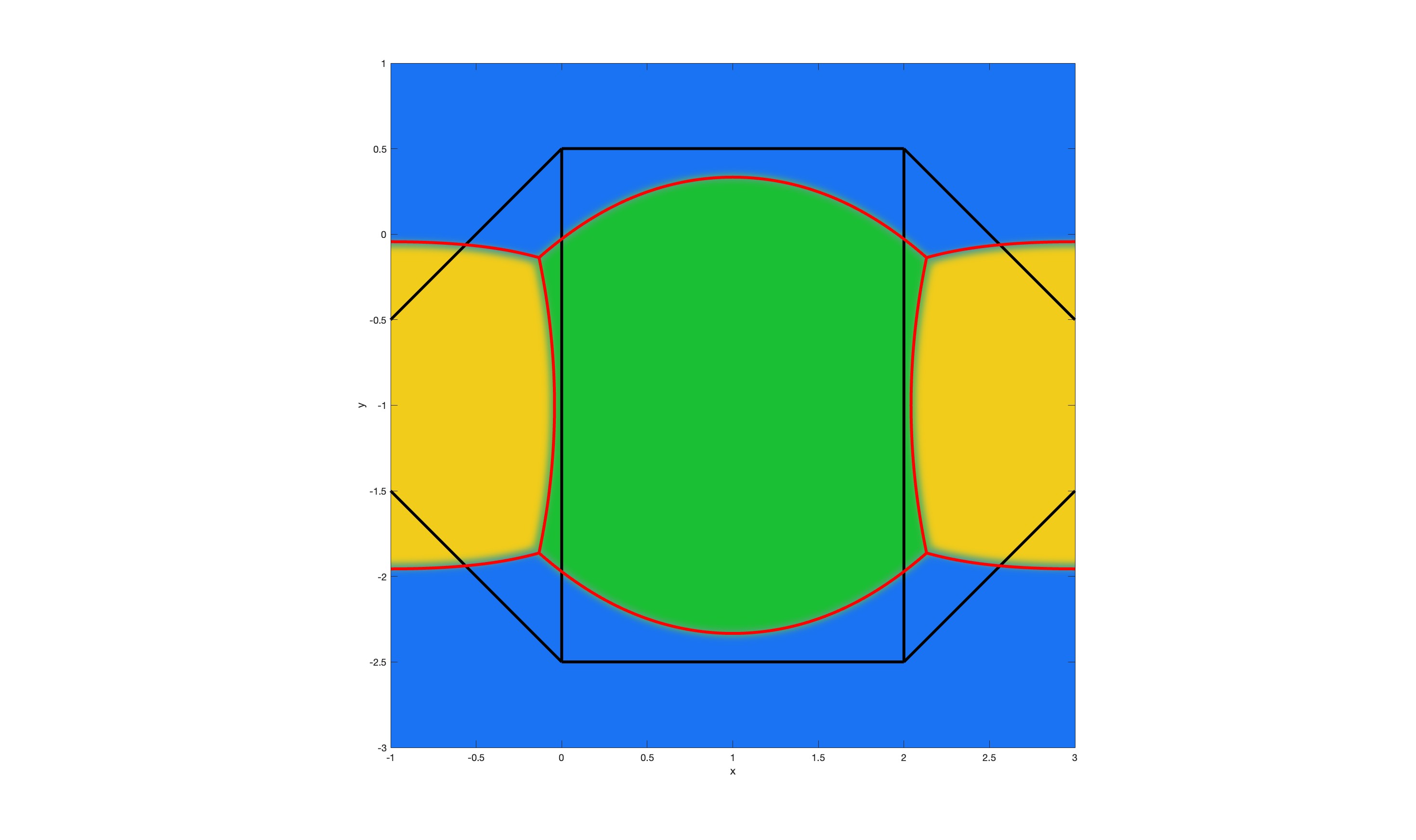}%
            \label{subfig:forkasym3}%
        }\hspace{40pt}
        \subfloat[$n=2048$]{%
            \includegraphics[width=.31\linewidth,
    trim=20cm 0cm 20cm 0cm,
    clip]{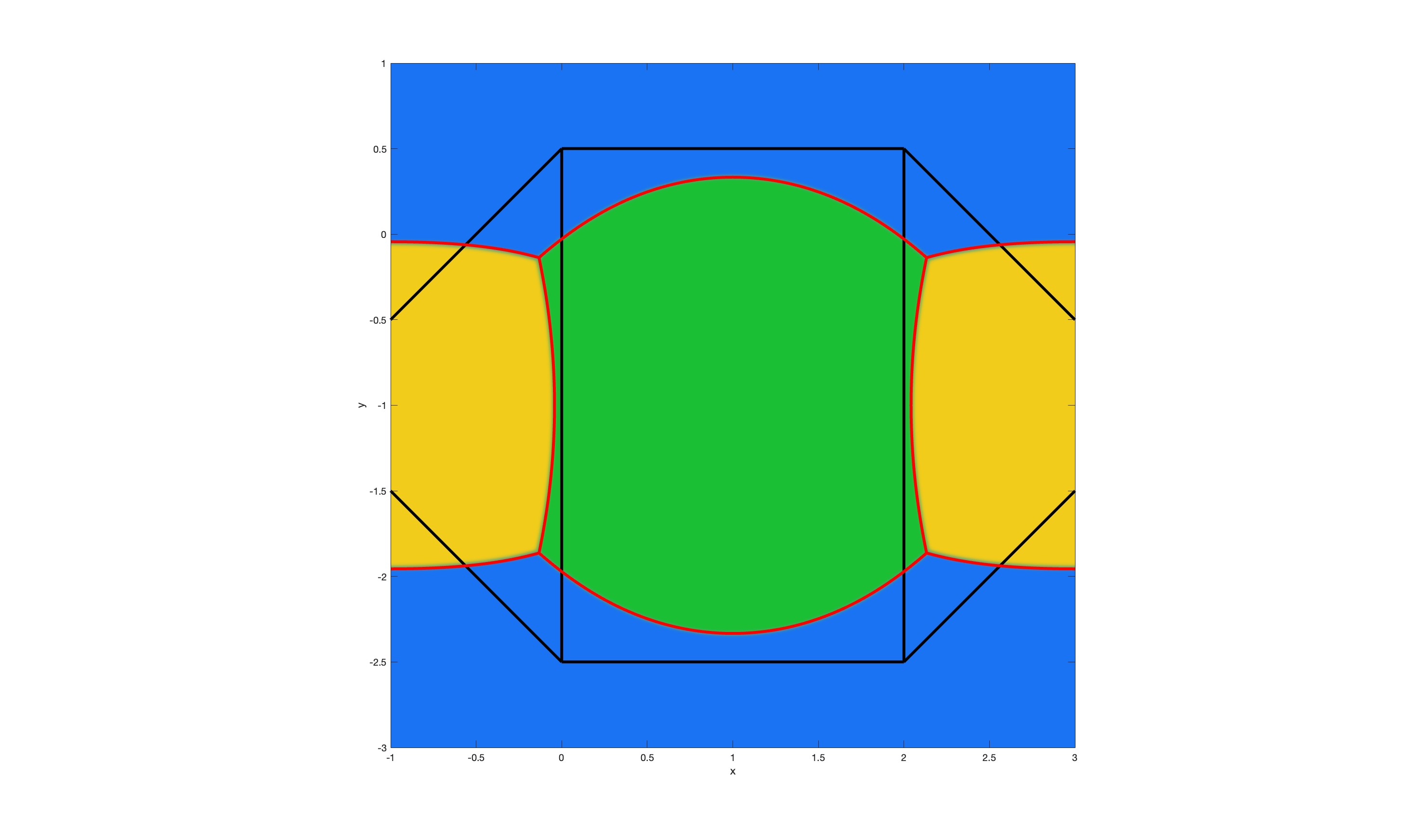}%
            \label{subfig:forkasym4}%
        }\hfill
        \caption{Convergence study to a configuration where the motion of the triple junction is not aligned with any of the interfaces. Benchmark curves (red) were obtained via a front tracking method. $m_{TJ} = 10$, final time $T=0.4$.}
        \label{fig:forkasym}
\end{center}
\end{figure}

\begin{center}
\begin{minipage}[t]{0.54\textwidth}
    \centering
    \vspace{30pt}
    \def\arraystretch{1.5}
    \small
    \resizebox{\linewidth}{!}{
        \begin{tabular}{|l|l|l|l|l|}
        \hline
        $\delta x$ & $\eps$ &  $\delta t$            & Error  & Order \\ \hline
        4/(256-1) = 0.0156    & 0.0941  &  $4.92 \times 10^{-5}$ & 1.018  & -     \\ \hline
        4/(512-1) = 0.00782    & 0.0469  &  $1.22 \times 10^{-5}$ & 0.515  & 0.981 \\ \hline
        4/(1024-1) = 0.00391    & 0.0234 &  $3.05 \times 10^{-6}$ & 0.259  & 0.990  \\ \hline
        4/(2048-1) = 0.00195    & 0.0117 &  $7.63\times 10^{-7}$  & 0.130 & 0.995 \\ \hline
        \end{tabular}
        }
        \vspace{30pt}
        \captionof{table}{Error and order of convergence for the numerical test shown in Figure \ref{fig:forkasym}, using the proposed method \eqref{eq:dragphasefield} \& \eqref{eq:mobility}.}
        \label{tab:table_forkasym}
    \end{minipage}
    \hfill
    \begin{minipage}[t]{0.44\textwidth}
        \centering
        \begin{figure}[H]
            \centering
            \includegraphics[width=1\linewidth]{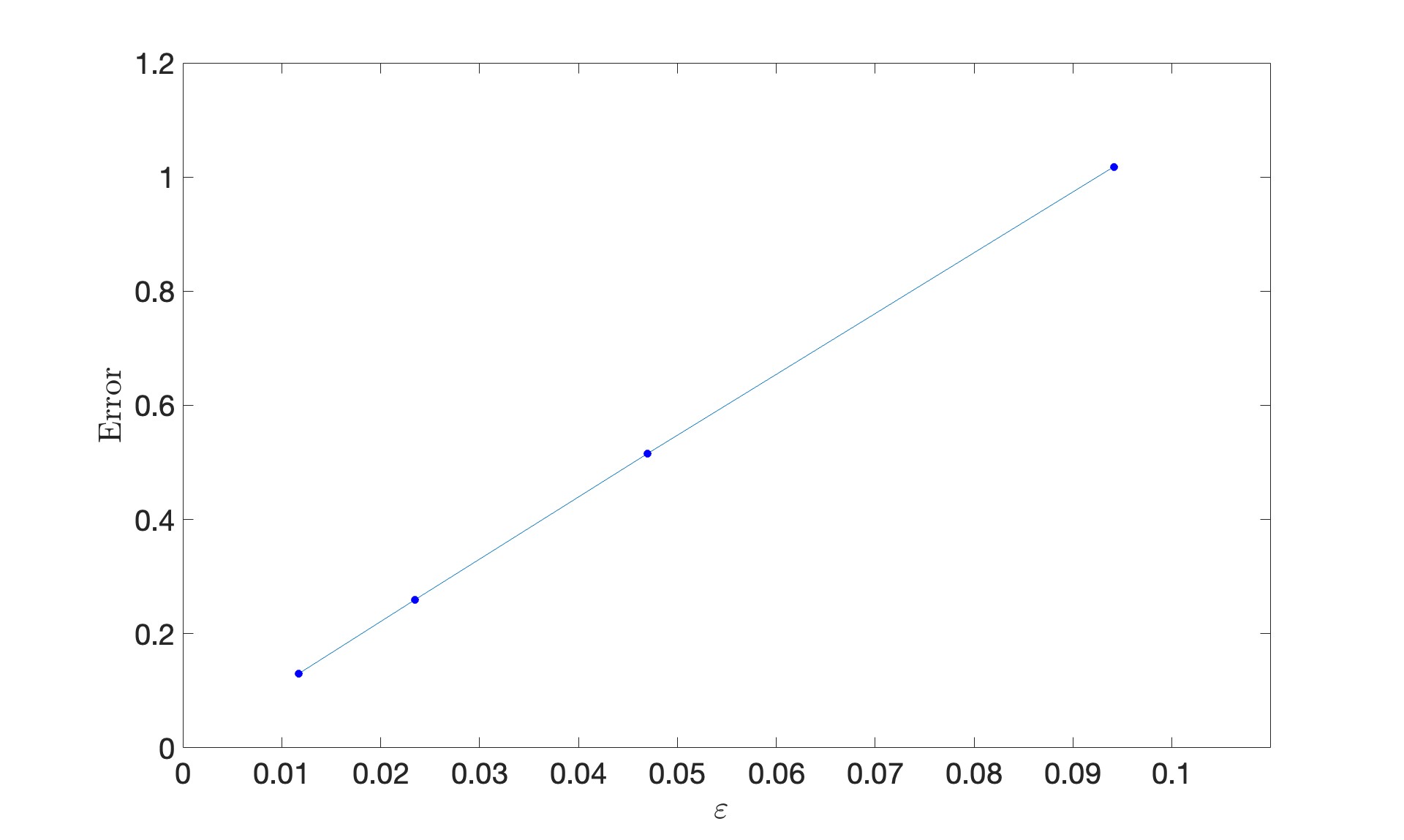}%
            \caption{Behavior of the error in the numerical solution generated by the proposed method \eqref{eq:dragphasefield} \& \eqref{eq:mobility}, on the test case of Figure \ref{fig:forkasym}.}
            \label{fig:error_forkasym}
        \end{figure}
    \end{minipage}
\end{center}

\clearpage

\subsection{The Jacobian determinant under triple junction drag}

In the following simulation, we demonstrate that the Jacobian determinant \eqref{eq:jacobian} of the solution $u(x,t)$ to \eqref{eq:dragphasefield} \& \eqref{eq:mobility} is a robust triple junction detector even in the presence of triple junction drag (i.e. when there are multiple triple junctions with different time-varying angle configurations). 
Moreover, this simulation verifies that our phase-field method can handle topological changes seamlessly, as expected.
Here, we choose $m_{TJ} = 5$. The evolution through multiple critical events are shown in Figures \ref{fig:drag1} through \ref{fig:drag3}.
Figure \ref{fig:jacobiandrag} shows that $N_{TJ}(t)$ as defined in \eqref{eq:ntj} is indeed approximately integer valued at non-critical times, providing clear evidence for the independence of \eqref{eq:jacobian} from junction geometry.

\begin{figure}[H]
    \centering        \includegraphics[width=.7\linewidth]{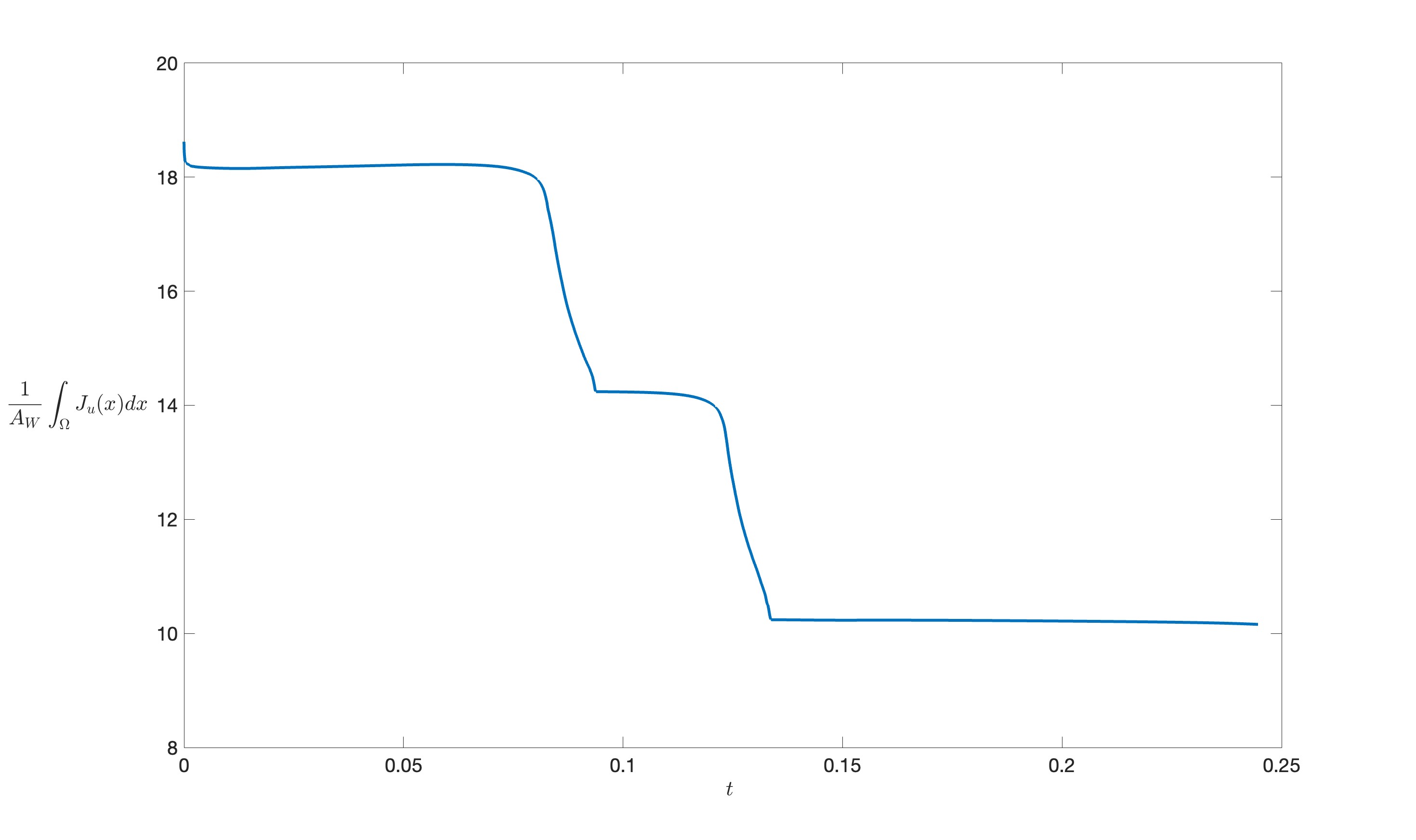}%
        \caption{A plot of the quantity $N_{TJ}(t) = \frac{1}{A_{W}}\int_\Omega J_u(x) \;dx$ against time $t$, where $u(x,t)$ is a solution to the new phase-field method \eqref{eq:dragphasefield} \& \eqref{eq:mobility}.}
        \label{fig:jacobiandrag}
\end{figure}

\begin{figure}[H]
    \centering        \includegraphics[width=.9\linewidth]{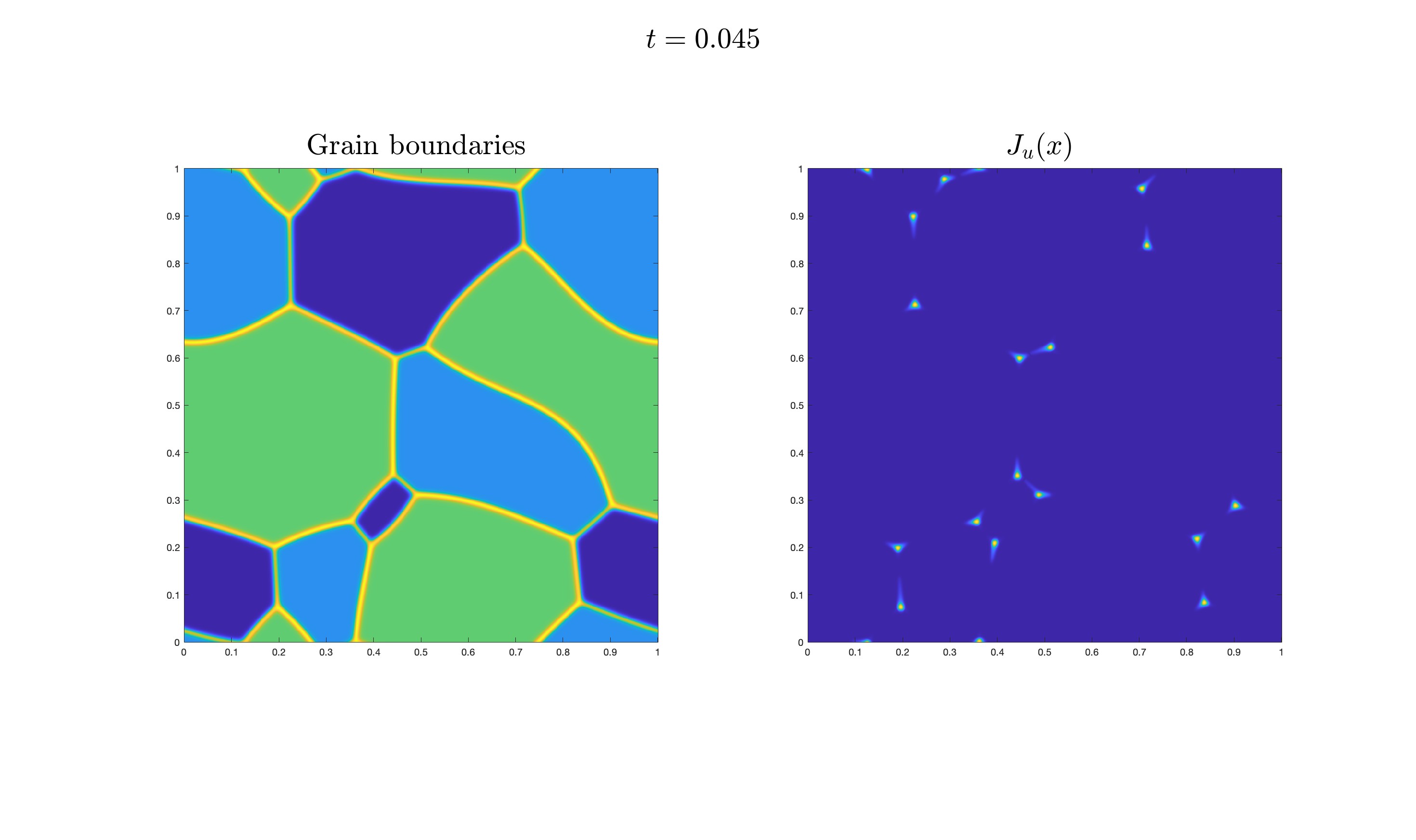}%
        \caption{Evolution of grain boundaries with triple junction drag, computed using the proposed method \eqref{eq:dragphasefield} \& \eqref{eq:mobility}.
        There are three phases, with 18 junctions at time 
$t=0.045$. Junction angles change, yet the mass of $J_u$ near each junction remains essentially constant in time.}
        \label{fig:drag1}
\end{figure}

\begin{figure}[H]
    \centering        \includegraphics[width=.9\linewidth]{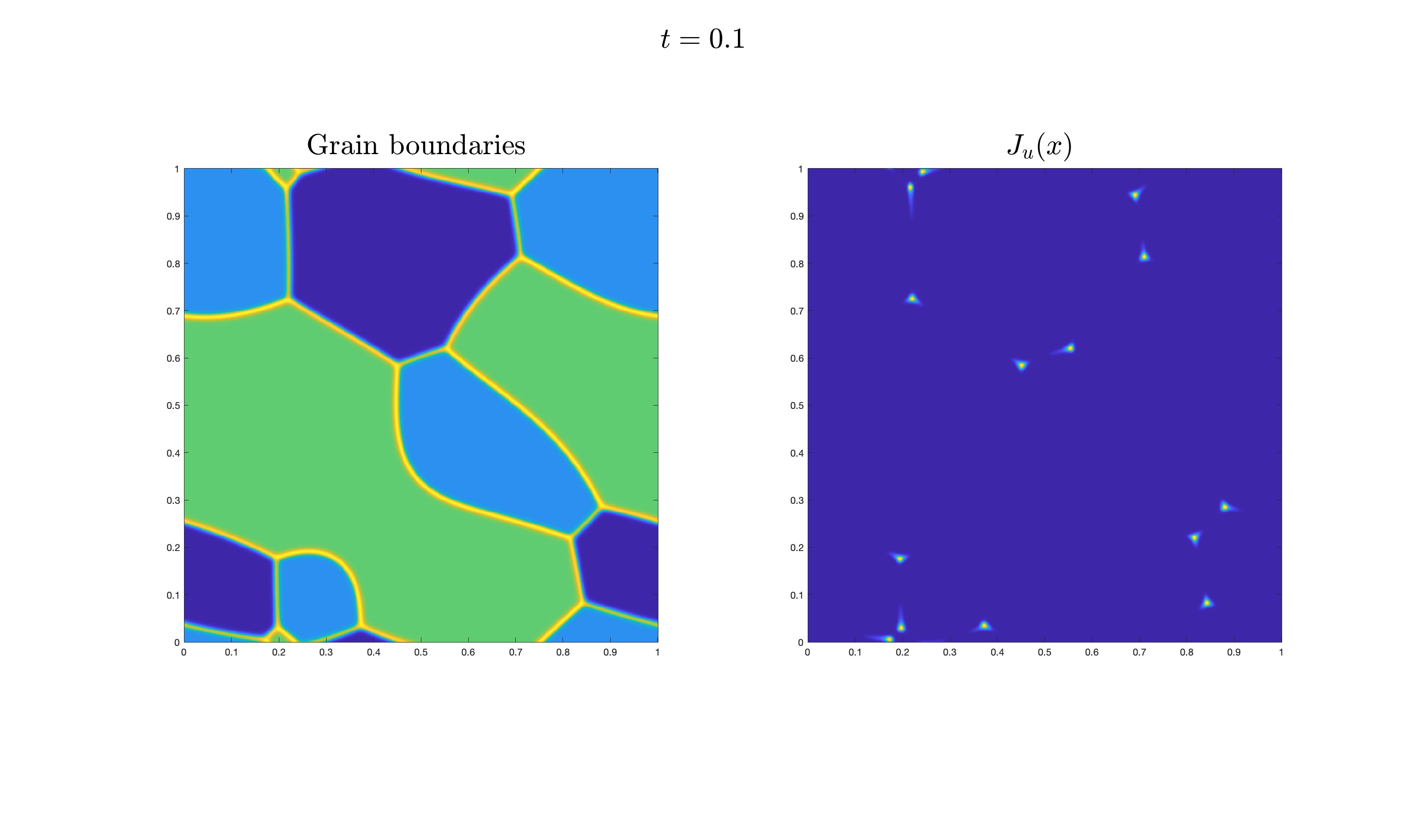}%
        \caption{Further evolution, at time $t=0.1$, of the experiment of Figure \ref{fig:drag1}, under the proposed method \eqref{eq:dragphasefield} \& \eqref{eq:mobility}. Many topological changes have taken place; $14$ junctions remain.}
        \label{fig:drag2}
\end{figure}

\begin{figure}[H]
    \centering        \includegraphics[width=.9\linewidth]{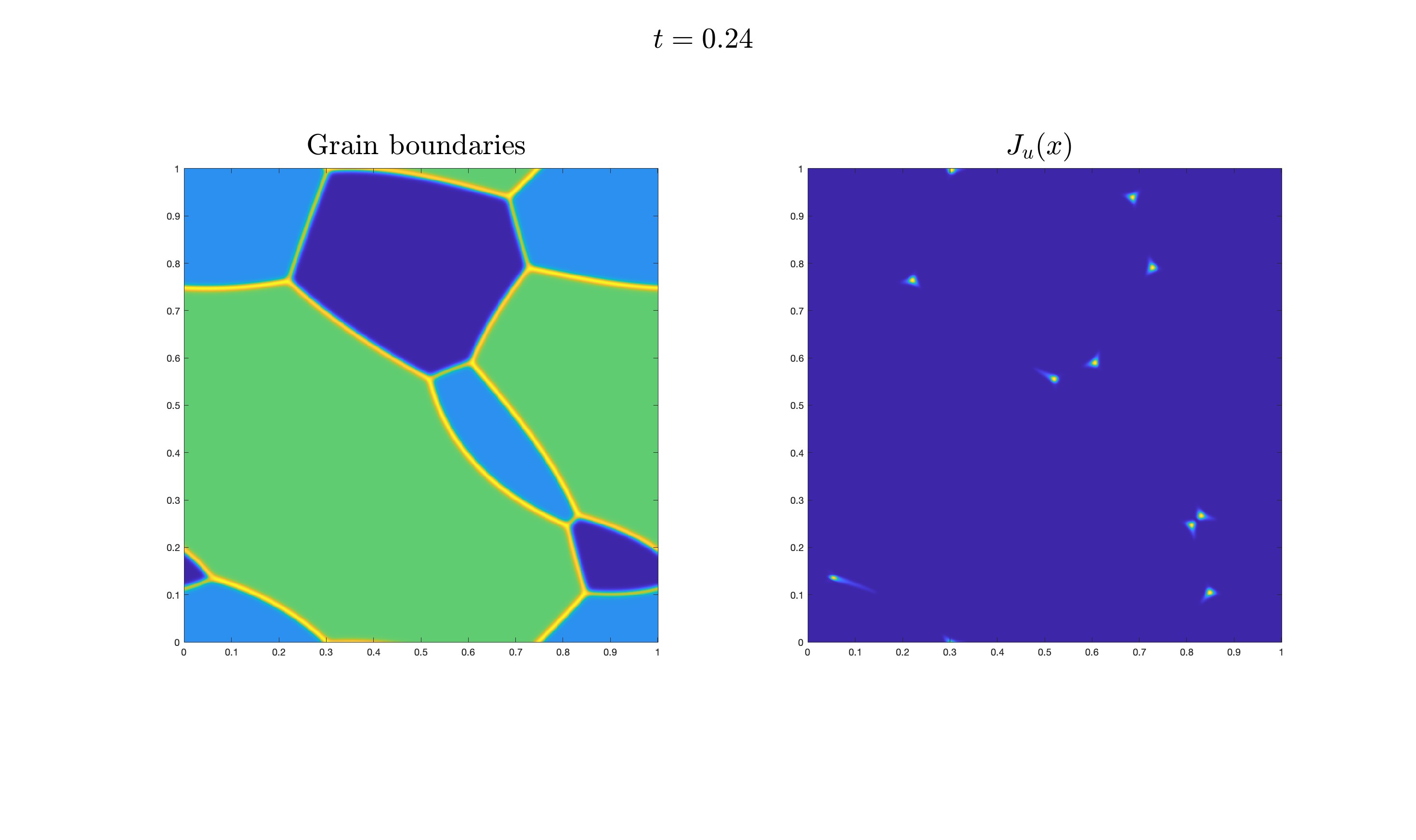}%
        \caption{Even further evolution of the experiment of Figures \ref{fig:drag1} and \ref{fig:drag2}, now at time $t=0.24$. $10$ triple junctions remain. Although $J_u$ is smeared in complicated ways at junctions with acute angles, its mass near each junction remains very closely the same, as intended.}
        \label{fig:drag3}
\end{figure}

\section{Generalization to $N\geq 4$ phases}
\label{sec:Nphases}

Our PDE \eqref{eq:dragphasefield} \& \eqref{eq:mobility} extends to $N\geq 4$ phases naturally.
In this case, as in \cite{BALDO199067}, the order parameter $u(x,t)$ takes values in $\R^{N-1}$ and the potential $W:\R^{N-1} \to \R$ has $N$ wells located at vertices of a regular $(N-1)-$simplex.
The mobility factor \eqref{eq:mobility} -- the novelty of our method -- becomes
\begin{equation}
\label{eq:mobility4}
    M(\nabla u) \coloneqq I_{(N-1)\times (N-1)} - \frac{1}{A_W}J_u\nabla u \bigg[\eps m_{TJ}\bigg(\nabla u^\top \nabla u + I_{2\times 2} \bigg)^2 + \frac{1}{A_W}J_u \nabla u^\top \nabla u \bigg]^{-1}\nabla u^\top.
\end{equation}

We give an example with $N=4$ phases:
\begin{equation}
    \begin{split}
        W(x) &= C_\beta |x-\beta_1|^2 ||x-\beta_2|^2|x-\beta_3|^2|x-\beta_4|^2\\
        \beta_1 &= \bigg(\frac{2}{\sqrt{3}} , 0 , -\frac{1}{2}\sqrt{\frac{2}{3}}\bigg)\\
        \beta_2 &= \bigg( -\frac{1}{\sqrt{3}} , 1 , -\frac{1}{2}\sqrt{\frac{2}{3}}\bigg)\\
        \beta_3 &= \bigg( -\frac{1}{\sqrt{3}} , -1 , -\frac{1}{2}\sqrt{\frac{2}{3}}\bigg)\\
        \beta_4 &= \bigg( 0 , 0 , \frac{3}{2}\sqrt{\frac{2}{3}}\bigg)\\
        C_\beta &\approx 0.0322.
    \end{split}
\end{equation}
As before, the constant $C_\beta$ was computed numerically to ensure that the surface tensions induced by $W$ according to \eqref{eq:distance} are all equal to one.

A key question in the $N\geq 4$ case is the validity of identity \eqref{eq:changeofvar}, since (in general) there could be infinitely many 2-dimensional curved surfaces that share the same boundary (geodesics between any three wells of $W$).
It is conceivable that this surface, namely $u(B_\varepsilon,t)$, changes with the profile of the order parameter $u(\cdot,t)$ in the vicinity $B_\varepsilon$ of a junction the angles of which evolves in time.
This raises the question of whether the quantity $\int_{B_\eps} J_u (x) \;dx$ remains a constant depending only on $W$, and if so, what that constant is. 

Our numerical simulations clearly indicate that the quantity $\int_{B_\eps} J_u (x) \;dx$ is indeed very nearly constant during the evolution by \eqref{eq:dragphasefield} \& \eqref{eq:mobility4}.
Moreover, we find that the image $u(B_\varepsilon,t)$ of $B_\eps$ (a small neighborhood containing a triple junction where phases $i,j,k$ meet) under $u(\cdot,t)$ very closely matches  the surface $\mathcal{S}_{ijk} := u_*(\mathbb{R}^2)$, where $u_*$ solves the equilibrium problem 
\begin{equation}
\label{eq:eqsol}
    \Delta u_* (x) = \nabla W(u_* (x)), \qquad x\in\R^2
\end{equation}
subject to far field conditions given by equations 16-18 of \cite{bronsardreitich} (with wells $\beta_i,\beta_j,\beta_k$), i.e. 
\begin{equation}\label{eq:conjecture}
    u(B_\eps,t) \approx u_*(\R^2) = \mathcal{S}_{ijk}
\end{equation}
regardless of the local geometry of the junction; see Figures \ref{fig:drag_evolution} and \ref{fig:drag_surface}.

\begin{figure}[H]
        %\subfloat[Initial configuration for \eqref{eq:dragphasefield} with $m_{TJ}=2$]{%
            \includegraphics[width=.5\linewidth]{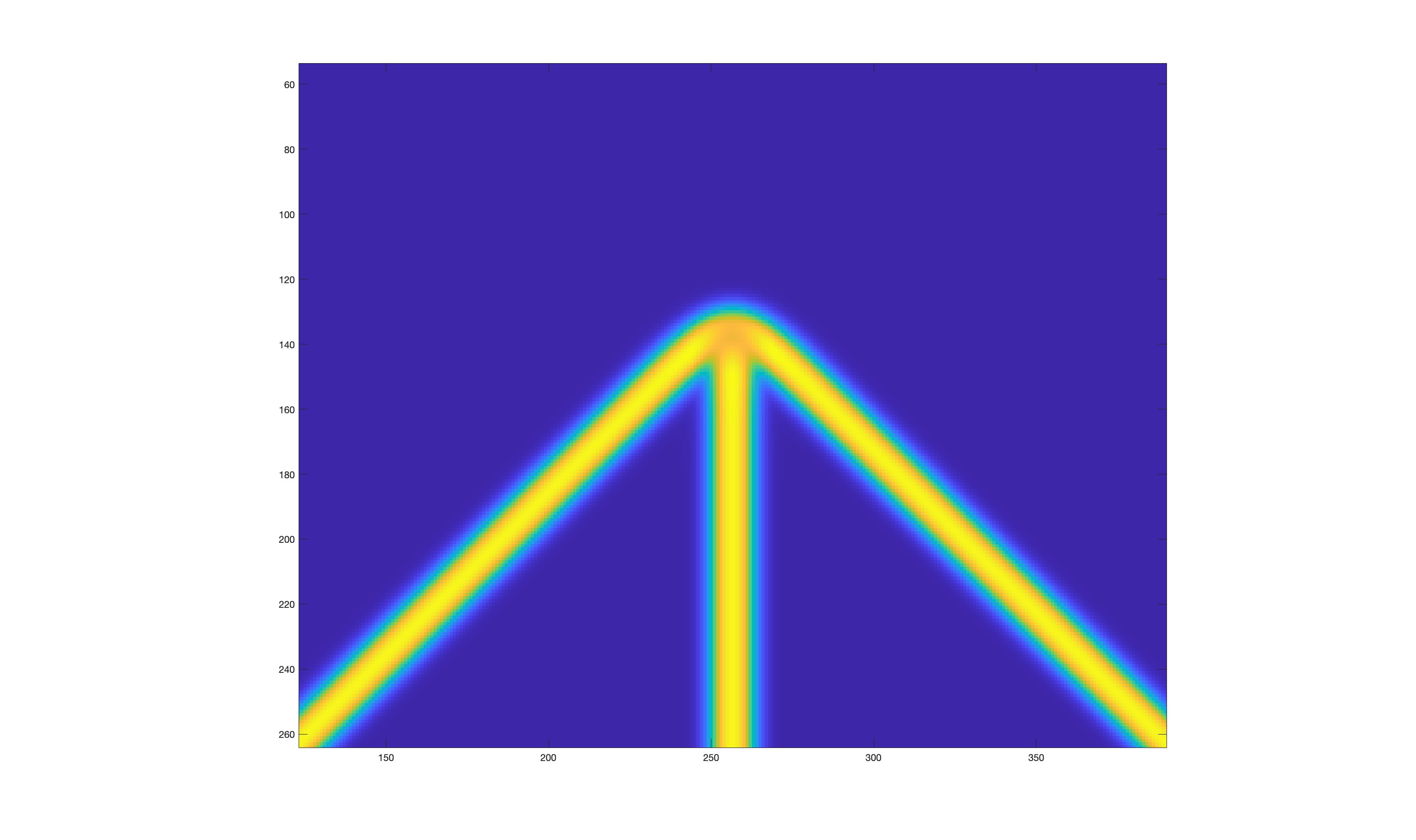}%
            % \label{subfig:drag_junction}%
        %}\hfill
        %\subfloat[Final configuration at $T = 0.2$]{%
            \includegraphics[width=.5\linewidth]{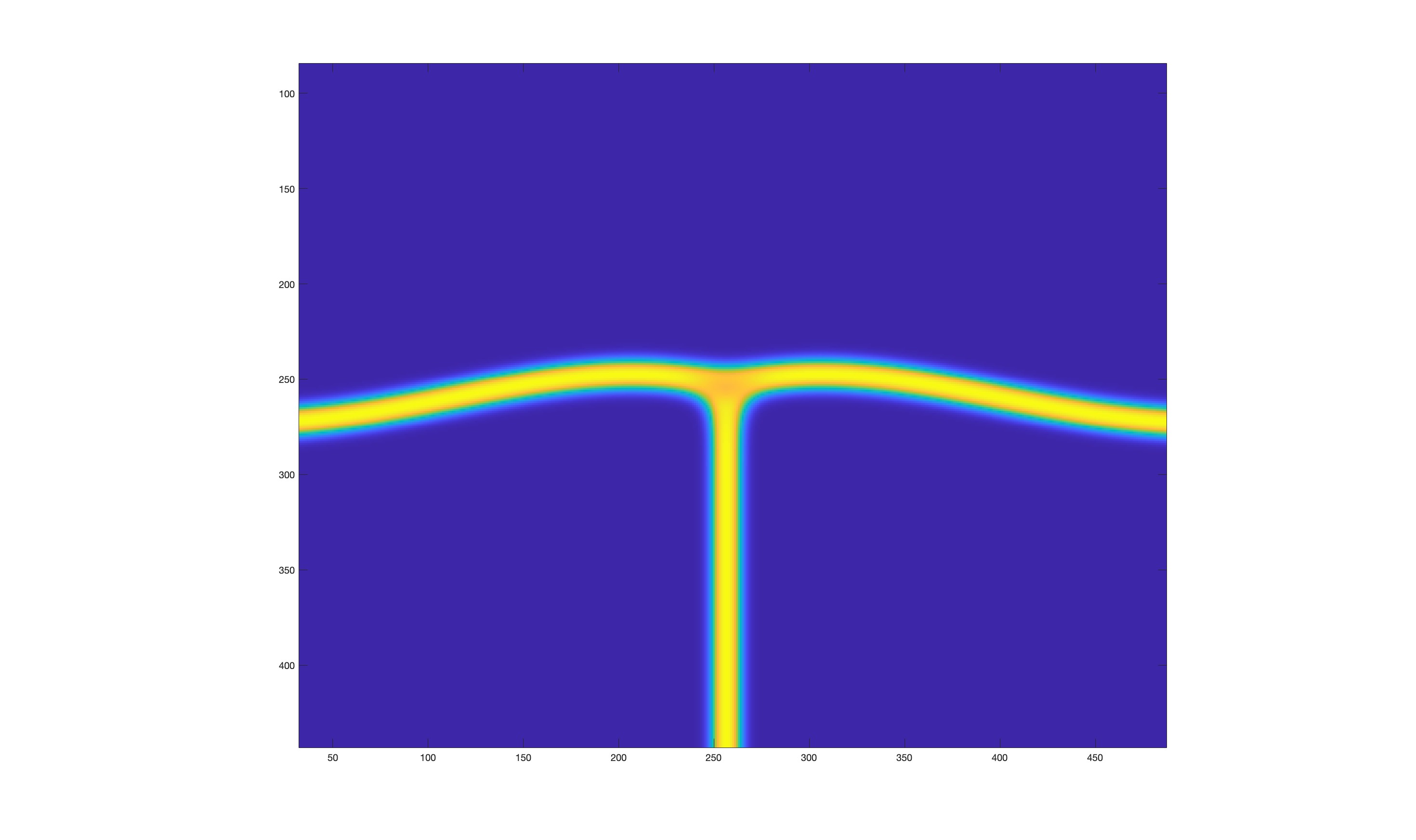}%
            % \label{subfig:drag_surface}%
        %}
        \caption{Left panel: Initial condition for proposed method \eqref{eq:dragphasefield} \& \eqref{eq:mobility4} allowing $N=4$ phases, with $m_{TJ}=2$. Right panel: Computed solution at time $T=0.2$. Angle at the junction has changed dramatically.}
        \label{fig:drag_evolution}
\end{figure}

\begin{figure}[H]
        \includegraphics[width=.5\linewidth]{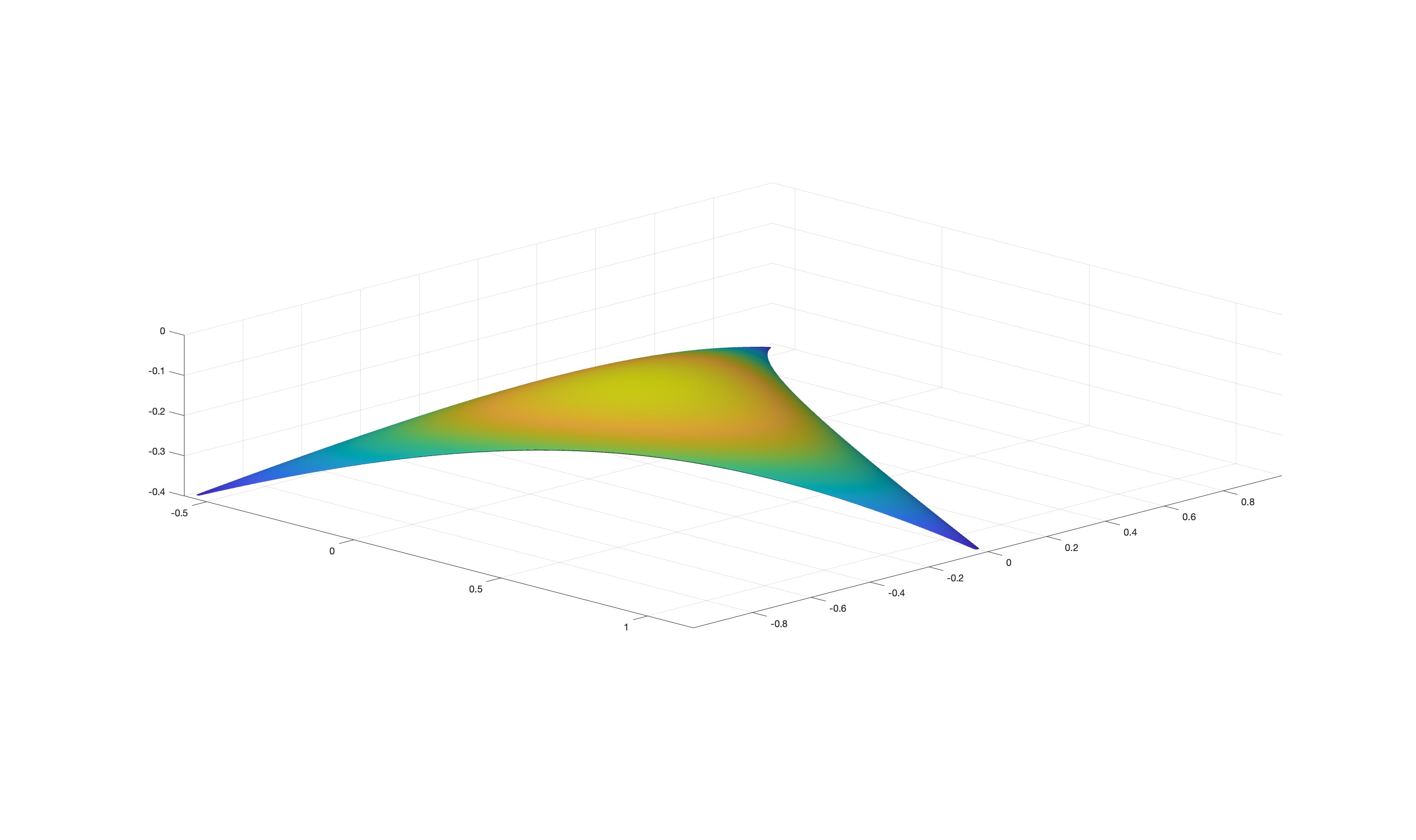}
            \includegraphics[width=.5\linewidth]{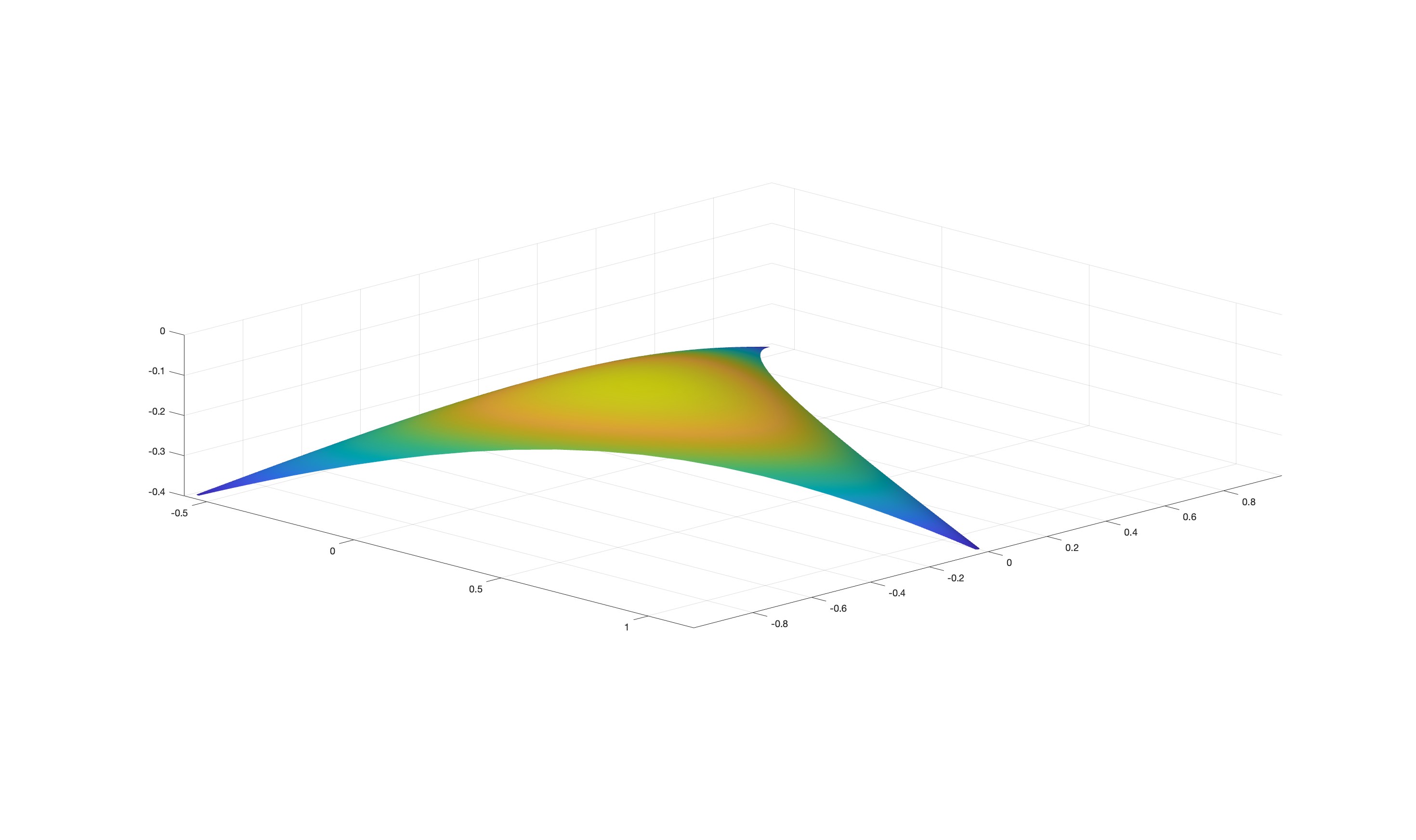}
        \caption{Left panel: The image $u(B_\varepsilon,t)$ of a small neighborhood $B_\varepsilon$ of the triple junction at time $T=0.2$ shown in the right panel of Figure \ref{fig:drag_evolution}. Right panel: Range $\mathcal{S}_{ijk} = u_*(\R^2)$ of the equilibrium solution $u_*$ as defined in \eqref{eq:eqsol}.}
        \label{fig:drag_surface}
\end{figure}

The natural extension of $A_W$ to the setting of $N\geq 4$ phases is then
\begin{equation}
    A_W \coloneqq \text{Area}(\mathcal{S}_{ijk}) \approx 1.026
\end{equation}
which does not depend on the choice of $i,j,k\in\{1,2,3,4\}$ due to the symmetry of $W$ with respect to its four wells. 

At this stage, we only provide numerical evidence that supports \eqref{eq:conjecture}, leaving an analysis of its validity to future work.
To demonstrate \eqref{eq:conjecture}, we simulate \eqref{eq:dragphasefield} \& \eqref{eq:mobility4} starting from an initial configuration with extreme angles (Figure \ref{fig:drag_evolution}) and track the evolution of the quantity
\begin{equation}
\label{eq:quantity}
    \int_{\Omega}\frac{J_u(x)}{A_W}\;dx \approx 1
\end{equation}
over time (Figure \ref{fig:massrecord}) as the angles at the junction change dramatically.
It can be seen to be indeed very nearly one for all time, after an initial rapid transition.

\vfill

\begin{figure}[H]
    \centering        \includegraphics[width=.9\linewidth]{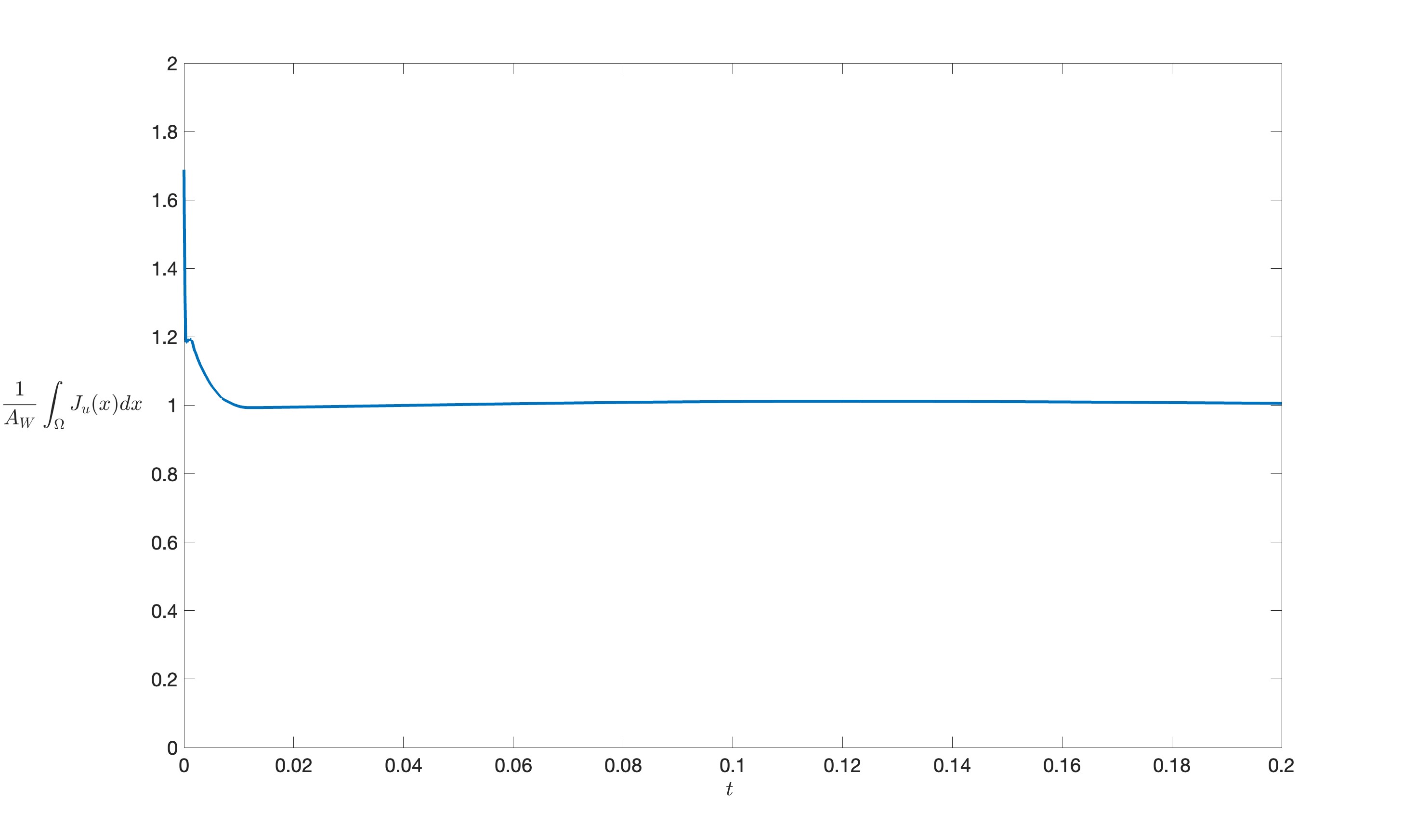}%
        \caption{Evolution of the quantity $\int_{\Omega}\frac{J_u(x)}{A_W}\;dx$ where $u(x,t)$ solves the proposed phase-field method \eqref{eq:dragphasefield} \& \eqref{eq:mobility4} for $N\geq 4$ phases. Initial and final conditions are shown in Figure \ref{fig:drag_evolution}.}
        \label{fig:massrecord}
\end{figure}

\subsection{Numerical simulations for $N=4$ phases}

Various numerical convergence studies, analogous to those in Section \ref{sec:numerics}, for the natural extension \eqref{eq:dragphasefield} \& (\ref{eq:mobility4}) to $N\geq 4$ phases of our new phase-field method all indicate clear convergence to the desired sharp interface model.
In the interest of space, we only show results from the most challenging test, Figure \ref{fig:forkasym4}, where the junctions formed between $N=4$ phases travel along curved paths while their angles change in time.
Plot \ref{fig:error_forkasym_4phase} and Table \ref{tab:table_forkasym_4phase} show very clean linear convergence rate with respect to the parameter $\varepsilon$.

\clearpage

\begin{figure}[H]
\begin{center}
        \subfloat[$n=256$]{%
            \includegraphics[width=.31\linewidth,
    trim=20cm 0cm 20cm 0cm,
    clip]{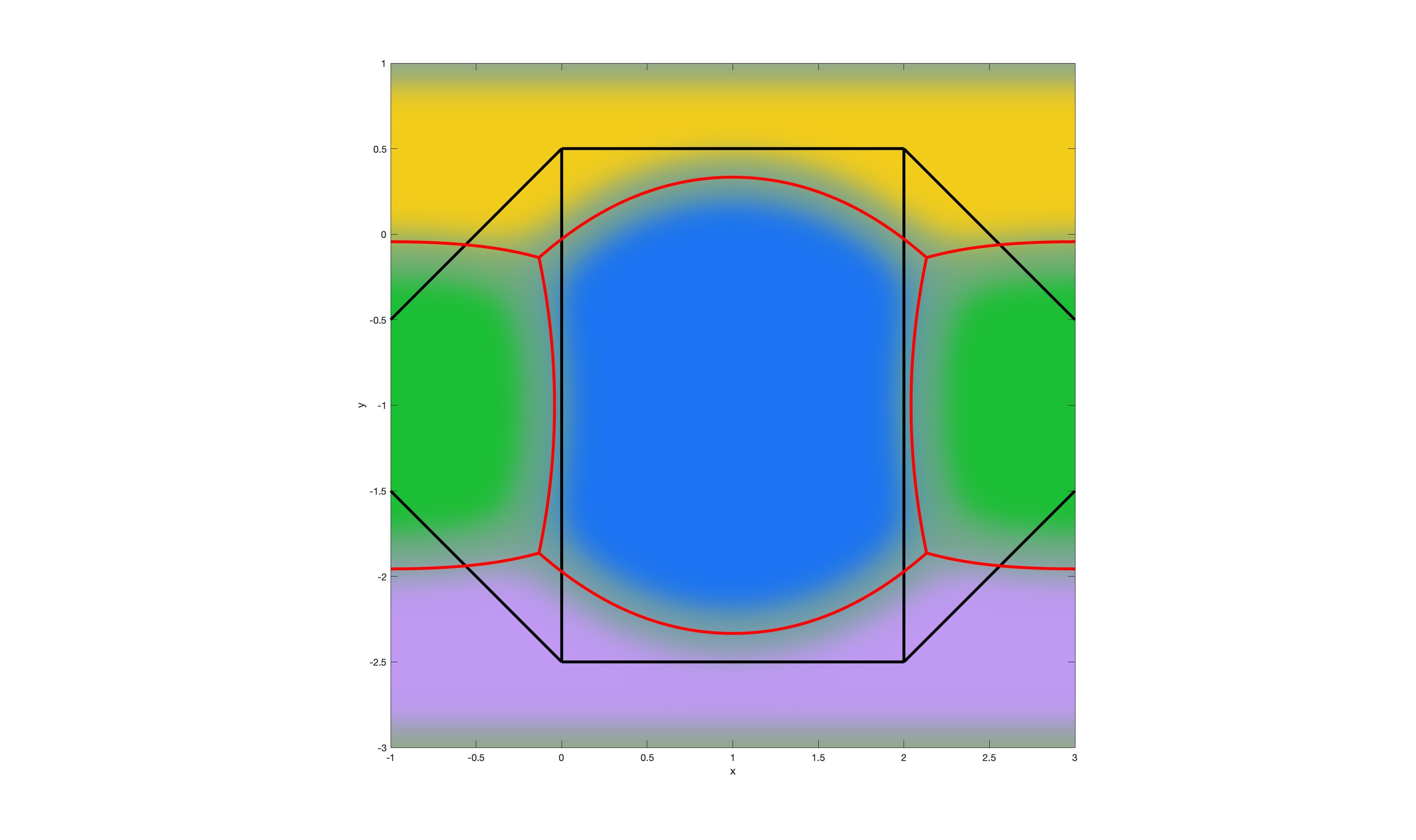}%
            \label{subfig:fork_asym_4phase_1}%
        }\hspace{40pt}
        \subfloat[$n=512$]{%
           \includegraphics[width=.31\linewidth,
    trim=20cm 0cm 20cm 0cm,
    clip]{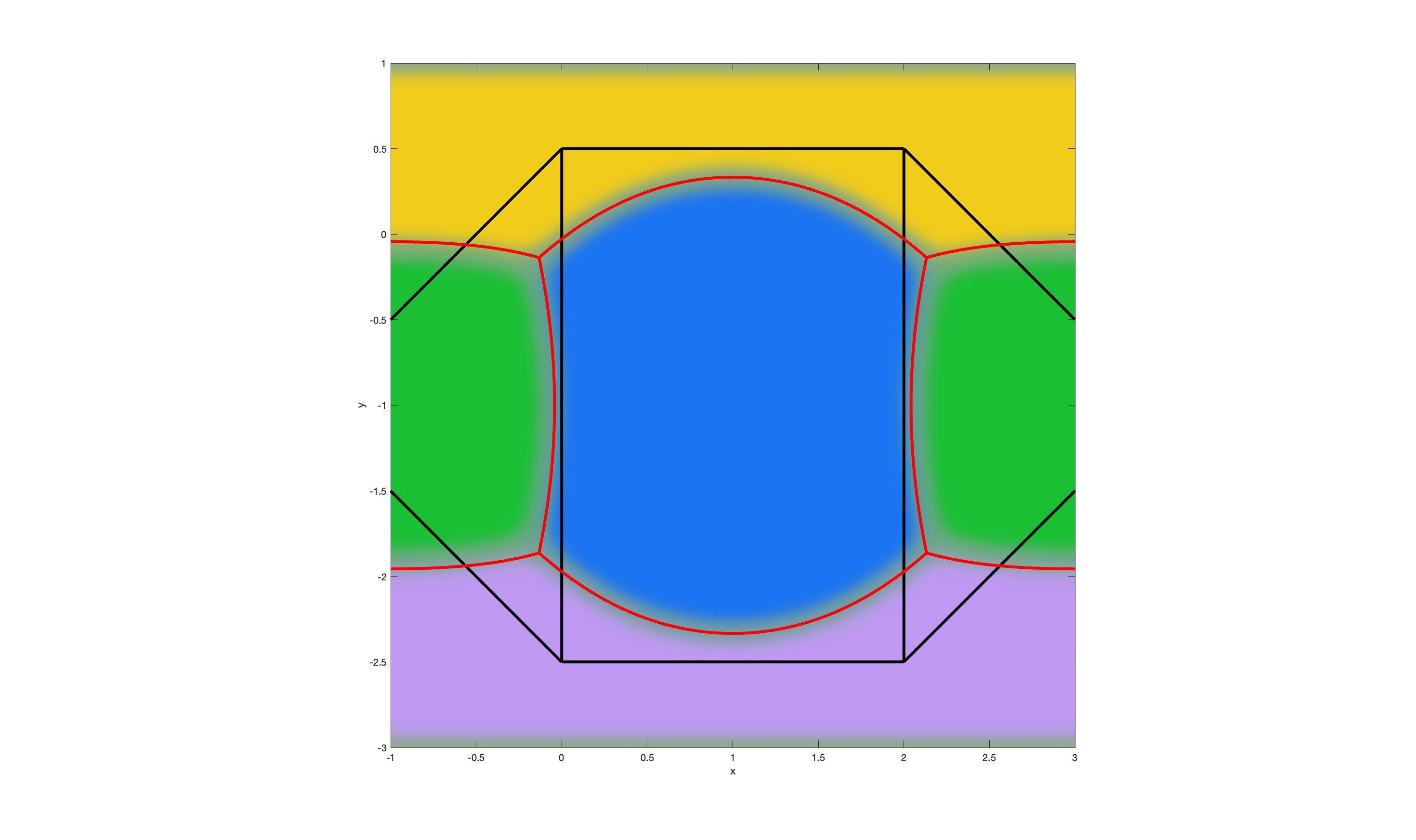}%
            \label{subfig:fork_asym_4phase_2}%
        }\\
        \subfloat[$n=1024$]{%
            \includegraphics[width=.31\linewidth,
    trim=20cm 0cm 20cm 0cm,
    clip]{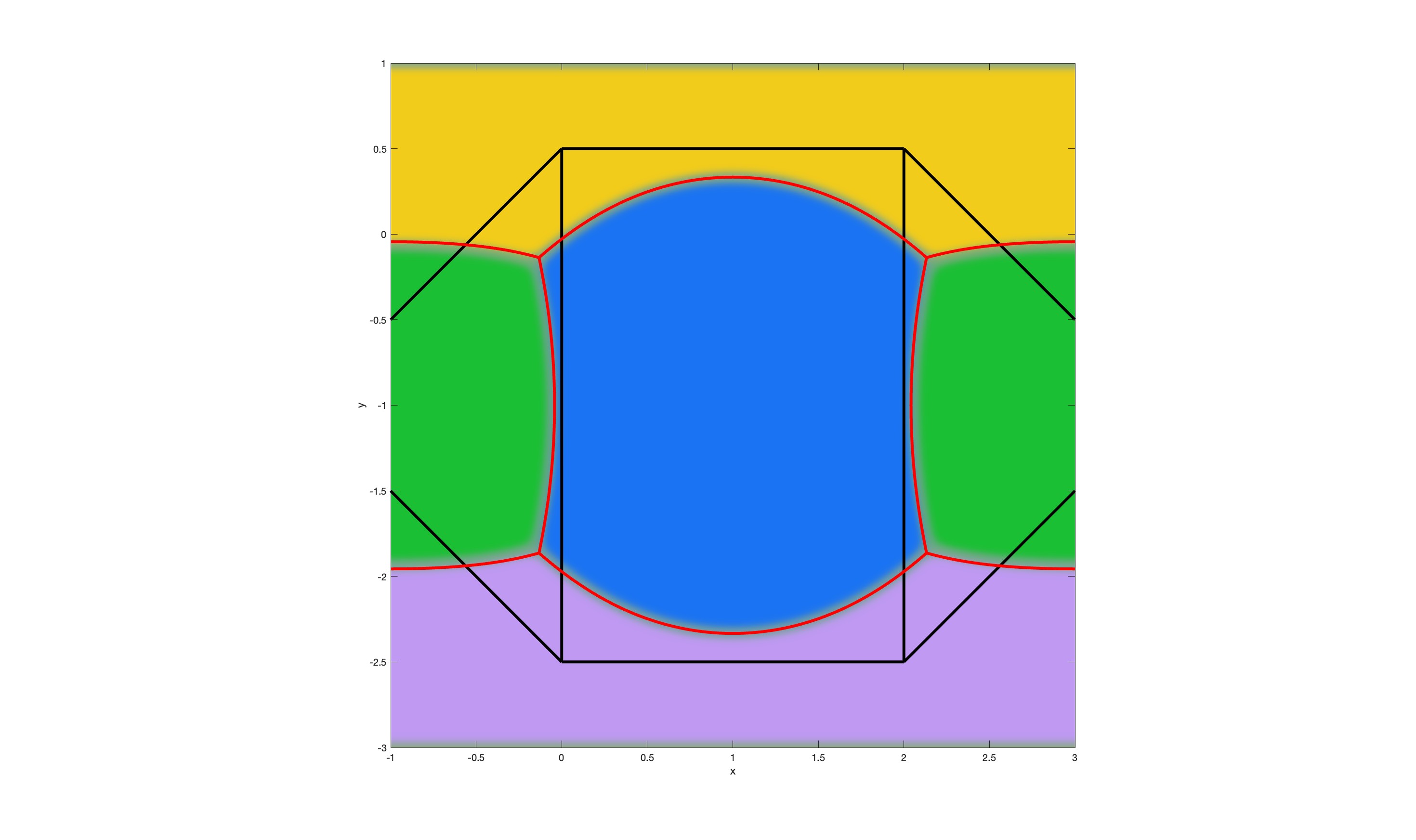}%
            \label{subfig:fork_asym_4phase_3}%
        }\hspace{40pt}
        \subfloat[$n=2048$]{%
            \includegraphics[width=.31\linewidth,
    trim=20cm 0cm 20cm 0cm,
    clip]{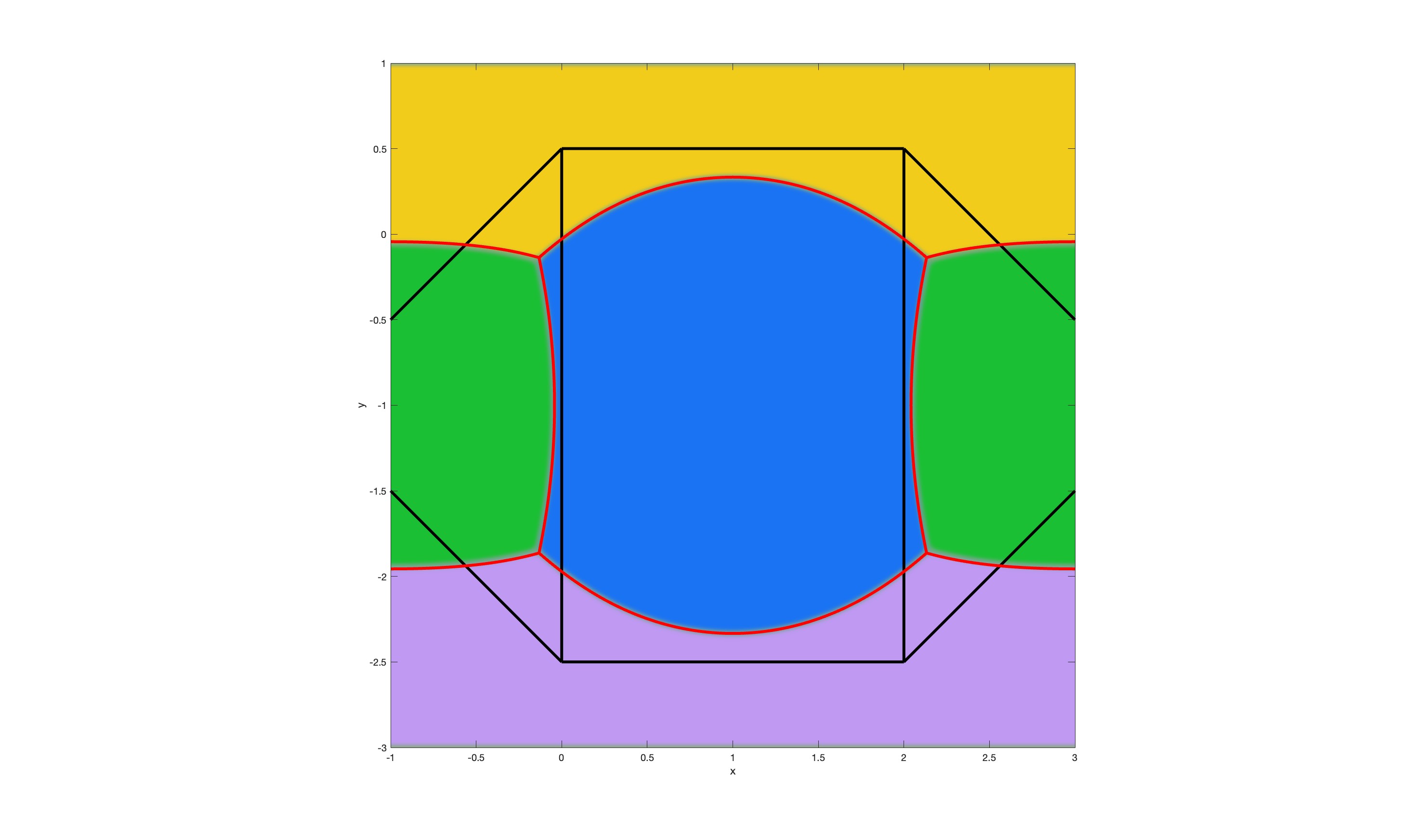}%
            \label{subfig:fork_asym_4phase_4}%
        }\hfill
        \caption{Convergence study with $N=4$ phases with a configuration where the motion of the triple junction is not aligned with any of the interfaces. Benchmark curves (red) were obtained via a front tracking method. $m_{TJ} = 10$, final time $T=0.4$.}
        \label{fig:forkasym4}
\end{center}
\end{figure}

\begin{center}
\begin{minipage}[t]{0.54\textwidth}
    \centering
    \vspace{30pt}
    \def\arraystretch{1.5}
    \small
    \resizebox{\linewidth}{!}{
        \begin{tabular}{|l|l|l|l|l|}
        \hline
        $\delta x$ & $\eps$ &  $\delta t$            & Error  & Order \\ \hline
        4/(256-1) = 0.0156    & 0.0941  &  $4.92 \times 10^{-5}$ & 1.419  & -     \\ \hline
        4/(512-1) = 0.00782    & 0.0469  &  $1.22 \times 10^{-5}$ & 0.716  & 0.987 \\ \hline
        4/(1024-1) = 0.00391    & 0.0234 &  $3.05 \times 10^{-6}$ & 0.360  & 0.991  \\ \hline
        4/(2048-1) = 0.00195    & 0.0117 &  $7.63\times 10^{-7}$  & 0.181 & 0.987 \\ \hline
        \end{tabular}
        }
        \vspace{30pt}
        \captionof{table}{Error and order of convergence for the numerical test shown in Figure \ref{fig:forkasym4}, using the proposed method \eqref{eq:dragphasefield} \& \eqref{eq:mobility4}.}
        \label{tab:table_forkasym_4phase}
    \end{minipage}
    \hfill
    \begin{minipage}[t]{0.44\textwidth}
        \centering
        % \vspace{0pt}
        \begin{figure}[H]
            \centering
            \includegraphics[width=1\linewidth]{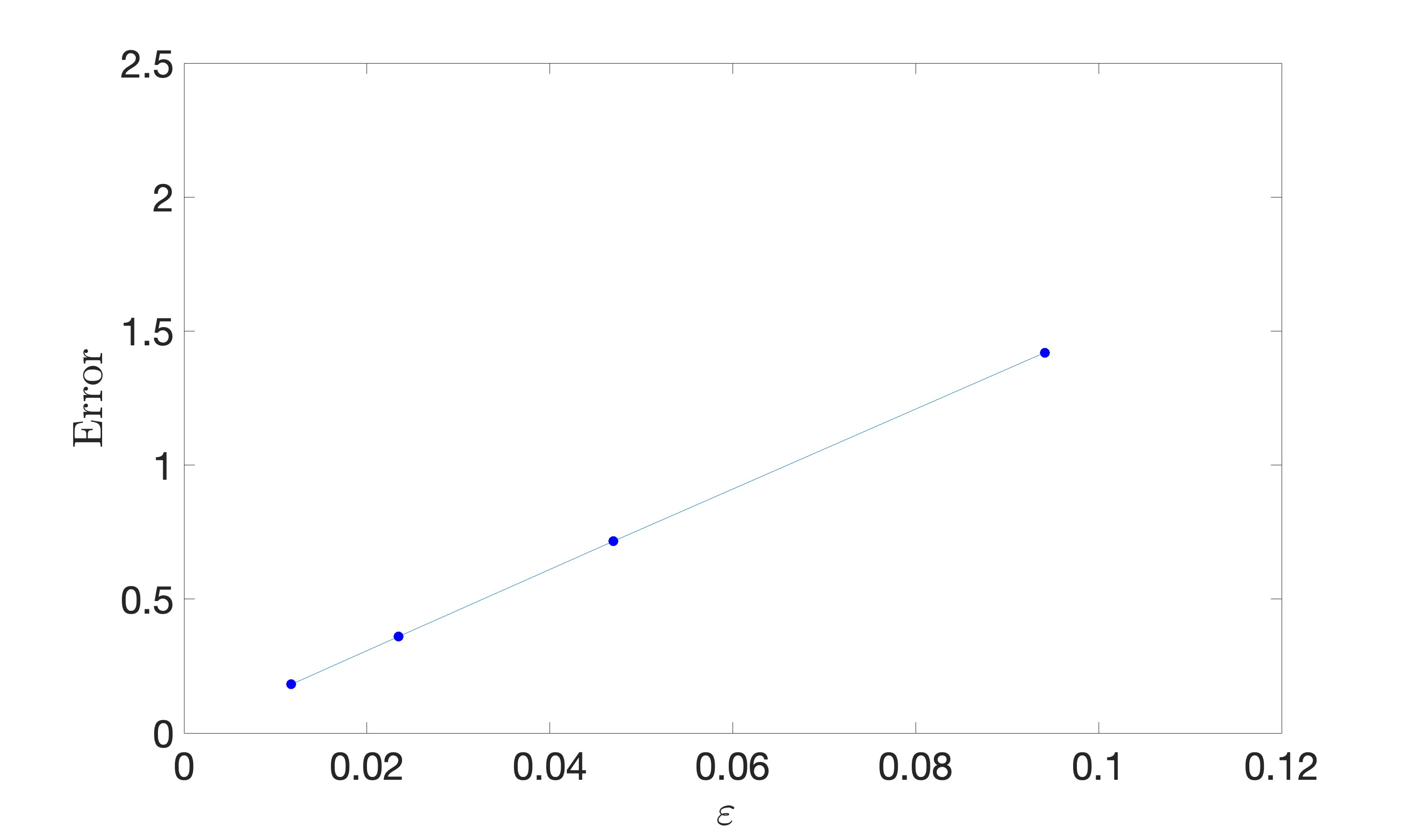}%
            \caption{Behavior of the error in the numerical solution generated by the proposed method \eqref{eq:dragphasefield} \& \eqref{eq:mobility4}, on the test case of Figure \ref{fig:forkasym4}.}
            \label{fig:error_forkasym_4phase}
        \end{figure}
    \end{minipage}
\end{center}

% %\clearpage

% \begin{figure}[H]
%     \centering
%     \includegraphics[width=0.7\linewidth]{figures/error_forkasym_4phase.jpg}%
%     %\caption{Plot of error against $\eps$}
%     \caption{\textcolor{blue}{Behavior of the error in the numerical solution generated by the proposed method \eqref{eq:dragphasefield} \& \eqref{eq:mobility4}, on the test case of Figure \ref{fig:forkasym4}.}}
%     \label{fig:error_forkasym_4phase}
% \end{figure}

% \begin{table}

% \centering
% \def\arraystretch{1.5}
% \begin{tabular}{|l|l|l|l|l|}
% \hline
% $\delta x$ & $\eps$ &  $\delta t$            & Error  & Order \\ \hline
% 4/(256-1) = 0.0156    & 0.0941  &  $4.92 \times 10^{-5}$ & 1.419  & -     \\ \hline
% 4/(512-1) = 0.00782    & 0.0469  &  $1.22 \times 10^{-5}$ & 0.716  & 0.987 \\ \hline
% 4/(1024-1) = 0.00391    & 0.0234 &  $3.05 \times 10^{-6}$ & 0.360  & 0.991  \\ \hline
% 4/(2048-1) = 0.00195    & 0.0117 &  $7.63\times 10^{-7}$  & 0.181 & 0.987 \\ \hline
% \end{tabular}
% %\caption{Table of errors}
% \caption{\textcolor{blue}{Error and order of convergence for the numerical test shown in Figure \ref{fig:forkasym4}, using the proposed method \eqref{eq:dragphasefield} \& \eqref{eq:mobility4}.}}
% \label{tab:table_forkasym_4phase}

% \end{table}

\clearpage

In Figure \ref{fig:switching}, we verify that our phase-field model \eqref{eq:dragphasefield} \& \eqref{eq:mobility4} handles topological changes seamlessly. Here, a typical neighbor switching event (T1 process) occurs:
Two triple junctions collide, then split, resulting in junctions formed by a new triplet of phases.
Figure \ref{fig:manyjunctions} shows a four-phase evolution with many junctions through many topological events, seamlessly handled by method \eqref{eq:dragphasefield} \& (\ref{eq:mobility4}).
Figure \ref{fig:Ju} provides clear evidence of the integer-valuedness (measuring the number of junctions) of the integral \eqref{eq:quantity} during the evolution shown in Figure \ref{fig:manyjunctions}, as long as junctions remain apart.

\begin{figure}[H]
        \subfloat[$t = 0.14$]{%
            \includegraphics[width=.5\linewidth]{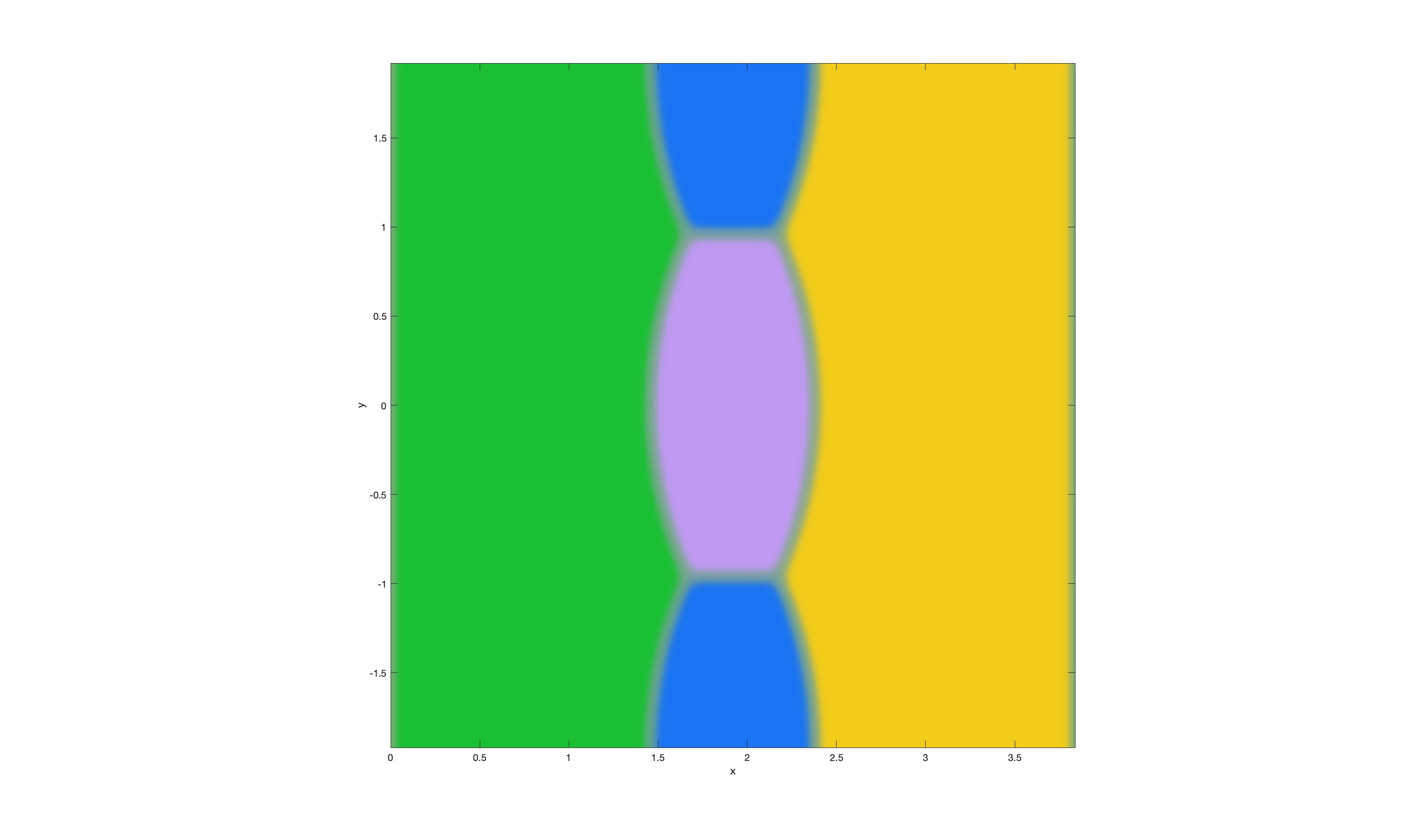}%
            \label{subfig:switching1}%
        }\hfill
        \subfloat[$t = 0.56$]{%
            \includegraphics[width=.5\linewidth]{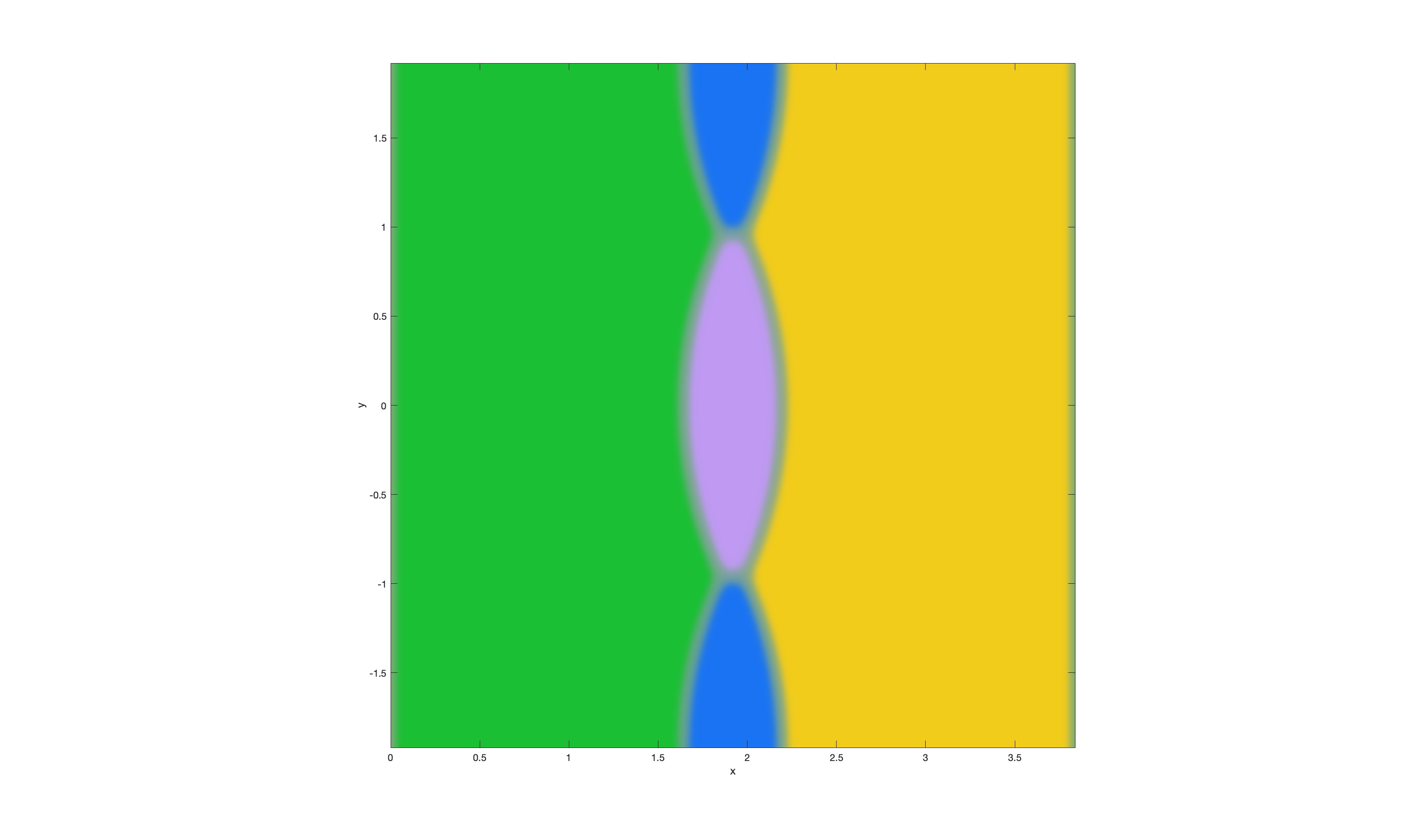}%
            \label{subfig:switching2}%
        }\\
        \subfloat[$t = 0.70$]{%
            \includegraphics[width=.5\linewidth]{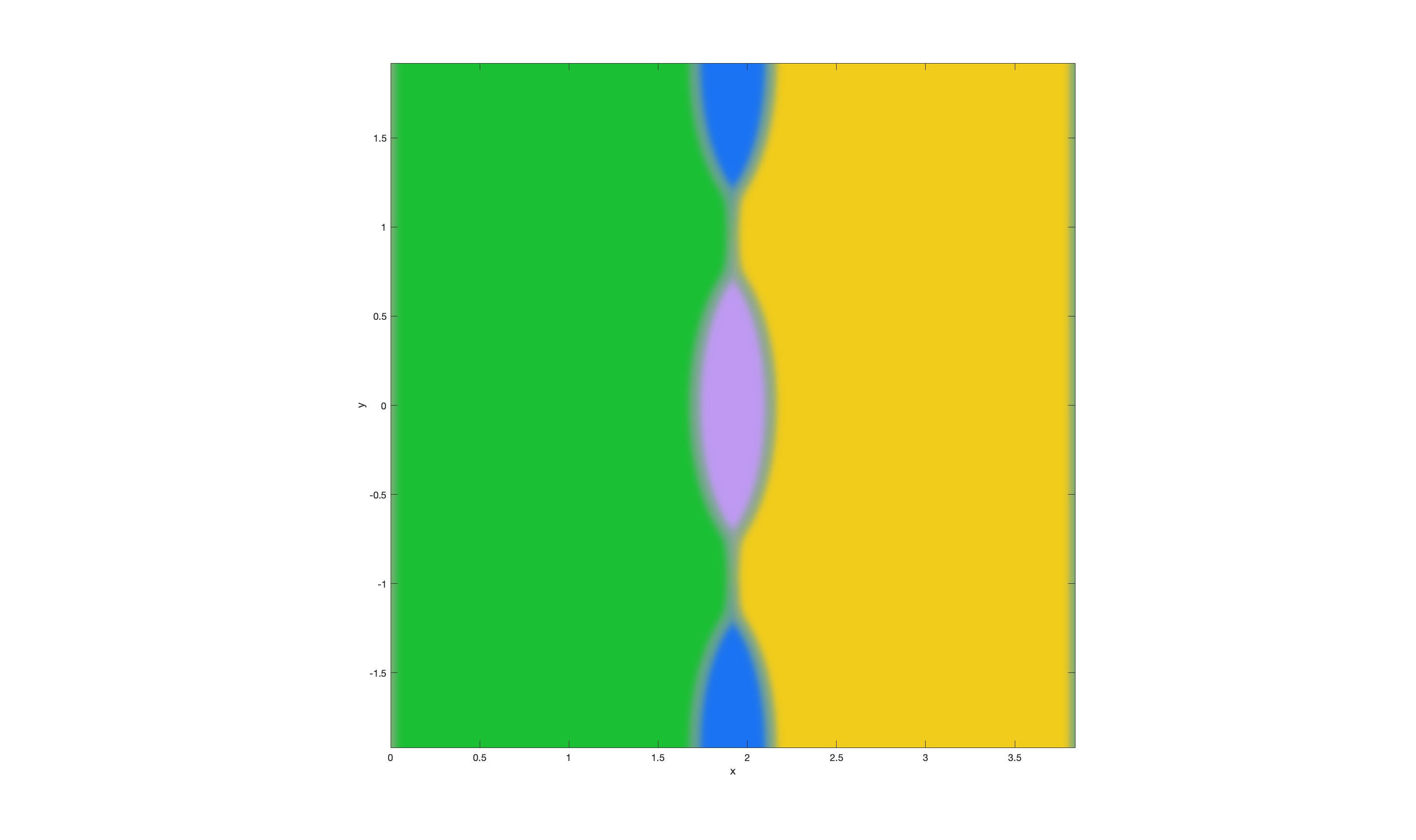}%
            \label{subfig:switching3}%
        }\hfill
        \subfloat[$t = 0.84$]{%
            \includegraphics[width=.5\linewidth]{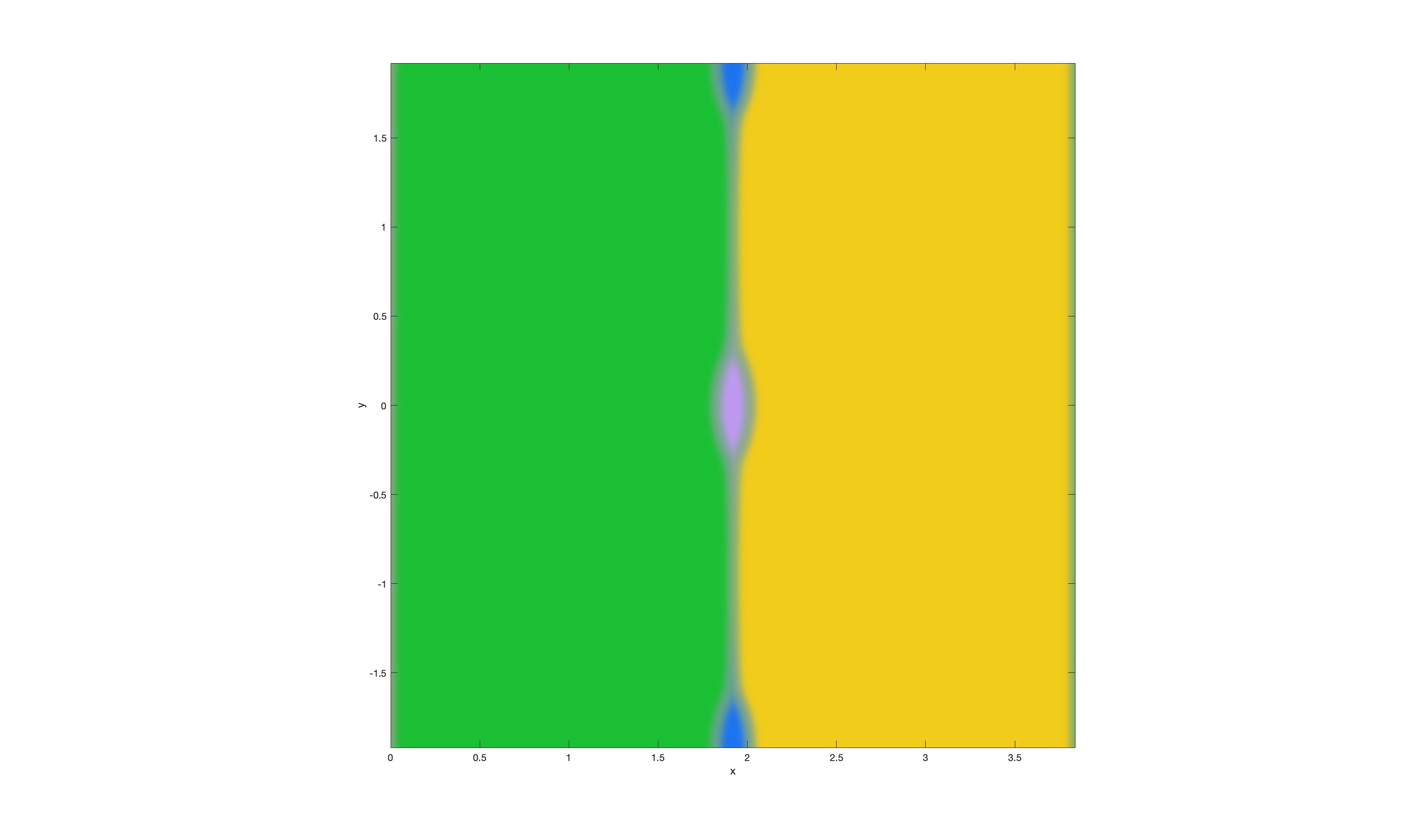}%
            \label{subfig:switching4}%
        }\hfill
        \caption{A neighbor switching event -- a common topological change -- takes place as expected using \eqref{eq:dragphasefield} \& (\ref{eq:mobility4}), with four phases in participation.}
        \label{fig:switching}
\end{figure}

%\clearpage

%\begin{figure}[H]
%    \centering        \includegraphics[width=.9\linewidth]{figures/fourphase_step.jpg}%
%        \caption{A plot of the quantity $\frac{1}{A_{W}}\int_\Omega J_u(x) \;dx$ against time $t$. Here, $u(x,t) \in \R^3$.}
%        \label{fig:jacobian4phase}
%\end{figure}
%
%\begin{figure}[H]
%    \centering        \includegraphics[width=.9\linewidth]{figures/fourphase_counting_1.jpg}%
%        \caption{Evolution of grain boundaries (4 phases) at time $t=0.061$. Each phase is represented by a different color and there are 18 triple junctions.}
%        \label{fig:fourphase1}
%\end{figure}
%
%\begin{figure}[H]
%    \centering        \includegraphics[width=.9\linewidth]{figures/fourphase_counting_2.jpg}%
%        \caption{Evolution of grain boundaries (4 phases) at time $t=0.24$. There are 16 triple junctions.}
%        \label{fig:fourphase2}
%\end{figure}
%
%\begin{figure}[H]
%    \centering        \includegraphics[width=.9\linewidth]{figures/fourphase_counting_3.jpg}%
%        \caption{Evolution of grain boundaries (4 phases) at time $t=0.55$. There are 12 triple junctions.}
%        \label{fig:fourphase3}
%\end{figure}
%
%\begin{figure}[H]
%    \centering        \includegraphics[width=.9\linewidth]{figures/fourphase_counting_4.jpg}%
%        \caption{Evolution of grain boundaries (4 phases) at time $t=0.91$. There are 10 triple junctions.}
%        \label{fig:fourphase4}
%\end{figure}

\begin{figure}[H]
\centering
\includegraphics[width=.9\linewidth]{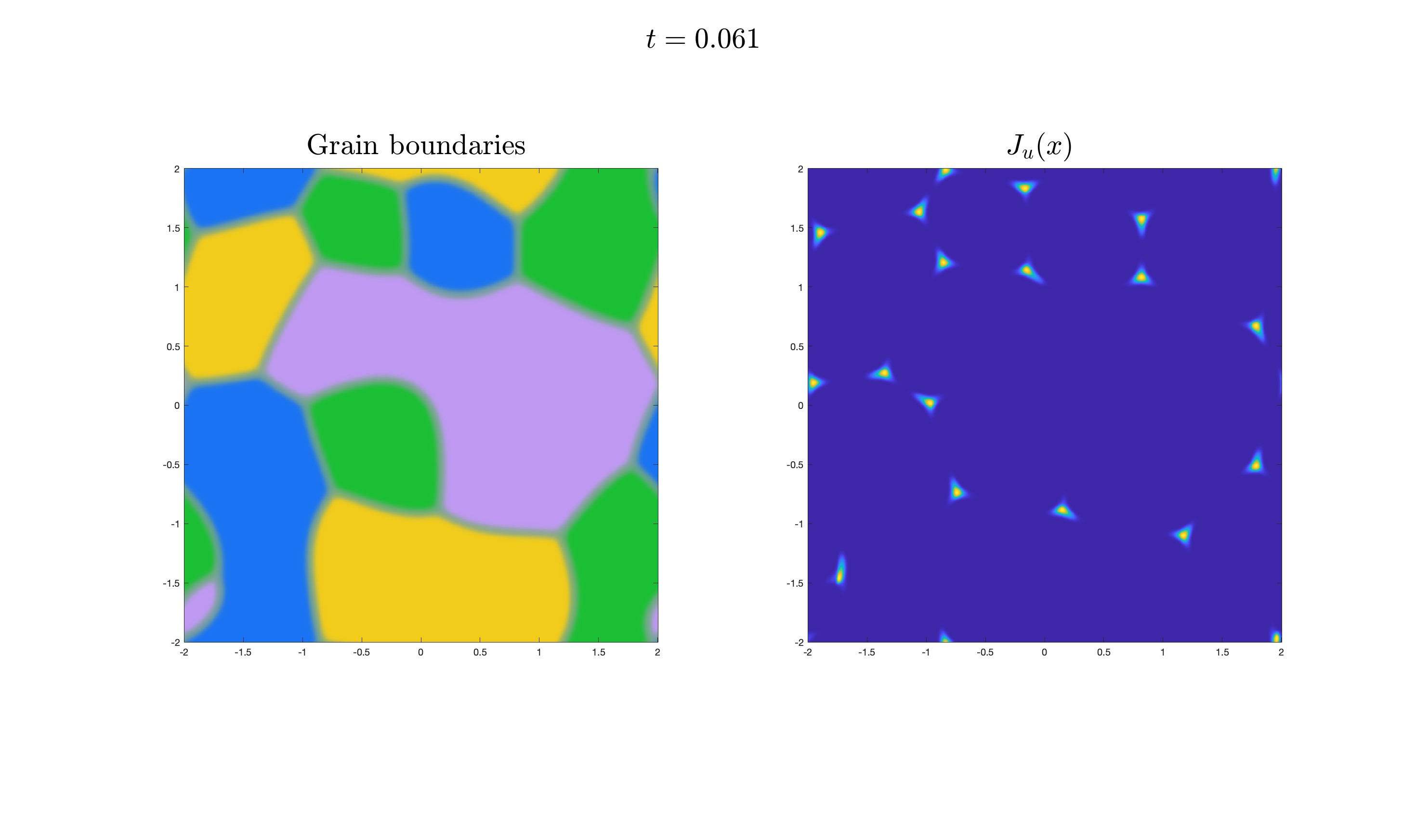}
\includegraphics[width=.9\linewidth]{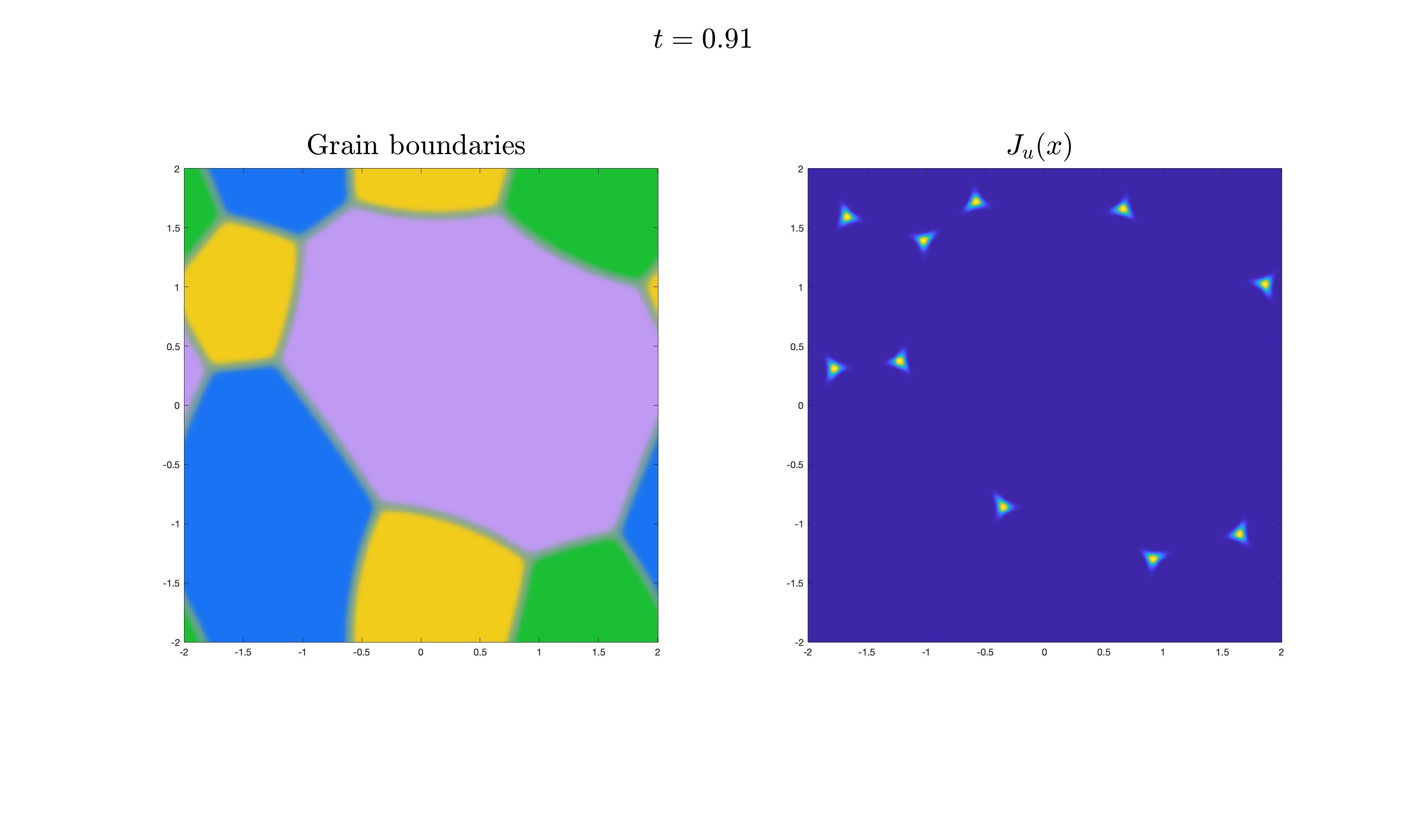}
\caption{Evolution of four phases and many junctions through many topological events, as automatically handled by our method \eqref{eq:dragphasefield} \& (\ref{eq:mobility4}).}
\label{fig:manyjunctions}
\end{figure}

\begin{figure}[H]
\centering
\includegraphics[width=.9\linewidth]{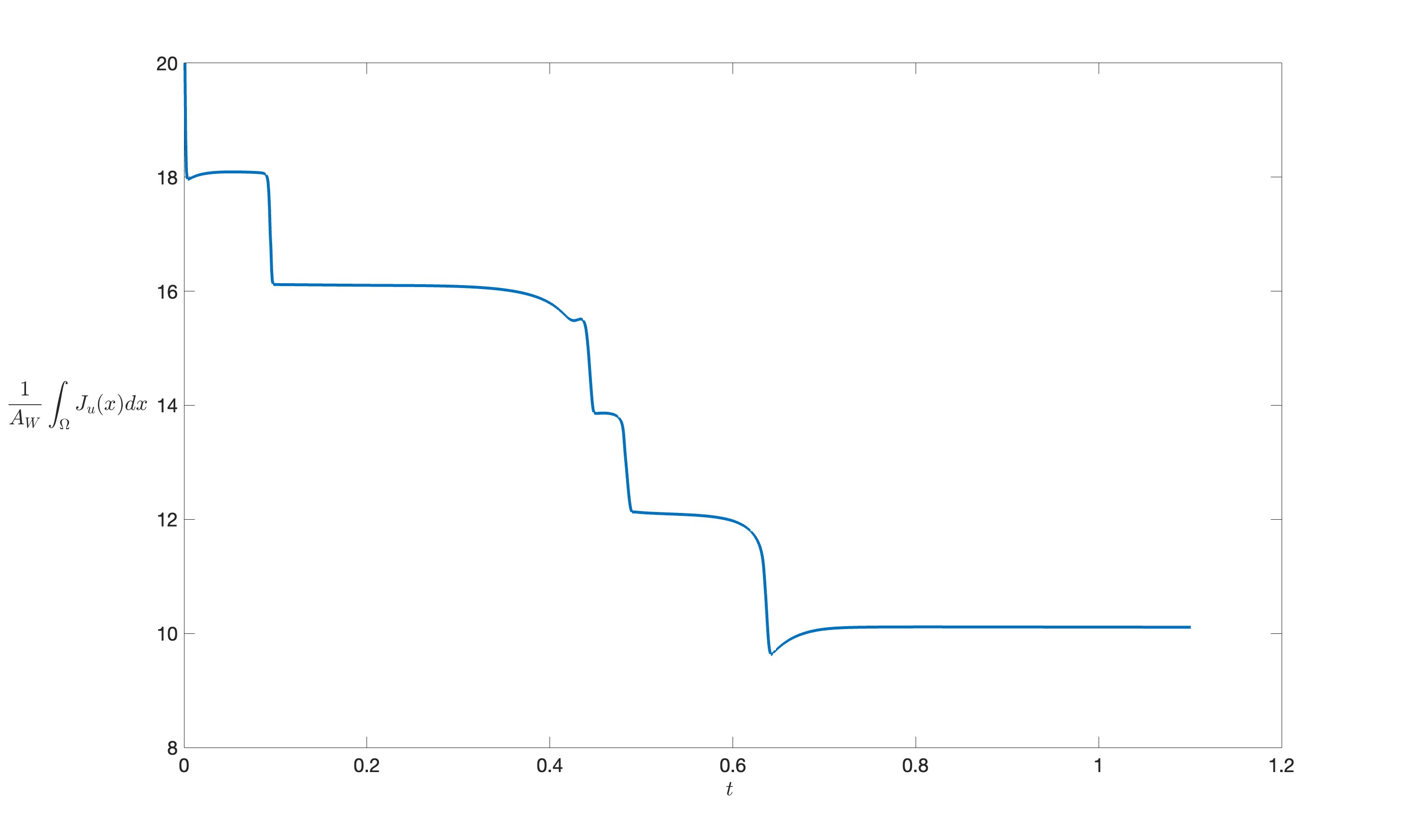}
\caption{The integral $\frac{1}{A_W} \int J_u \, dx$ given in \eqref{eq:quantity} as it evolves in time during the four-phase simulation shown in Figure \ref{fig:manyjunctions}.
Although there are many different junctions with time varying angles, this integral remains very close to integer valued, except during topological events, as expected.}
\label{fig:Ju}
\end{figure}

\section{Conclusion}

We have proposed a new phase-field method for curvature motion of networks with triple junction drag, supported by both formal asymptotic analysis and strong numerical evidence of convergence to the correct sharp-interface limit.
Along the way, we obtained a new formula that expresses the number of triple junctions as an integral of the order parameter that applies to a variety of vectorial Allen-Cahn systems. Future work will include extension to three dimensions, and more efficient numerical methods such as semi-implicit schemes for (\ref{eq:udiscretized}) with low per time step computational cost.

\section{Acknowledgements}
Yuchuan Yang and Selim Esedo\={g}lu were supported by NSF DMS-2410272.

\newpage

\bibliographystyle{plain}
\bibliography{12-references}

@article{peng_2022,
	author = {X. Peng and A. Bhattacharya and S. K. Naghibzadeh and D. Kinderlehrer and R. Suter and K. Dayal and G. S. Rohrer},
	journal = {Physical Review Materials},
	title = {{Comparison of simulated and measured grain volume changes during grain growth}},
	volume = {6},
	year = {2022}}

@article{rohrer_annual,
	author = {G. S. Rohrer and I. Chesser and A. R. Kruse and S. K. Naghibzadeh and Z. Xu and K. Dayal and E. A. Holm},
	journal = {Annual Review of Materials Research},
	number = {1},
	pages = {1-23},
	title = {{Grain boundary migration in polycrystals}},
	volume = {53},
	year = {2023}}

@inbook{herring,
	author = {C. Herring},
	chapter = {Surface tension as a motivation for sintering},
	editor = {W. Kingston},
	pages = {143-179},
	publisher = {McGraw Hill},
	title = {{The Physics of Powder Metallurgy}},
	year = {1951}}

@article{bronsardreitich,
	author = {Bronsard, L. and Reitich, F.},
	journal = {Archive for Rational Mechanics and Analysis},
	number = {4},
	pages = {355--379},
	title = {{On three-phase boundary motion and the singular limit of a vector-valued Ginzburg-Landau equation}},
	volume = {124},
	year = {1993}}

@article{epshteyn,
	author = {Epshteyn, Y. and Liu, C. and Mizuno, M.},
	journal = {SIAM Journal on Mathematical Analysis},
	number = {3},
	pages = {3072-3097},
	title = {{Motion of grain boundaries with dynamic lattice misorientations and with triple junctions drag}},
	volume = {53},
	year = {2021}}

@article{mantegazzanovagatortorelli,
	author = {Mantegazza, C. and Novaga, M. and Tortorelli, V.},
	journal = {Annali della Scuola normale superiore di Pisa, Classe di scienze},
	month = {03},
	title = {{Motion by Curvature of Planar Networks}},
	volume = {3},
	year = {2003}}

@article{GOTTSTEIN2002703,
title = {{Triple junction drag and grain growth in 2D polycrystals}},
journal = {Acta Materialia},
volume = {50},
number = {4},
pages = {703-713},
year = {2002},
issn = {1359-6454},
doi = {https://doi.org/10.1016/S1359-6454(01)00391-3},
url = {https://www.sciencedirect.com/science/article/pii/S1359645401003913},
author = {G. Gottstein and L.S. Shvindlerman}
}

@article{lauxsimon,
	author = {Laux, T. and Simon, T.M.},
	journal = {Communications on Pure and Applied Mathematics},
	number = {8},
	pages = {1597-1647},
	title = {{Convergence of the Allen‐Cahn Equation to Multiphase Mean Curvature Flow}},
	volume = {71},
	year = {2018}}

@article{pluda,
	author = {G{\"o}{\ss}wein, M. and Menzel, J. and Pluda, A.},
	doi = {10.4171/IFB/477},
	journal = {Interfaces and Free Boundaries},
	month = {06},
	title = {{Existence and uniqueness of the motion by curvature of regular networks}},
	volume = {25},
	year = {2022}}

@article{johnsonvoorhees,
title = {{A phase-field model for grain growth with trijunction drag}},
journal = {Acta Materialia},
volume = {67},
pages = {134-144},
year = {2014},
issn = {1359-6454},
doi = {https://doi.org/10.1016/j.actamat.2013.12.012},
url = {https://www.sciencedirect.com/science/article/pii/S1359645413009476},
author = {A.E. Johnson and P.W. Voorhees}
}

@article{ALIKAKOS_BETELU_CHEN_2006, 
title={{Explicit stationary solutions in multiple well dynamics and non-uniqueness of interfacial energy densities}}, 
volume={17}, 
DOI={10.1017/S095679250600667X}, 
number={5}, 
journal={European Journal of Applied Mathematics}, 
author={Alikakos, N.D. and Betel\'{u}, S.I. and Chen, X.}, 
year={2006}, 
pages={525–556}
}

@article{BALDO199067,
title = {{Minimal interface criterion for phase transitions in mixtures of Cahn-Hilliard fluids}},
journal = {Annales de l'Institut Henri Poincaré C, Analyse non linéaire},
volume = {7},
number = {2},
pages = {67-90},
year = {1990},
issn = {0294-1449},
author = {S. Baldo}
}

@inproceedings{degiorgi,
  author = {E. De Giorgi},
  title = {{New problems on minimizing movements}},
  booktitle = {Boundary value problems for PDE and
applications},
  year = {1993},
  pages = {81-98},
  address = {Masson}
}

@article{allencahn,
title = {{A microscopic theory for antiphase boundary motion and its application to antiphase domain coarsening}},
journal = {Acta Metallurgica},
volume = {27},
number = {6},
pages = {1085-1095},
year = {1979},
issn = {0001-6160},
author = {S.M. Allen and J.W. Cahn}
}

@misc{yangesedoglu25,
      title={Curvature Flow of Networks with Triple Junction Drag and Grain Rotation}, 
      author={Yang, Y. and  Esedoglu, S.},
      year={2025},
      eprint={2509.06125},
      note={Available at \url{https://arxiv.org/abs/2509.06125}}
}

@article{irregularnetworks,
author = {Lira, J. and Mazzeo, R. and Pluda, A. and Sáez, M.},
title = {Short-time existence for the network flow},
journal = {Communications on Pure and Applied Mathematics},
volume = {76},
number = {12},
pages = {3968-4021},
year = {2023}
}

@article{lqchen,
author = {L. Q. Chen},
title = {Phase-field models for microstructure evolution},
journal = {Annual Review of Materials Research},
volume = {32},
pages = {113-140},
year = {2002}
}

@article{osher_sethian,
	author = {S. Osher and J. Sethian},
	journal = {Journal of Computational Physics},
	pages = {12--49},
	title = {Fronts propagating with curvature-dependent speed: {A}lgorithms based on {H}amilton-{J}acobi formulation},
	volume = {79},
	year = {1988}
}

@inproceedings{mbo92,
	author = {B. Merriman and J. K. Bence and S. J. Osher},
	booktitle = {Proceedings of the Computational Crystal Growers Workshop},
	editor = {J. Taylor},
	pages = {73-83},
	publisher = {AMS},
	title = {Diffusion generated motion by mean curvature},
	year = {1992}
}

@article{mullins,
	author = {W. W. Mullins},
	journal = {J. Appl. Phys.},
	pages = {900-904},
	title = {Two dimensional motion of idealized grain boundaries},
	volume = {27},
	year = {1956}}

@article{barmak,
    author = {K. Barmak and E. Eggeling and D. Kinderleher and R. Sharp and S. Taasan and A. D. Rollett and K. R. Coffey},
    title ={Grain growth and the puzzle of its stagnation in thin films: {T}he curious tale of a tail and an ear},
    journal={Progress in Materials Science},
    year = {2013},
    volume = {58},
    pages = {987-1055}
}

@article{bernacki,
author = {M Barnacki and R. E. Loge and T. Coupez},
title = {Level set framework for the finite-element modeling of recrystallization and grain growth in polycrystalline materials},
journal ={Scripta Materialia},
volume = {64},
number = {6},
pages = {525-528},
year = {2011}
}

@article{elsey1,
author={M. Elsey and S. Esedo{\=g}lu and P. Smereka},
title={Diffusion generated motion for grain growth in two and three dimensions},
journal={Journal of Computational Physics},
volume = {228},
number = {21},
pages = {8015-8033},
year = {2011}
}

@article{esedoglu_otto,
	author = {S. Esedo{\=g}lu and F. Otto},
	institution = {Max Planck Institute for Mathematics in the Sciences},
	journal = {Communications on Pure and Applied Mathematics},
	number = {5},
	pages = {808-864},
	title = {Threshold dynamics for networks with arbitrary surface tensions},
	volume = {68},
	year = {2015}}

@article{elsey2,
	author = {M. Elsey and S. Esedo{\=g}lu and P. Smereka},
	journal = {Proceedings of the Royal Society A: Mathematical, Physical, and Engineering Sciences},
	pages = {381-401},
	title = {Large scale simulations of normal grain growth via diffusion generated motion},
	volume = {467:2126},
	year = {2011}}

@article{elsey3,
	author = {M. Elsey and S. Esedo{\=g}lu},
	institution = {UM},
	journal = {Mathematics of Computation},
	number = {312},
	pages = {1721-1756},
	title = {Threshold dynamics for anisotropic surface energies},
	volume = {87},
	year = {2018}}

@article{saye_sethian_2011,
	author = {R. I. Saye and J. A. Sethian},
	journal = {Proceedings of the National Academy of Sciences},
	pages = {19498-19503},
	title = {The Voronoi implicit interface method for computing multiphase physics},
	volume = {108},
	year = {2011}}

@article{mckenna,
author = {I. M. McKenna and S. O. Poulsen and E. M. Lauridsen and W. Ludwig and P. W. Voorhees},
title = {Grain growth in four dimensions: {A} comparison between simulation and experiment},
journal = {Acta Materialia},
volume = {78},
year = {2014},
pages = {125-134}
}

@article{lqchen2,
author = {L.-Q. Chen and W. Yang},
title = {Computer simulation of the domain dynamics of a quenched system with large number of nonconserved order parameters: {T}he grain-growth kinetics},
journal = {Physical Review B},
year = {1994},
volume = {50},
number = {21},
pages = {15752}
}

@article{fan,
author = {D. Fan and C. Geng and L.-Q. Chen},
title = {Computer simulation of topological evolution in {2-D} grain growth using a continuum diffuse-interface model},
journal = {Acta Materialia},
year = {1997},
volume = {45},
number = {3},
pages = {1115-1126}
}

@article{RSK,
 author = {J. Rubinstein and P. Sternberg and J. B. Keller},
 journal = {SIAM Journal on Applied Mathematics},
 number = {1},
 pages = {116--133},
 publisher = {Society for Industrial and Applied Mathematics},
 title = {Fast Reaction, Slow Diffusion, and Curve Shortening},
 volume = {49},
 year = {1989}
}

@book{alikakos2018elliptic,
  author    = {N.D. Alikakos and G. Fusco and P. Smyrnelis},
  title     = {Elliptic Systems of Phase Transition Type},
  series    = {Progress in Nonlinear Differential Equations and Their Applications},
  volume    = {91},
  publisher = {Birkh{\"a}user},
  year      = {2018}
}

@ARTICLE{ginzburglandau,
  title    = "The Jacobian and the {Ginzburg-Landau} energy",
  author   = "R.L. Jerrard and H.M. Soner",
  journal  = "Calculus of Variations and Partial Differential Equations",
  volume   =  14,
  number   =  2,
  pages    = "151--191",
  year     =  2002
}

@article{miyoshi2025,
title = {Phase-field framework for data-driven estimation of finite grain boundary junction mobilities},
journal = {Computational Materials Science},
volume = {259},
pages = {114161},
year = {2025},
issn = {0927-0256},
author = {E. Miyoshi and A. Yamanaka},
}

@article{miyoshi2026,
  title={Grain Growth Kinetics under Triple-Junction Drag: Phase-Field Simulations in Two and Three Dimensions},
  author={E. Miyoshi},
  journal={ISIJ International},
  volume={advpub},
  pages={ISIJINT-2025-361},
  year={2026}
 }

@article{steinbach,
title = {A generalized field method for multiphase transformations using interface fields},
journal = {Physica D: Nonlinear Phenomena},
volume = {134},
number = {4},
pages = {385-393},
year = {1999},
author = {I. Steinbach and F. Pezzolla},
}

@ARTICLE{modica_mortola,
  author = {L. Modica and S. Mortola},
  title = {Un esempio di Gamma-convergenza},
  journal = {Boll. Un. Mat. Ital. B (5)},
  year = {1977},
  volume = {14},
  pages = {285--299},
  number = {1}
}

\end{document}